\documentclass{article}[12pt]

\usepackage{arxiv}
\usepackage{quiver}

\usepackage[utf8]{inputenc} 
\usepackage[T1]{fontenc}    
\usepackage[colorlinks = true, linkcolor = purple, citecolor = purple, urlcolor = purple]{hyperref}       
\usepackage{url}            
\usepackage{booktabs}       
\usepackage{amsfonts}       
\usepackage{nicefrac}       
\usepackage{microtype}      
\usepackage{lipsum}		
\usepackage[normalem]{ulem}
\usepackage{stackengine}
\usepackage{color,graphicx} 

\usepackage{algorithm}%
\usepackage{algorithmicx}%
\usepackage{algpseudocode}%

\usepackage[font={up}]{caption}
\usepackage{subcaption} 
\usepackage{float}
\usepackage{colortbl}
\usepackage{soul} 
\usepackage{tabularx}
\usepackage{cite}
\usepackage{doi}
\usepackage{verbatimbox}
\usepackage[shortlabels]{enumitem}
\usepackage{mathtools}
\usepackage{amscd}
\usepackage{amsmath}
\usepackage{amssymb}
\usepackage{amsthm}
\usepackage{bm,nicematrix}

\newtheoremstyle{exampstyle}
  {3pt} 
  {0pt} 
  {} 
  {0pt} 
  {\bfseries} 
  {.} 
  {.5em} 
  {} 

\theoremstyle{exampstyle}
\newtheorem{theorem}{Theorem}[section]

\newtheorem{corollary}[theorem]{Corollary}

\newtheorem{definition}[theorem]{Definition}
\newtheorem{example}[theorem]{Example}

\newtheorem{lemma}[theorem]{Lemma}

\theoremstyle{remark}

\newtheorem*{remark}{Remark}

\newcommand{\N}{\mathbb{N}}
\newcommand{\R}{\mathbb{R}}
\newcommand{\Z}{\mathbb{Z}}

\newcommand{\B}{\mathcal{B}}

\newcommand{\D}{\mathcal{D}}
\newcommand{\E}{\mathcal{E}}
\newcommand{\F}{\mathcal{F}}
\renewcommand{\P}{\mathcal{P}}

\newcommand{\T}{\mathcal{T}}

\newcommand{\downup}[1][]{ \uparrow_{#1} \downarrow_{#1}}
\newcommand{\updown}[1][]{\downarrow_{#1} \uparrow_{#1}}
\DeclarePairedDelimiter{\ceil}{\lceil}{\rceil}
\newcommand{\lex}[1][]{\underset{#1}{\mathcal{L}}}
\newcommand{\End}{\text{End}}
\newcommand{\op}{\text{op}}

\newcommand{\myceil}[2][k]{\ceil[\Big]{\frac{#2 }{#1}}}

\newcommand{\tuple}[3][1]{#2_{#1}, \dots , #2_{#3}}

\title{A new family of distances over partially ordered sets}

\date{\today} 					

\author{ \href{https://orcid.org/0009-0005-9710-3966}{\includegraphics[scale=0.06]{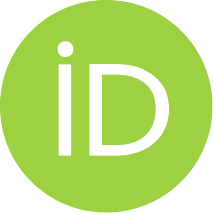}\hspace{1mm}Astrid A.~Olave} \thanks{https://aaolaveh.github.io/} \\
	Michigan State University\\
	\texttt{olaveher@msu.edu} \\
}

\renewcommand{\headeright}{}

\hypersetup{
pdftitle={A new family of distances over partially ordered sets},
pdfsubject={A new family of distances over partially ordered sets},
pdfauthor={Astrid A.~Olave},
pdfkeywords={Partially ordered set (poset), Poset metric, Metric characterization, Interleaving distance},
}

\begin{document}
\maketitle

\begin{abstract}

Order theory is increasingly relevant in applications where data is naturally structured as a partially ordered set (poset), often requiring meaningful notions of distance over posets. In this paper, we introduce a new family of extended metrics on path-connected and fence-connected posets that do not require additional structure. Unlike many existing distances, these metrics are not induced by valuations, but instead arise as a type of shortest-path distance determined by both path length and the number of alternations. For discrete posets, we show that these metrics converge to a type of shortest-fence metric. Our main result establishes that these metrics characterize most discrete path-connected posets up to isomorphism, and up to duality for modular posets. Finally, we prove that this family defines interleaving distances when posets are viewed as thin categories.
    
\end{abstract}

\keywords{Partially ordered set (poset) \and Poset metric \and Metric characterization  \and Interleaving distance}

\section{Introduction}

Partially ordered sets (posets) and the theory of order relations are gaining increasing relevance in applications across multiple fields, including statistics \cite{model_oriented_graphdistance, Arcagni2022}, management sciences \cite{distances_agent_preferences},  socioeconomics \cite{fattore2017partial, Arcagni2022}, environmental sciences, chemistry \cite{Bruggemann1999,partial_order_chemistry_environmental}, phylogenetics \cite{phylogenetic}, social sciences \cite{Day1981, Barthlemy1982,discrete_poset_distances}, information theory \cite{Simovici} and data science \cite{machinelearning}. The relevance of order theory in these areas stems from the fact that posets provide a natural structure for organizing data and representing relationships that are not necessarily linear. In many real-world settings, elements may be comparable only in certain respects, and partial orders offer a mathematical framework to model such situations. In particular, many problems require quantifying the dissimilarity between data points, that is, defining a notion of distance between elements of a poset. Several metrics defined on posets have been proposed for this purpose. For instance:
\begin{itemize}
    \item In statistics, graphs that encode structure among a set of variables define a poset. The distance between two graphs is given by the length of a shortest path between them on the poset of graphs \cite{model_oriented_graphdistance}.
    \item In environmental science and chemistry, chemicals can be partially ordered according to a set of attributes, in what is known as the partial ranking method \cite{Bruggemann1999}. For ranking and prediction tasks, a connectivity operator between two chemicals is defined to be the length of the shortest path between them. \cite{partial_order_chemistry_environmental}.
    \item A partition of a set $A$ is a collection of subsets of $A$ such that every element belongs to exactly one subset. Partitions play an important role in the social sciences in problems involving the classification of objects and finding consensus \cite{Day1981}. In this context, different metrics are used to compare elements in the lattice of partitions. In particular, some of these metrics are induced by real-valued functions called (lower) valuations \cite{monjardet}. 
    \item The degree of kinship between individuals is crucial in legal contexts such as inheritance rules, prohibited degrees of marriage, and legal adoption. This degree of kinship can be modeled as a distance function defined on upper semilattices \cite{discrete_poset_distances}.
    \item In information theory, a generalization of Shannon entropy can be formulated in the context of lattices. Entropies on lattices are defined as a type of lower valuation; as such, they induce metrics on the lattices, thereby opening the possibility of new axiomatizations of entropies starting from axiomatizations of metrics on lattices \cite{Simovici}.
\end{itemize}

While several metrics have been proposed in these contexts, many of them rely on additional structural assumptions on the underlying poset such as the existence of a maximal element, modularity, or a lattice structure. These assumptions restrict the class of posets to which such metrics can be applied when analyzing data. Consequently, it is desirable to develop more general metrics that do not rely on these structural conditions, as this allows the analysis of a broader range of ordered sets and highlights different structural properties of the underlying poset \cite{distances_agent_preferences}.


Motivated by the need for distance measures that apply to general posets, we introduce two new families of extended metrics defined over partially ordered sets, called the \textit{climber} and \textit{diver} metrics. These distances are inspired by the interleaving distance from topological data analysis (TDA) \cite{Flow}, and are constructed through a novel approach in which the distance is determined by counting the minimal number of iterations of certain self-maps on the poset of the power set. In contrast to many existing metrics, these distances are not induced by valuations and do not depend on additional structural assumptions beyond the order relation itself.

We then investigate the relationship between these metrics and structural properties of posets such as linearity, modularity, and lattice structure. In particular, we show that the $1$-climber and $1$-diver metrics together characterize most discrete path-connected posets up to isomorphism, and up to duality in the case of modular posets.

\textbf{Outline} 

The remainder of this article is organized as follows. In Section \ref{Sec:k-distances} we define the climber and diver distances and provide the necessary background to define them. In Section \ref{Sec:notable_posets}, we give explicit formulas for calculating these distances in notable classes of posets, such as linear or modular posets as well as in constructions like direct products. We also compare them with other metrics appearing in the literature. Section \ref{Sec:characterization}, is devoted to characterizing discrete path-connected posets using the 1-climber and 1-diver distances. Finally, in Section \ref{Sec:concluding}, we discuss the connection of these distances with TDA, summarize the main results, and outline possible directions for future research. Additionally, in Appendix \ref{Appendix}, we present a SageMath implementation for computing these metrics. 

\section{Climber and diver metrics}
\label{Sec:k-distances}

In this section, we introduce several definitions and notations used to define the $k$-climber and $k$-diver distances for path-connected posets and the fence-climber and fence-diver distances for fence-connected posets.

\subsection{Definitions and notations}

We give some definitions for poset terminology used in this paper; see \cite{Schroder} for more details in order theory. A \textit{poset} $P$ is a set with a reflexive, antisymmetric, and transitive relation $\leq_P$ or simply $\leq$ if $P$ is clear by context.
We say that $x$ and $y$ in $P$ are \textit{comparable} and write $x \sim y$ if and only if $x \leq y$ or $y \leq x$. For $x,y \in X$ we say that  $x$ is the \textit{lower cover} of  $y$ or $y$ is the \textit{upper cover} of $x$ if $x < y$ and $x \leq z \leq y$ implies that $z =x$ or $z =y$. We denote it $x \prec y$. Points that satisfy such covering relation are called \textit{adjacent}.

The \textit{(Hasse) diagram} of $P$ is the graph whose vertices are elements of $P$ and whose edges are pairs $(x,y)$ where $x$ and $y$ are adjacent. Note that the covering relations is the smallest relation that carries all the information for a given finite order. To visually incorporate the hierarchy if $x \prec y$, we draw $x$ lower than $y$ in the diagram. 

For $x,y \in P$, a \textit{path} $\gamma$ from $x$ to $y$ is a subset $\gamma = \{ \tuple{x}{n}\}$ such that $x_0 = x$, $x_n =y$ and for $ 0 \leq i \leq n$, $x_i$ and $x_{i+1}$ are adjacent. We say that $\gamma$ connects $x$ and $y$. Note that a path of $P$ is a path of the diagram of $P$ in the usual sense of a graph.  The length of $\gamma$, $\ell(\gamma)$, is equal to $n$. If $x_0 \prec x_1$, then we call $\gamma$ an \textit{upward} path, otherwise, $\gamma$ is a \textit{downward} path. We say that $S \subset P$ is a \textit{path-connected component} of $P$ if for all $x,y \in S$ there is a path from $x$ to $y$ and for any $z \notin S$, $S \cup {z}$  is not a path-connected component. We say $P$ is \textit{path-connected} if $P$ is itself a path-connected component. For $P$ a path-connected poset, the shortest-path function $\D$, defined by setting $\D(x,y)$ to be the minimal length of a path from $x$ to $y$ in $P$ is a well-defined metric on $P$.

A \textit{chain} $C$ is a subset of $P$ such that for all $x,y \in C$ we have that $x \sim y$. The \textit{height} of $P$, $h(P)$ is the length of the longest chain in $P$. If there are chains of arbitrary length, then $P$ has infinite height. A subset $S$  of $P$ is called an \textit{antichain} if for all $x,y$ in $S$, $x$ and $y$ are not comparable. We define the \textit{width} of $P$, $w(P)$, to be the size of the largest antichain in $P$ if such set exists, otherwise, $w(P)$ is infinity.

We say the path $\gamma =  \{ \tuple[0]{x}{n}\}$ has an \textit{alternation} at $x_i$ if $x_{i-1} \prec x_i \succ x_{i+1}$ or $x_{i-1} \succ x_i \prec x_{i+1}$. Let $A(\gamma)$ to be the number of alternations on $\gamma$ and take $\{ a_{j}\}_{j=1}^{A(\gamma)}$, the set of alternations of $\gamma$. Set $a_0 = x_0$ and $a_{A(\gamma) + 1} = x_n$ and for $j$ between $0$ and $A(\gamma)$ take the chain $C_j = \{ x_i \mid x_i \text{ 
is comparable to } a_j \text{ and } a_{j+1} \}$. Hence, $\gamma$ is the union of chains, $\gamma = \bigcup_{j=0}^{A(\gamma)} C_{j}$. Let us call $\{\tuple[0]{C}{A(\gamma)}\}$ the chain decomposition of $\gamma$. Given that $C_j$ is a path itself we obtain that 
\begin{equation*}
    \ell(\gamma) = \sum_{j=0}^{A(\gamma)} \ell(C_j) 
\end{equation*}

The subset $F_\gamma := \{\tuple[0]{a}{A(\gamma) + 1} \}$ is a special set called a \textit{fence}. A \textit{fence} $F$ from $x$ to $y$ in $P$ is a subset $\{ \tuple[0]{f}{n} \}$ such that $f_0 = x$, $f_n = y$ and for $i$ odd, $f_{i-1} < f_i > f_{i+1}$ or  $f_{i-1} > f_i < f_{i+1}$. The length of the fence is $n$ and the number of alternations, $A(F)$ is $n-1$. Similarly to paths, if $f_0 < f_1$ then $F$ is an \textit{upward} fence. Otherwise, $F$ is a \textit{downward} fence. We call $F$ a zig-zag path if all order relations are covering relations, i.e. for $i$ odd, $f_{i-1} \prec f_i \succ f_{i+1}$ or  $f_{i-1} \succ f_i \prec f_{i+1}$. We say that $S \subset P$ is a \textit{fence-connected component} of $P$ if for all $x,y \in S$ there is a fence from $x$ to $y$ and for any $z \notin S$, $S \cup \{z\}$  is not a fence-connected component. We say that $P$ is \textit{fence-connected} if $P$ is itself a fence-connected component. For fence-connected posets, the shortest-fence function $\F$, defined by setting $\F(x,y)$ to be the minimal length of a fence is a well-defined metric \cite{belding}.

Note that a path-connected poset is fence-connected as any path $\gamma$ that connects $x,y$ in $P$ defines the fence $F_\gamma$ from $x$ to $y$. However, a poset can be fence-connected without being path-connected. Consider the poset $P = \N \cup \{\omega\}$ with the usual order and $\omega > m$ for all $m \in \N$. The poset $P$ is fence-connected but no path connected, since for any $m \in N$ there is not a path from $m$ to $\omega$ on the diagram of $P$. Nevertheless, adding an extra element $a$ to the poset $P$ (Figure \ref{fig:path_connected_no_discrete}) with the relations $0 \prec a$ and  $\omega \prec a$ makes $P \cup \{a\}$ a path connected poset. 

\begin{figure}[h]
    \centering
        \begin{tikzpicture}[
        node/.style={
        circle, draw=black,
        inner sep=0pt, minimum size=3pt,
        fill=white},
        every edge quotes/.style = {auto, font=\footnotesize, sloped}
        ]
        
        \node (0) [node, label = left: $0$]  at (0,0){};
        \node (w) [node, label = right: $\omega$]  at (4,1){};
        \node (c) [node, label = above: $a$]  at (2,2){};
        
        \draw (0) -- (w) node [midway, fill=white,        sloped, below] {$\cong \mathbb{N} \cup \{\omega\}$}
              (c) -- (0)
              (c) -- (w);
        \end{tikzpicture}
    \caption{Diagram of a poset that is path connected but not discrete.}
    \label{fig:path_connected_no_discrete}
\end{figure}
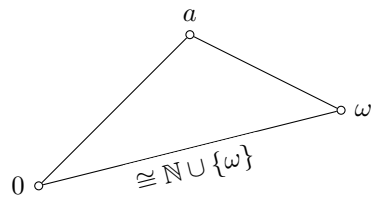
    
Observe that $\N \cup \{\omega\}$ failed to be path-connected because for any $m \in \N$ the subset $[m,\omega] = \{m ,m+1,m+2, ..., \omega\} $ is infinite. We will call a poset discrete when such a situation does not occur. Formally, a poset is called \textit{discrete} if for all $x \leq y$ in $P$, every maximal chain contained in $[x,y]$ is finite \cite{discrete_poset_distances}. We will refine the definition of the interval $[x,y]$ after the following lemma.

\begin{lemma}
If $P$ is a discrete, then, for any fence $F$ in $P$ from  $x$ to $y$ in $P$, there is a path $\gamma$ from  $x$ to $y$  such that $F_\gamma = F$.
\label{lemma:each F has gamma}
\end{lemma}
\begin{proof}
   Take $P$ a poset that is discrete and fence-connected. Consider $x,y \in P$. Since $P$ is fence-connected there is a fence $F = \{\tuple[0]{f}{n}\}$ such that $f_0 = x$ and $f_n=y$. Take $C_i$ to be a maximal chain contained in $[f_i,f_{i+1}]$ for $ 0\leq i <n$. Given that $P$ is discrete we know that $C_i$ is finite. Hence $\gamma = \bigcup_{i=0}^{n-1} C_i$ is a path that connects $x$ and $y$.
\end{proof}

\begin{corollary}
If $P$ is a discrete and fence-connected poset, then, $P$ is path-connected.
\label{coro: fence+discrete}
\end{corollary}

The converse of Corollary \ref{coro: fence+discrete} is not true. Figure \ref{fig:path_connected_no_discrete} shows a poset that is path-connected but not discrete.

Take $x \in P$. We define the \textit{ filter} or \textit{up-set} of $x$ to be $\uparrow x : = \{y \in P \mid y \geq x \}$ and the \textit{ideal} or \textit{down-set} of $x$ to be $\downarrow x : = \{y \in P \mid y \leq x \}$. Note that the difference between the up-set and the down-set of $x$ is that the one inequality in the definition is reversed.
In general, taking an order-theoretical statement and reversing all inequalities is called dualizing the statement .

For $x \leq y$ in P we define the interval $[x,y]$ as
\begin{equation*}
    [x,y] : = (\uparrow x) \cap (\downarrow y) = \{ z \in P \mid x \leq z \leq y \}
\end{equation*}

Let $Q$ be a poset.  A function $f: P \to Q$ is called an \textit{order-preserving map} or \textit{isotone} if $x \leq_P y$  implies that $f(x) \leq_Q f(y)$. $f$ is \textit{strictly isotone} if $x <_P y$  implies that $f(x) <_Q f(y)$. Dually, $f$ is an \textit{order-reversing map} or \textit{antitone} if $x \leq_P y$  implies that $f(x) \geq_Q f(y)$. Similarly, $f$ is \textit{strictly antitone} if $x <_P y$  implies that $f(x) >_Q f(y)$. We say that $f$ is an \textit{(order) isomorphism} if $f$ is a bijection, and $f$ and  $f^{-1}$ are order preserving maps. In that case we say that $P$ and $Q$ are \textit{isomorphic} posets and we note it $P \cong Q$.  In the case that $f$ and $f^{-1}$ are order-reversing maps we say that $f$ is and \textit{anti-isomorphism}.

\begin{example}
    For a poset $P$ consider $\P(P)$ the power set of $P$.  The set $\P(P)$ forms a poset under set inclusion. We define the order-preserving self map $\uparrow$ on $\P(P)$ given by the union of the up-sets of all elements in a subset. Explicitly, for $R \in \P(P)$,
    
    \[\uparrow R = \bigcup_{x \in R} (\uparrow x)\]. 

    The set $\uparrow R$ is called the up-set of $R$. Dually, the map $\downarrow$ is defined by taking the the union of the down-sets of the elements.  The set $\downarrow R$ is called the down-set of $R$.
\end{example}

Let $P$ and $Q$ be posets equipped with the metric $\mathcal{M}$. We say that $f: P \to Q$ \textit{preserves} $\mathcal{M}$ or that $\mathcal{M}$ is \textit{invariant under} $f$ if for all $x,y \in P$, $\mathcal{M}(x,y) = \mathcal{M}(f(x),f(y))$.

\begin{lemma}
 Let $P$ and $Q$ be posets equipped with the metric $\mathcal{M}$. If $f$ is a bijection that preserves $\mathcal{M}$ then $f^{-1}$ also preserves $\mathcal{M}$.  
 \label{lemma_f-1_preserves}
\end{lemma}
\begin{proof}
Take $x,y \in Q$ and $f: P \to Q$ a bijection that preserves $M$. Then,
    \begin{equation*}
        \mathcal{M}(x,y) = \mathcal{M}(f(f^{-1}(x)), f(f^{-1}(y))) = \mathcal{M}(f^{-1}(x), f^{-1}(y))
    \end{equation*} 
Therefore, $f^{-1}$ also preserves $\mathcal{M}$.
\end{proof}






\subsection{1-step metrics for path-connected posets}

For $R \in \P(P)$ let us define the set $\uparrow_1 R$, as $R$ with the union of upper covers of all elements in $R$. Explicitly \[\uparrow_1 R := R \cup \{ y \in P \mid \text{ there is } x \in P \text{ such that } y \succ x \}. \]

We define dually $\downarrow_1 R$, the set of $R$ with the lower covers of all elements in $R$. Note that $\uparrow_1$ and $\downarrow_1$ are order preserving self-maps on $\P(P)$. Let us define recursively the self-map on $\P(P)$, $(\updown[1])^{n}$ for $n \in \N$ by taking the set of upper covers and lower covers $n$ times. Explicitly, for $R \in \P(P)$,
\begin{itemize}
    \item $(\updown[1])^0 (R) = R$.
    \item $(\updown[1])^{n+1} (R) = \downarrow_1 \big( \uparrow_1 ( \; (\updown[1])^{n} (R) \;) \big)$.
\end{itemize}

\begin{theorem}
        Let $P$ be a poset. Define the function $\E_1: \P(P)^ 2 \to \R \cup \{ \infty\}$ given by 
            \begin{equation*}
                \E_1(R,S) = \min \{ n \mid R \subseteq (\updown[1])^{n} (S) \text{ and } S \subseteq (\updown[1])^{n} (R)\}.
            \end{equation*}
            and $\E_1(R,S) = \infty$ if such $n$ does not exist.

        Then, $\E_1$ is an extended metric on $\P(P)$.
    \label{th:E1_is_distance}
\end{theorem}

Intuitively, the operation goes up and down on the poset by 1-step each time and the distance adds up how may times such movement is done. Given that climber is an \textit{\textbf{E}scalador} in spanish, we call $\E_1$ the \textit{1-climber distance}. 

\begin{proof}
    Let us check each property that $\E_1$ has to satisfied to be a metric. By definition, $\E_1$ is a non negative and symmetric function. For $R,S \in \P(P)$, we obtain that $\E_1(R,S) = 0$ if and only if $R \subseteq S$ and $S \subseteq R$ which is equivalent to $R = S$. We only need to prove the triangle inequality.

    Let us prove that for $T \in \P(P)$ we have the triangle inequality,
        \begin{equation*}
            \E_1(R,S) \leq \E_1(R,T) + \E_1(T,S).
            \label{eq:triangle_in}
        \end{equation*}

    Assume that $\E_1(R,S) < \infty$. If either $\E_1(R,T)$ or $\E_1(T,S) $ is $\infty$ we are done. Then, let $\E_1(R,T) = n$ and $\E_1(S,T) = m$. We obtain that 
       \begin{equation}
        R \subseteq (\updown[1])^n(T) \subseteq (\updown[1])^{n+m}(S) \qquad \text{ and } \qquad S \subseteq (\updown[1])^m(T) \subseteq (\updown[1])^{m+n}(R).
        \label{eq:R_S_interleaved}
    \end{equation}

    Then, by definition of $\E_1$, $\E_1(R,S) \leq m+n$ and the triangle inequality holds. 
        
    Let $\E_1(R,S) = \infty$. If for $T \in \P(P)$ we have that $\E_1(R,T) = \infty$, then we are done. Else, if $\E_1(R,T) = n < \infty$ we must have that $\E_1(S,T) = \infty$, otherwise, taking $\E_1(S,T) = m$ we obtain the contradiction $\E_1(R,S) \leq m+n$ using eq. (\ref{eq:R_S_interleaved}). Hence, the triangle inequality holds as well.
\end{proof}

\begin{remark}
    The metric $\E_1$ induces a distance on the poset $P$. If $x,y \in P$, we can define  $\E_1(x,y) := \E_1(\{x\},\{y\})$.
\end{remark}

The metric $\E_1$ on $P$ can be used to calculate $\E_1(R,S)$ for $R,S \in \P(P)$,

\begin{lemma}
    Let $x \in P$ and $S \subseteq P $. Then
    \begin{equation*}
        \E_1(\{x\},S) = \sup_{y \in S} \{ \E_1(x,y) \}.
    \end{equation*}
\end{lemma}
\begin{proof}

    Notice that $\E_1(x,y) =  \infty$ if and only there is no $n$ such that $y \in (\updown[1])^n(\{x\})$. Hence, if for some $y \in S$, $\E_1(x,y) =  \infty$, then, there is no $n$ such that $S \subseteq (\updown[1])^n(\{x\})$. Therefore, $\E_1(x,S) = \infty$.

    Now, let $n = \sup\{ \E_1(x,y) \mid  y \in S \} < \infty$. Thus, for all $y \in S$, 
    \begin{equation}
        x \in (\updown[1])^{\E_1(x,y)}(\{y\}) \subseteq (\updown[1])^{n}(\{y\}), \qquad y \in (\updown[1])^{\E_1(x,y)}(\{x\}) \subseteq (\updown[1])^{n}(\{x\}).
        \label{eq:x_S}
    \end{equation}
    
    Which implies that $x \in (\updown[1])^{n}(S) $ and $S \subseteq (\updown)^{n}(\{x\})$. Moreover, by definition of supremum, for any $m< n$ there is $y \in S$ such that $y \notin  (\updown[1])^{m}(\{x\})$. Hence, $n$ is the minimum such that eq. (\ref{eq:x_S}) holds. Therefore $n = \E_1(x,S) $.
\end{proof}

\begin{theorem}
    Let $R,S \subseteq P $. Then
    \begin{equation*}
        \E_1(R,S) = \sup\{ \E_1(x,y) \mid  x \in R, y \in S \}.
    \end{equation*}
    \label{th:distance_subsets}
\end{theorem}
\begin{proof}
   Analogous to the previous proof, we can prove that 
        $\displaystyle \E_1(R,S) = \sup_{y \in S}\{ \E_1(R,\{y\})\}$ interchanging $\{ x \}$ by $R$ in the argument. Then, the result follows.
\end{proof}

Given that we only need to find $\E_1$ on $P$ to obtain $\E_1$ on $\P(P)$ let us show an explicit formula to calculate $\E_1(x,y)$ for $x,y \in P$.

\begin{lemma}
    Take $x,y \in P $ such that there is a path $\gamma = \{ \tuple[0]{x}{n}\}$ with one alternation at $x_e$ of the form $x_{e-1} \prec x_e \succ x_{e+1}$. Then, $\E_1(x,y) \leq \ell(\gamma)-1$
    \label{lemma:1-1_one}
\end{lemma}
\begin{proof}
    We have
    \begin{align*}
        \{ x, \tuple[1]{x}{j}\} \subseteq (\updown[1])^{j}(\{x\}), & \quad  \text{when } j < e  \\
         \{ x, \tuple[1]{x}{e}, ..., x_{j+1}\} \subseteq (\updown[1])^{j}(\{x\}), & \quad  \text{when } j \geq e 
    \end{align*}
    Hence $y \in (\updown[1])^{n-1}(\{x\})$ and by symmetry $x \in \E_1^{n-1}(y)$. Therefore, $\E_1(x,y) \leq n- 1 = \ell(\gamma)-1$.
\end{proof}

Lemma \ref{lemma:1-1_one} gives how many steps we need to "climb and descend a peak" from $x$ to $y$, thus, we can iterate this process as many times as needed to walk all peaks between $x$ and $y$.

\begin{definition}
    Take $x,y \in P$ and $\gamma$ a path from $x$ to $y$. Define $L_1(\gamma)$, the \textit{1-climber length} of $\gamma$ by 
    \begin{equation*}
        L_1(\gamma) =  \left\{ 
        \begin{array}{cc}
            \ell(\gamma) - \frac{A(\gamma)}{2} &  \text{ if } A(\gamma) \text{ even}\\
            \ell(\gamma) - \frac{A(\gamma)+1}{2} &  \text{ if } A(\gamma) \text{ odd}, \gamma \text{ upward}\\
            \ell(\gamma) - \frac{A(\gamma)-1}{2} &  \text{ if } A(\gamma) \text{ odd}, \gamma \text{ downward}.\\
        \end{array}
        \right.
    \end{equation*}
\end{definition}

\begin{theorem}
    Take $x,y \in P$ such that the set of paths from $x$ to $y$ is not empty. Then, \[\E_1(x,y) = \min \{ L_1(\gamma) \mid \gamma \text { is a  path from } x \text{ to } y\}. \]
    In words, the 1-climber distance between $x$ and $y$ is the minimal 1-climber length over all paths from $x$ to $y$.
    \label{th:1climber}
\end{theorem}
\begin{proof}
    Take $\gamma$ a path from $x$ to $y$ in $P$ with chain decomposition $\{ \tuple[0]{C}{A(\gamma)}\}$. Let us prove that $\E_1(x,y) \leq L_1(\gamma)$. For the following cases, we are calculating $N$ such that $ y \in (\updown[1])^ N (\{x\})$ using Lemma \ref{lemma:1-1_one}. If $C_1$ is an upward path and $A(\gamma)$ is odd. Then,  
    \begin{equation*}
         N= \sum_{i=0}^{\frac{A(\gamma)-1}{2}} ((\ell(C_{2i}) + \ell(C_{2i+1})) -1) = \sum_{i=0}^{A(\gamma)} \ell(C_i) - \frac{A(\gamma) +1}{2} = \ell(\gamma) - \frac{A(\gamma) +1}{2}.
    \end{equation*}
    In the same way, if $A(\gamma)$ is even we have 
    \begin{equation}
         N =\sum_{i=0}^{\frac{A(\gamma)-2}{2}} (\ell(C_{2i}) + \ell(C_{2i+1}) -1) + \ell(C_{A(\gamma)})  = \sum_{i=0}^{A(\gamma)} \ell(C_i) - \frac{A(\gamma)}{2} = \ell(\gamma) - \frac{A(\gamma)}{2}.
         \label{eq:even_alt_down}
    \end{equation}

    On the other hand, if $C_1$ is a downward path and  $A(\gamma)$ even we obtain the same expression of equation (\ref{eq:even_alt_down}) since 
    \[
    N =\sum_{i=1}^{\frac{A(\gamma)}{2}} (\ell(C_{2i-1}) + \ell(C_{2i}) -1) + \ell(C_{0}).
    \]

    Ultimately, if $A(\gamma)$ is odd, we obtain
    \begin{equation*}
         N =\sum_{i=1}^{\frac{A(\gamma)-1}{2}} (\ell(C_{2i-1}) + \ell(C_{2i}) - 1) + \ell(C_{0}) + \ell(C_{A(\gamma)})  = \sum_{i=0}^{A(\gamma)} \ell(C_i) - \frac{A(\gamma) -1}{2} = \ell(\gamma) - \frac{A(\gamma) -1}{2}.
    \end{equation*}

    So, for each case $N = L_1(\gamma)$. Thus, $ y \in (\updown[1])^{L_1(\gamma)} (\{x\})$
    Next, consider the reverse path $\tilde{\gamma} = \{ y, x_{n-1}, ..., x_1,x\}$. We have that $A(\gamma) =  A (\tilde{\gamma}) $ and notice that if $A(\gamma)$ is odd and $\gamma$ is upward (downward) , then, $\tilde{\gamma}$ is also upward (downward).  Hence, $L_1(\gamma) = L_1(\tilde{\gamma})$. Using this fact and the argument from above we obtain \mbox{$x \in (\updown[1])^{L_1(\tilde{\gamma})}(\{y\}) = (\updown[1])^{L_1(\gamma)}(\{y\})$.} In this way, $\E_1(x,y) \leq L_1(\gamma)$. In particular 
    \[\E_1(x,y) \leq \min \{ L_1(\gamma) \mid \gamma \text { is a path that connects } x \text{ and } y\}.\]

    Note that the argument above proves that $\E_1(x,y) < \infty$. Thus, assume that $\E_1(x,y) = n$. Then, by, definition, $y \in (\updown[1])^n(\{x\})$. This implies that we can create a path $\gamma = \{\tuple[0]{x}{\ell(\gamma)}\}$ from $x$ to $y$ iteratively selecting elements from  $(\updown[1])^i(\{x\})$ recursively. More explicitly, choose $x_1 \in \; \uparrow_1 \{x\}$ if possible, and choose $x_2 \in \; \downarrow_1( \{x_1\} ) \subset \; \downarrow_1( \; \uparrow_1(\{x\}) \; )$ if possible, such that $x \succ x_1 \prec x_2$. 
    It may happen that either $\uparrow_1 \{x\}$ is empty or $ \downarrow_1( \{x_1\} )$ is empty. In such a case, step $i=1$ contributes only one new element (either $x_1$ or $x_2$) yielding a segment of length 1 or 2. However, both sets cannot be empty, for that would contradict the assumption that the minimal 1-climber distance from $x$ to $y$ is $n$.
    
    Proceed inductively: at step $i$, choose the next one or two elements (whenever possible) from $(\updown[1])^i(\{x\})$. By construction, each step introduces at least one new element not contained in $(\updown[1])^{i-1}(\{x\})$, otherwise, the climber distance would be strictly less than $n$ contradicting minimality. The process continues until reaching  $(\updown[1])^{n}(\{x\})$, where $y$ appears. Hence, at each step, we add either one or two new elements to the path. Consequently, $ n \leq \ell(\gamma) \leq 2n$.
    
    We have that if $A(\gamma) = 0 $ it implies that we only add one element at each step. So, $\ell(\gamma) = n $. 
    Hence, we obtain that
    \begin{equation*}
        L_1(\gamma) = n  - \frac{0}{2} = n.
    \end{equation*}
    Else, if $\gamma$ is not a chain, we can check that if $\ell(\gamma) = n + j$ for $ 0 \leq j \leq n$ then $A(\gamma) \geq j+1$. For instance, for $x,y$ in $P$ with  $\E_1(x,y) = 2$, the process may create the path $\gamma: x \succ x_1 \prec y$. Then, $\ell(\gamma) = 2 $ and $A(\gamma) = 1$.
    Hence, 
    \begin{equation*}
        L_1(\gamma) \leq (n + j) - \frac{(j+1)-1}{2} = n - \frac{j}{2} \leq n. 
    \end{equation*}

    Then, we have proved that 
    \begin{equation*}
        \min \{ L_1(\gamma) \mid \gamma \text { is a path that connects } x \text{ and } y\} \leq L_1(\gamma) \leq n = \E_1(x,y).
    \end{equation*}

    Since $\E_1(x,y) \leq \min \{ L_1(\gamma) \mid \gamma \text { is a path that connects } x \text{ and } y\} \leq \E_1(x,y)$ the proof is complete. 
\end{proof}

\begin{corollary}
    Take $x,y \in P$. $\E_1(x,y) < \infty$ if and only if $x,y$ are in the same path-connected component. 
    \label{coro:finite_if_connected_1}
\end{corollary}
\begin{proof}
    If $x$ and $y$ are in the same path-connected component, there is a path in $P$ between them. The existence of such path implies that $\E_1(x,y) < \infty$. If $\E_1(x,y) < \infty$, then a path from $x$ and $y$ can be construct, so they are in the same path-connected component.
\end{proof}

\begin{corollary}
    If $P$ is a path-connected poset, then the $1$-climber distance is a metric on $P$ and $\P(P)$.
\end{corollary}
\begin{proof}
    If $P$ is a path-connected poset, by Corollary \ref{coro:finite_if_connected_1}, we obtain that for any $x,y \in P$, $\E_1(x,y) < \infty$. Similarly for any $S,R \in \P(P)$, we obtain that $\E_1(S,R) = \sup\{\E_1(x,y) \mid x \in R, y \in S\} < \infty$.
\end{proof}

 By duality, for $R \in P$ we can define the set $(\downup[1])^n(R)$ by iterating the $\downarrow_1$ and $\uparrow_1$ operations on $R$, $n$ times. Then, we can define the extended metric on $\P(P)$, $\B_1: \P(P)^ 2 \to \R \cup \{ \infty\}$ given by 
    \begin{equation*}
        \B_1(R,S) = \min \{ n \mid R \subseteq (\downup[1])^{n} (S) \text{ and } S \subseteq (\downup[1])^{n} (R)\}.
    \end{equation*}
    and $\B_1(R,S) = \infty$ if such $n$ does not exist.

We call we call $\B_1$ the $1$-diver distance given that diver is a \textit{\textbf{B}uceador} in spanish.

All the previous results hold for $\B_1$ dually. In particular, by this dual correspondence we obtain the following definition and theorem.

\begin{definition}
    Take $x,y \in P$ and $\gamma$ a path from $x$ to $y$. Define $L^*_1(\gamma)$, the \textit{1-diver length} of $\gamma$ by 
     \begin{equation*}
        L^*_1(\gamma) =  \left\{ 
        \begin{array}{cc}
            \ell(\gamma) - \frac{A(\gamma)}{2} &  \text{ if } A(\gamma) \text{ even}\\
            \ell(\gamma) - \frac{A(\gamma)+1}{2} &  \text{ if } A(\gamma) \text{ odd}, \gamma \text{ downward }\\
            \ell(\gamma) - \frac{A(\gamma)-1}{2} &  \text{ if } A(\gamma) \text{ odd}, \gamma \text{ upward}.\\
        \end{array}
        \right.
    \end{equation*}
\end{definition}
\begin{theorem}
    Take $x,y \in P$ such that the set of paths from $x$ to $y$ is not empty. Then, the 1-diver distance between $x$ and $y$ is the minimal 1-diver length over all paths from $x$ to $y$. This means \[\B_1(x,y) = \min \{ L^*_1(\gamma) \mid \gamma \text { is a  path from } x \text{ to } y\}. \]
    \label{th:1diver}
\end{theorem}
\begin{proof}
    Analogous to Theorem \ref{th:1climber}
\end{proof}

\begin{example}
     Consider the following poset represented in its Hasse diagram. To calculate the $1$-climber distance between $a$ and $f$ we want to find a path $\gamma$ with the minimum 1-climber length. This implies we want to choose the path with the minimum length but the maximum  number of alternations and if possible upwards. Some paths are
         \begin{align*}
          \gamma_1&:a,c,f. \quad \ell(\gamma_1) = 2, A(\gamma_1) =  1  &\gamma_2:a,b,c,f. \quad \ell(\gamma_2) = 3, A(\gamma_2) =  1 \\
          \gamma_3&:a,d,f. \quad \ell(\gamma_3) = 3, A(\gamma_3) =  1 &\gamma_4:a,b,c,e,f. \quad \ell(\gamma_4) = 4, A(\gamma_4) =  3
         \end{align*}
     
     \begin{minipage}{0.33\textwidth}
         \includegraphics[width = \textwidth]{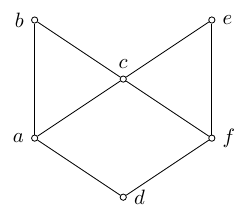}
     \end{minipage}
     \hfill
     \begin{minipage}{0.62\textwidth}
        
         We can check that $L_1(\gamma_1) = 1$ and $L_1(\gamma_2) = L_1(\gamma_3) = L_1(\gamma_4) = 2$. Then, \[\E_1(a,f) = 1 . \] Note that $\gamma_3$ has the same number of alternations but is downwards, which does not guarantee the minimum to find $\E_1(a,f)$, however, it is the path with optimal 1-diver length. Explicitly,
         \[
           \B_1(a,f) = L^*_1(\gamma_3) = \ell(\gamma_3) - \frac{A(\gamma_3)+1}{2} = 1.
         \]
          
     \end{minipage}
    \label{eg:poset_gato}
\end{example}






\subsection{Extension of 1-step metrics}

We can define more extended metrics taking compositions of the maps $\uparrow_1$ and $\downarrow_1$. Take $R \in \P(P)$ and let  $k,l \in \N \setminus \{0\}$. Define recursively $(\downarrow_l \uparrow_k)^n (R)$ by 
\begin{itemize}
    \item $(\downarrow_l \uparrow_k)^0(R) = R$.
    \item $(\downarrow_l \uparrow_k)^{n+1}(R) = (\downarrow_1)^l \big\{ (\uparrow_1)^k \{(\downarrow_l \uparrow_k)^{n}(R) \}\big\}$.
\end{itemize}

where $(\uparrow_1)^k$ denotes the $k$-fold composition of $\uparrow_1$ and $(\downarrow_1)^l$ denotes the $l$-fold composition of $\downarrow_1$.

Note that for $x,y \in P$, $y \in (\uparrow_1)^k(\{x\})$ if and only if there exists a chain of length at most $k$ from $x$ to $y$.

\begin{theorem}
        Let $P$ be a poset. Define the function $\E_{k,l}: \P(P)^ 2 \to \R \cup \{ \infty\}$ given by 
            \begin{equation*}
                \E_{k,l}(R,S) = \min \{ n \mid R \subseteq (\downarrow_l \uparrow_k)^{n} (S) \text{ and } S \subseteq (\downarrow_l \uparrow_k)^{n} (R) \}.
            \end{equation*}
            and $\E_{k,l}(R,S) = \infty$ if such $n$ does not exist.

        Then, $\E_{k,l}$ is an extended metric on $\P(P)$.
\end{theorem}
\begin{proof}
    The proof is analogous to Theorem \ref{th:E1_is_distance} replacing $\updown[1]$ by $\downarrow_l \uparrow_k$.
\end{proof}

As before, we can define the metric $\E_{k,l}$ on $P$ by $\E_{k,l}(x,y) := \E_{k,l}(\{x\},\{y\})$. Then we obtain the following theorem.
 
\begin{theorem}
    Let $R,S \subseteq P $. Then
    \begin{equation*}
        \E_{k,l}(R,S) = \sup\{ \E_{k,l}(x,y) \mid  x \in R, y \in S \}.
    \end{equation*}
\end{theorem}
\begin{proof}
    The proof is similar to  Theorem \ref{th:distance_subsets}.
\end{proof}

Without loss of generality, assume that $k \leq l$. Hence for any $x,y \in P$ we can check that 
\[ \E_{k,k}(x,y) \leq \E_{k,l}(x,y) \leq \E_{l,l}(x,y).\] Consequently for any $R,S \in \P(P)$ we have that 
\[ \E_{k,k}(R,S) \leq \E_{k,l}(R,S) \leq \E_{l,l}(R,S).\]

Hence, let us focus on $\E_{k,l}$ when $k = l$. In this case we reduce the notation from $\E_{k,k}$ to $\E_{k}$. 

\begin{definition}
    For $k = l$ we reduce the notation from $\E_{k,k}$ to $\E_{k}$.  We call $\E_{k}$ the \textit{k-climber distance}. 
\end{definition}

 By duality, for $R \in P$ we can define the set $(\uparrow_l\downarrow_k)^n(R)$ by iterating the $\uparrow_1$ and $\downarrow_1$ operations on $R$, $n$ times. Then, we can define the extended metric on $\P(P)$, $\B_{k,l}: \P(P)^ 2 \to \R \cup \{ \infty\}$ given by 
    \begin{equation*}
        \B_{k,l}(R,S) = \min \{ n \mid R \subseteq (\uparrow_l\downarrow_k)^{n} (S) \text{ and } S \subseteq (\uparrow_l\downarrow_k])^{n} (R) \}.
    \end{equation*}
    and $\B_{k,l}(R,S) = \infty$ if such $n$ does not exist.

We let $\B_{k}$ to be the \textit{k-diver distance}. We call both $\E_{k}$ and $\B_{k}$ a $k$-step distance. All the following results in this section hold for $\B_k$ by duality.

\begin{definition}
    Take $x,y \in P$ and $\gamma$ a path from $x$ to $y$. For $k \neq 0$ a natural number define $L_k(\gamma)$, the \textit{k-climber length} of $\gamma$ by 
    \begin{equation*}
        L_k(\gamma) = \left\{ 
        \begin{array}{cc}
            \sum\limits_{i=0}^{A(\gamma)} \myceil{\ell(C_i)}- \frac{A(\gamma)}{2} &  \text{ if } A(\gamma) \text{ even} \\
            \sum\limits_{i=0}^{A(\gamma)} \myceil{\ell(C_i)}- \frac{A(\gamma)+1}{2} &  \text{ if } A(\gamma) \text{ odd}, \gamma \text{ upward}\\
            \sum\limits_{i=0}^{A(\gamma)} \myceil{\ell(C_i)}- \frac{A(\gamma)-1}{2} &  \text{ if } A(\gamma) \text{ odd}, \gamma \text{ downward}.\\
        \end{array}
        \right.
    \end{equation*}
    where $\ceil{\cdot}$ is the ceiling function and $\{\tuple[0]{C}{A(\gamma)}\}$ is the chain decomposition of $\gamma$.
\end{definition}

\begin{lemma}
    Take $x,y \in P $ and $\gamma$ an upward path that connects $x,y$ with chain decomposition $\{C_0,C_1\}$. Then $\E_k(x,y) \leq L_k({\gamma}) = \myceil{\ell(C_0)} + \myceil{\ell(C_1)} -  1$. 
    \label{lemma_k-k_one}
\end{lemma}
\begin{proof}
    Assume that $\gamma$ has the alternation at $x_e$. Then
    \begin{align*}
        \{ x, \tuple[1]{x}{kj}\} \subseteq (\updown[k])^{j}(\{x\}), & \quad  \text{when } j < \myceil{\ell(C_0)}, \\
         \{ a, \tuple[1]{x}{e}, ..., x_{e + km}\} \subseteq (\updown[k])^{j}(\{x\}), & \quad  \text{when } j \geq \myceil{\ell(C_0)}  \text{ and } m = j - \myceil{\ell(C_0)} + 1.
    \end{align*}
    
    Hence $y \in (\updown[k])^{j}(a)$ where $j = \myceil{\ell(C_0)} + \myceil{\ell(C_1)} -  1$ given that $e + \ell(C_1)  \leq e + k\myceil{\ell(C_1)}$.
    
   By symmetry we obtain that $x \in (\updown[k])^{j}(\{y\})$ for the same $j$. 
\end{proof}

Similarly to the 1-climber distance, the lemma gives rise to the following theorem.

\begin{theorem}
   Consider $k \neq 0$ a natural number. Take $x,y \in P $ such that the set of paths from $x$ to $y$ is not empty. Then, the k-climber distance between $x$ and $y$ is the minimal k-climber length over all paths from $x$ to $y$. This means,\[\E_k(x,y) = \min \{ L_k(\gamma) \mid \gamma \text { is a path that connects } x \text{ and } y\}. \]
    \label{th:kclimber}
\end{theorem}
\begin{proof}
    The proof is analogous to Theorem \ref{th:1climber}. On one hand, to prove that for a path $\gamma$ from $x$ to $y$ in $P$, $\E_k(x,y) \leq L_k(\gamma) $ , we take the same summations in the proof and replace $\ell(C_i)$ by  $\myceil{\ell(C_i)}$. On the other hand, to show that $\E_k(x,y) \geq L_k(\gamma)$, for some $\gamma$, we create a path $\gamma$ from $x$ to $y$ iteratively selecting elements from $(\updown[k])^i(\{x\})$ recursively,
    by taking two chains, one upward and the other downward of at most length $k$ in $(\updown[k])^i(\{x\})$.
\end{proof}

As expected, the formulas for $L_1$ and $L_k$ coincide when $k=1$.

\begin{corollary}
    $\E_k(x,y) < \infty$ if and only if $x,y$ are in the same path-connected component. 
    \label{coro_finite_if_connected_1}
\end{corollary}

\begin{corollary}
    If $P$ is a path-connected poset, then the k-climber distance is a metric on $P$ and $\P(P)$.
\end{corollary}

\begin{definition}
    Consider $k \neq 0$ a natural number. Take $x,y \in P $ and $\gamma$ a path from $x$ to $y$. Define $L^*_k(\gamma)$, the \textit{k-diver length} by
    \begin{equation*}
        L^*_k(\gamma) = \left\{ 
        \begin{array}{cc}
            L_k(\gamma) &  \text{ if } A(\gamma) \text{ even}\\
            L_k(\gamma) - 1 &  \text{ if } A(\gamma) \text{ odd}, \gamma \text{ downward}\\
            L_k(\gamma) +1  &  \text{ if } A(\gamma) \text{ odd}, \gamma \text{ upward}.\\
        \end{array}
        \right.
    \end{equation*}
    
\end{definition}

\begin{theorem}
   Consider $k \neq 0$ a natural number. Take $x,y \in P $ such that the set of paths from $x$ to $y$ is not empty. Then, \[\B_k(x,y) = \min \{ L^*_k(\gamma) \mid \gamma \text { is a path that from } x \text{ to } y\} \]
   So, the k-diver distance between $x$ and $y$ is found by minimizing the  k-diver length over all paths connecting $x$ to $y$. 
    \label{th:kdiver}
\end{theorem}
\begin{proof}
    For $\gamma$ a path from $x$ to $y$ in $P$  let us define $L^*_k(\gamma)$ by interchanging the words upward and downward in the definition of $L_k(\gamma)$. Then, the proof is analogous to Theorem \ref{th:kclimber} based on the duality of $\E_k$ and $\B_k$. Notice that if $A(\gamma)$ is odd and $\gamma$ is a downward path, then, $L^*_k(\gamma) = L_k(\gamma) - 1$ and if $\gamma$ is upward then  $L^*_k(\gamma) = L_k(\gamma) + 1$. 
   
\end{proof}

Returning to Example \ref{eg:poset_gato} we obtain for $k=2$
\begin{align*}
    L_2(\gamma_1) &= \myceil[2]{1} +  \myceil[2]{1} - \frac{1+1}{2} =  1,  &L_2(\gamma_2) &= \myceil[2]{1} +  \myceil[2]{2} - \frac{1+1}{2} =  1 \\
    L_2(\gamma_3) &= \myceil[2]{1} +  \myceil[2]{1} - \frac{1-1}{2} =  2 &L_2(\gamma_4) &= \myceil[2]{1} +  \myceil[2]{1} + \myceil[2]{1} +  \myceil[2]{1} - \frac{3+1}{2} =  2
\end{align*}

In this way, $\E_2(a,f) = 1$. In contrast with $\E_1$, $\gamma_2$ also minimizes $L_2$. Similarly, $\gamma_3$ is the optimal path for the 2-diver distance.

Note that the values of $L_1$ and $L_2$ coincide on $\gamma_1$, $\gamma_3$ and $\gamma_4$ since the maximum chain length in their respective chain decompositions is $1$ . We formalize this observation in the following lemma.

\begin{lemma}
    Take $\gamma$ a path on $P$ with chain decomposition $\{ \tuple[0]{C}{A(\gamma)}\}$. Then $\left( L_k(\gamma) \right)_{k=1}^\infty$ is a non increasing sequence that stabilizes for $k \geq N$ with  
    \[
    N = \max \{\ell(C_i) \mid 0 \leq i \leq A(\gamma) \}
    \]
    Then, $\left( L_k(\gamma) \right)_{k=1}^\infty$ converges to $L_{\infty}(\gamma)$ given by 
    \begin{equation*}
        \arraycolsep=1.6pt\def\arraystretch{2}
        L_\infty(\gamma) = \left\{ 
        \begin{array}{lc}
            0 & \text{ if } \ell(\gamma) = 0 \\
           \frac{A(\gamma) + 2}{2} &  \text{ if } \ell(\gamma) > 0, \;A(\gamma) \text{ even}\\
            \frac{A(\gamma)+1}{2} &  \text{ if } A(\gamma) \text{ odd}, \gamma \text{ upward}\\
            \frac{A(\gamma)+3}{2}&  \text{ if } A(\gamma) \text{ odd}, \gamma \text{ downward}.\\
        \end{array}
        \right.
    \end{equation*}
    \label{lemma: L_convergent}
\end{lemma} 
\begin{proof}
    Take $\gamma$ a path on $P$ with chain decomposition $\{ \tuple[0]{C}{A(\gamma)}\}$. Note that for any chain $C_i$ and for any $k$ we have that \[\myceil{\ell(C_i)} \geq \myceil[k+1]{\ell(C_i)}.\] 
    
    Hence, $L_{k}(\gamma) \geq L_{k+1}(\gamma)$. Moreover, if  $ k \geq N$, where $N$ is the maximum length of the chains we obtain that 

    \[1 =  \myceil[N]{\ell(C_i)} \geq \myceil{\ell(C_i} \geq 1. \] 
        
    Therefore, $L_k(\gamma) =  L_N(\gamma)$ and  $\left( L_k(\gamma) \right)_{k=1}^\infty$ converges to $L_{N}(\gamma)$.

    Note that if $\ell(\gamma) =0$, $\displaystyle \sum\limits_{i=0}^{A(\gamma)} \myceil[N]{\ell(C_i)} = 0$. Meanwhile, if $\ell(\gamma) > 0$,  $\displaystyle \sum\limits_{i=0}^{A(\gamma)} \myceil[N]{\ell(C_i)} =  \sum\limits_{i=0}^{A(\gamma)} 1  = A(\gamma)+1 $ 
    
    Thus, 
    \begin{equation*}
            \arraycolsep=1.4pt\def\arraystretch{2}
            L_N(\gamma) = \left\{ 
            \begin{array}{lc}   
                0 - \frac{0}{2} = 0 & \text{ if } \ell(\gamma)=0 \\ 
                A(\gamma) + 1- \frac{A(\gamma)}{2} = \frac{A(\gamma) + 2}{2} &  \text{ if } \ell(\gamma)=0, \; A(\gamma) \text{ even}\\
                A(\gamma) + 1- \frac{A(\gamma)+1}{2} = \frac{A(\gamma)+1}{2} &  \text{ if } A(\gamma) \text{ odd}, \gamma \text{ upward}\\
                A(\gamma) + 1- \frac{A(\gamma)-1}{2} = \frac{A(\gamma)+3}{2}&  \text{ if } A(\gamma) \text{ odd}, \gamma \text{ downward}.\\
            \end{array}
            \right.
        \end{equation*}
\end{proof}

\begin{corollary}
    If $P$ is a poset such that $h(P) < \infty$, then for all paths $\gamma$ in $P$,  $\left( L_k(\gamma) \right)_{k=1}^\infty$ stabilizes for $k \geq h(P)$.
\end{corollary}
\begin{proof}
    By definition $h(P)$ is the length of the longest chain of $P$.
\end{proof}

\begin{definition}
    Take $P$ a path-connected poset and $x,y \in P$. The $\infty$-climber distance between $x$ and $y$ is the minimum length $L_\infty$ over all paths from $x$ to $y$. So,
    \[\E_\infty(x,y) = \min \{ L_{\infty}(\gamma) \mid \gamma \text { is a path that connects } x \text{ and } y\}. \]
    
\end{definition}

The next theorem proves that $\E_\infty(x,y)$ is the limit of the sequence $\left( \E_k(x,y) \right)_{k=1}^\infty$. Explicitly,
    \begin{equation*}
        \E_\infty(x,y) = \lim_{k \to \infty } \E_k(x,y).
    \end{equation*}

\begin{theorem}
    Take $P$ a path-connected poset. For any $x,y$ in $P$, $\left( \E_k(x,y) \right)_{k=1}^\infty$ is a convergent non increasing sequence. 
    \label{thm:E_convergent}
\end{theorem}
\begin{proof}
        Take $x,y$ in $P$. For $k \geq 1$  set $\gamma_k$ to be the path from $x$ to $y$ that minimizes the $k$-climber length $L_k$. Then, 
        \begin{equation*}
            \E_{k+1}(x,y) \leq L_{k+1}(\gamma_k) \leq L_{k}(\gamma_k) =  \E_k(x,y)
        \end{equation*}

    Let us show there is $N \geq 1$ such that $\gamma_N$ minimizes $L_k$ for all $k \geq N$.
    If $\gamma_1$ satisfies such property we are done. Otherwise there is $k_1 > 1$ such that
    \[
    L_{k_1}(\gamma_{k_1}) < L_{k_1}(\gamma_1) \leq L_1(\gamma_1)
    \]
    If $\gamma_{k_1}$ satisfies the property, then we are done. Otherwise, we can find $k_2 > k_1$ such that $ L_{k_2}(\gamma_{k_2}) <  L_{k_1}(\gamma_{k_1})$. Iterating this process we have a decreasing sequence
    \[
    L_1(\gamma_1) > L_{k_1}(\gamma_{k_1}) > \cdots > L_{k_n}(\gamma_{k_n}) > \cdots
    \]
    Since the values of the i-climber length are positive integers, the sequence must stabilize in at most $L_1(\gamma_1)$ steps. Hence, there is $N$ that satisfies the desired property. In this way,

    \[
     \lim\limits_{k \rightarrow\infty} \E_k(x,y) = \lim\limits_{k \rightarrow\infty} L_k(\gamma_N) = L_\infty(\gamma_N) = \E_\infty(x,y)
    \]
    \end{proof}

\subsection{Metrics for fence-connected posets}

In the previous section, we have established that for any path $\gamma$ the sequence of $k$-climber lengths $L_k(\gamma)$ converges to $L_\infty(\gamma)$, a function depending only on the alternations of the path. Consequently, $L_\infty(\gamma)$ can be regarded as a function of $F_\gamma$, the fence induced by $\gamma$. This observation naturally leads us to define a climber (diver) metric $\E$ ($\B$) for fence-connected posets which turns out to be the limit of the climber (diver) metrics $\E_k$ ($\B_k$) when $P$ is discrete.

Let us define recursively the order preserving self-map on $\P(P)$, $(\updown)^{n}$ for $n \in \N$ by taking the up-set and the down-set of a subset, $n$ times. Explicitly, for $R \in \P(P)$,
\begin{itemize}
    \item $(\updown)^0 (R) = R$
    \item $(\updown)^{n+1} (R) = \downarrow \big\{ \uparrow \{(\updown)^{n} (R) \} \big\}$
\end{itemize}

\begin{theorem}
        Let $P$ be a poset. Define the function $\E: \P(P)^ 2 \to \R \cup \{ \infty\}$ given by 
            \begin{equation*}
                \E(R,S) = \min \{ n \mid R \subseteq (\updown)^{n} (S) \text{ and } S \subseteq (\updown)^{n} (R) \}
            \end{equation*}
            and $\E(R,S) = \infty$ if such $n$ does not exist. Then, $\E$ is an extended metric on $\P(P)$.
\end{theorem}
\begin{proof}
   Using the same proof of Theorem $\ref{th:E1_is_distance}$ by considering the operation $\updown$ instead of $\updown[1]$.
\end{proof}

\begin{definition}
   The extended metric $\E$ is called the fence-climber distance. 
\end{definition}

Like before, $\E$ induces a distance on the poset $P$. If $x,y \in P$, we define  $\E(x,y) := \E(\{x\},\{y\})$. Then, the following theorem holds.

\begin{theorem}
    Let $R,S \subseteq P $. Then
    \begin{equation*}
        \E(R,S) = \sup\{ \E(x,y) \mid  x \in R, y \in S \}.
    \end{equation*}
\end{theorem}
\begin{proof}
   Analogous to the proof of Theorem \ref{th:distance_subsets}
\end{proof}

\begin{definition}
    Take $x,y \in P$ and consider $F$ to be a fence from $x$ to $y$. Define $L(F)$ the fence-climber length of $F$ by
    \begin{equation*}
        L(F) = \left\{ 
        \def\arraystretch{1.7}
        \begin{array}{cl}
            0 & \text{ if } F = \{x\}\\
            \frac{A(F) +2}{2} & \text{ if } x \neq y,  A(F) \text{ even} \\
            \frac{A(F) +1}{2} & \text{ if } A(F) \text{ odd}, F \text{ upward}\\
            \frac{A(F) +3}{2} & \text{ if } A(F) \text{ odd}, F \text{ downward}
        \end{array}
        \right.
    \end{equation*}
\end{definition}

\begin{theorem}
   Take $x,y \in P$ and consider the set of fences from $x$ to $y$ to be non-empty. Then, the fence-climber distance between $x$ and $y$ is the minimal fence-climber length over all fences from $x$ to $y$. This means, 
        
    \[\E(x,y)  = \min \{ L(F) \mid F \text { is a fence connecting } x \text{ and } y\} \]. 
    \label{th_E_is_metric}
\end{theorem}   
\begin{proof}
     Take $F = \{\tuple[0]{f}{n+1}\}$ a fence that connects $x$ and $y$. Let us prove that $\E(x,y) \leq L(F)$. Suppose that $F$ is an upward fence. Then, for any $k$  we have  $\{ x,f_1,...,f_{2k} \} \subseteq (\updown)^k(x) $. Hence, if $A(F)$ is odd, we obtain 
    \begin{equation*}
         y \in \{ x, f_1, ..., f_{A(F) + 1} \} \subseteq (\updown)^{\frac{A(F) +1}{2}}(\{x\}). 
    \end{equation*}
       
    And if $A(F)$ is even,
    \begin{equation}
         y \in \{ x, f_1, ..., f_{A(F) + 1} \} \subseteq (\updown)^{\frac{A(F) + 2}{2}}(\{x\}).
         \label{eq:even_alt}
    \end{equation}
       
    On the other hand, if $F$ is downward, for any $k$ we have that $\{ x,f_1, ... , f_{2k-1} \} \subseteq (\updown)^k(x)$. Hence, if  $A(F)$ is even we obtain (\ref{eq:even_alt}) again. If $A(F)$ is odd we get 
    \begin{equation*}
          y \in \{ x, f_1, ..., f_{A(F) + 1} \} \subseteq (\updown)^{\frac{A(F) +3}{2}}(\{x\}).
    \end{equation*}

    Hence, we have proved that $y \in (\updown)^{L(F)}(\{x\})$. Now, consider the fence $\tilde{F} = \{y , f_{n}, \cdots, f_1 , x \}$. We have that $A(F) =  A (\tilde{F}) $ and notice that if $A(F)$ is odd and $F$ is upward (downward) , then, $\tilde{F}$ is also upward (downward).  Hence, $L(F) = L(\tilde{F})$. Using this fact and the proof from above we obtain $x \in (\updown)^{L(F)}(\{y\})$.  
    
    In this way, $\E(x,y) \leq L(F)$. In particular $\E(x,y) \leq \min \{ L(F) \mid F \text { is a fence connecting } x \text{ and } y\}$

   Thus, we have proved that $\E(x,y) < \infty $. Then, assume that $\E(x,y) = n $. This implies that $y \in (\updown)^n(\{x\})$. Let us create iteratively a fence from $x$ to $y$ selecting elements from $(\updown)^i(\{x\})$. If $\uparrow x \neq \emptyset$, take $f_1 \in \uparrow x$ and $f_2 \in (\updown)(\{x\})$ such that $x < f_1 > f_2$. Note that such $f_2$ must exist, otherwise $\uparrow ((\updown)(\{x\})) = \uparrow \{x\}$ and contradicts the minimality of $n$. Inductively, we choose $f_{2i-1}$ and $f_{2i}$ from $(\updown)^i(\{x\})$. Hence Hence we can find a fence of one of the following two forms:
   
    \[ F_1: x < f_1 > f_2 < \cdots > f_{2n-2} < f_{2n-1} > y, \quad \quad F_2:  x < f_1 > f_2 < \cdots > f_{2n-2} <  y. \]

    In a similar fashion, if $\uparrow x = \emptyset$, then we can find a fence of one of the following two forms:

    \[ F_3: x > f_1 <  \cdots > f_{2n-3} < f_{2n-2} > y \quad \text{or} \quad   F_4: x > f_1 <  \cdots > f_{2n-3} < y. \]

    We can check that $L(F_i) = n $ for all $i$ between 1 and 4.  Hence,
    \[
    \E(x,y)  =  L(F) \geq \min \{L(F) \mid F \text { is a fence connecting } x \text{ and } y\}. 
    \]
\end{proof}

\begin{corollary}
    $\E(x,y) < \infty$ if and only if $x,y$ are in the same fence-connected component.
    \label{coro_finite_if_connected}
\end{corollary}

\begin{corollary}
    If $P$ is a fence-connected poset, then the fence-climber distance is a metric on $P$ and $\P(P)$.
    \label{coro:fence_metric}
\end{corollary}

 By duality, for $R \in P$ we can define the map $(\downup)^n$  and create the fence-diver distance $\B$. All the previous results hold for $\B$ dually. In particular, the following theorem holds.
 
\begin{theorem}
   Take $x,y \in P$ and consider $F$ to be a fence from $x$ to $y$. Define the fence-diver length, $L^*(F)$, by 
        \begin{equation*}
        L^*(F) = \left\{ 
        \def\arraystretch{1.7}
        \begin{array}{cl}
            0 & x=y \\
            \frac{A(F) +2}{2} & a \neq b,  A(F) \text{ even} \\
            \frac{A(F) +1}{2} & A(F) \text{ odd}, F \text{ downward}\\
            \frac{A(F) +3}{2} & A(F) \text{ odd}, F \text{ upward}.
        \end{array}
        \right.
    \end{equation*}
    Then $\B(x,y)  = \min \{ L^*(F) \mid F \text { is a fence connecting } x \text{ and } y\} $. 
    \label{th_B_is_metric}
\end{theorem}

Let us consider the same poset from Example \ref{eg:poset_gato}. To calculate the fence-climber distance between $a$ and $f$ we want to 
     
     \begin{minipage}{0.33\textwidth}
         \includegraphics[width = \textwidth]{Images/poset_gato_tikz.pdf}
     \end{minipage}
     \hfill
     \begin{minipage}{0.65\textwidth}
        find a fence $F$ with the minimum $L(F)$, i.e, the fence $F$ with the the minimum number of alternations and if possible  upwards. Let us take the fences from $a$ to $f$, $F_i = F_{\gamma_i}$ 
        \begin{align*}
          F_1&:a,c,f. \quad A(F_1) =  1  &F_2:a,b,f. \quad  A(F_2) =  1 \\
          F_3&:a,d,f. \quad A(F_3) =  1 &F_4:a,b,c,e,f. \quad A(F_4) =  3
         \end{align*}
        
         We can check that $L(F_1) = L_1(F_2) = 1$ and $ = L_1(F_3) = L_1(F_4) = 2$. Then, \[\E(a,f) = 1. \] Note that $F_3$ has the same number of alternations but is downwards, which does not guarantee the minimum to find $\E(a,f)$, however, it is the optimal path for the 1-diver distance. Explicitly,
         \[
           \B(a,f) = L^*(F_3) = \frac{A(F_3)+1}{2} = 1
         \]
     \end{minipage}

Recall that $F_1 = F_{\gamma_1}$. Since the the longest chain in  $\gamma_1$ has length 1, for every $k$ we obtain $L(F_1) = L_k(\gamma_1)$. More generally, for any path $\gamma$ in $P$,  the sequence $\left( L_k(\gamma) \right)_{k=1}^\infty$ converges to $L(F_\gamma)$ since  $L_\infty(\gamma) = L(F_\gamma)$. This follows directly from the definitions of $L_\infty$ and $L$. Then, we have the following theorem.

\begin{theorem}
    Let $P$ be fence-connected poset that is discrete. Then for any $x,y \in P$, $\E_{\infty}(x,y) = \E(x,y)$.
    \label{th:E_k_converges_E}
\end{theorem}
\begin{proof}
    Let $P$ be a discrete fence-connected poset. By corollary \ref{coro: fence+discrete}, it follows that $P$ is a path-connected poset. Hence, by Theorem \ref{thm:E_convergent}, for any $x,y \in P$, $\E_{\infty}(x,y) $ exists and is equal to $L_\infty(\gamma)$ where $\gamma$ is a path from $x$ to $y$ that minimizes $L_k$ for all $k \geq N$ for some $N$. Thus, one one hand we have that 
    \begin{equation*}
        \E_{\infty}(x,y) = L_\infty(\gamma) = L(F_\gamma) \geq \E(x,y).
    \end{equation*}

    On the other hand, let $F$ be a fence from $x$ to $y$ that minimizes $L(F)$. By lemma \ref{lemma:each F has gamma}, choose $\gamma$ from $x$ to $y$ such that $F_{\gamma} = F$. Then,
    \begin{equation*}
        \E(x,y) =  L(F_\gamma) = L_\infty(\gamma) \geq  \E_{\infty}(x,y).
    \end{equation*}

    Therefore, $ \E_{\infty}(x,y) =\E(x,y)$.
\end{proof}

The result of Theorem \ref{th:E_k_converges_E} does not hold if the poset is not discrete.  Let us consider the poset whose diagram is shown  

    \begin{minipage}{0.57\textwidth}
        at right. We observe that $P$  is fence-connected poset but is not discrete.  The elements $\omega_1$ and $\omega_2$ are connected by $\gamma$, an upward path with three alternations and by the fence $F = \{ \omega_1, 0 , \omega_2 \}$ where the intervals $[0,\omega_1]$ and $[0,\omega_2]$ are isomorphic to $\N \cup \omega$. We see that for all $k$, \mbox{$\E_k(\omega_1,\omega_2) = L_k(\gamma) = 2$}. Hence 
        \[
             \E_{\infty}(x,y) = 2.
        \]
             
        On the other hand, $x,y$ are connected by two fences $F$ and $F_\gamma$. We see that $L(F) = 1 < 2 = L(F_\gamma)$. Then, we obtain that $\E(x,y) = 1$. Therefore,  $\E_{\infty}(x,y) $ and $\E(x,y) $ are not equal.\\
         \end{minipage} 
         \hfill
    \begin{minipage}{0.38\textwidth}
        \includegraphics[width = \textwidth]{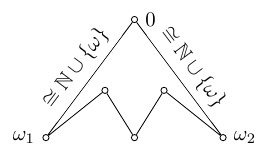}
    \end{minipage}

\begin{corollary}
    If $P$ is a path-connected poset, for any $x,y \in P$, $\E_{\infty}(x,y) \geq \E(x,y)$.
\end{corollary}

We conclude this section by exploring how the fence climber and diver metrics relate to the width of the poset.

\begin{lemma}
 Take $P$ to be fence-connected. Let $D_{\E} = \sup \{ \E(x,y) \mid x,y \in P \}$ be the diameter of $P$ with the fence-climber distance. Then \hbox{$D_{\E} \leq w(P)$}.
 \label{lemma_infty_diam}
\end{lemma}
\begin{proof}
    For $x,y \in P$ arbitrary, take a fence $F = \{\tuple[0]{f}{n}\}$ from $x$ to $y$ such that $L(F) = \E(x,y)$. If $A(F) = 0$, then $x \sim y$. Hence $P$ is a chain and $w(P) = 1$. We obtain $\E(x,y) = 1 = w(P)$.  
    
    Take $A(F) > 0$. The end goal is to prove that $\{f_{2i}\}$ is an antichain of $P$. To do so, let us show there is no $i$ and $j > i$ such that $f_{2i} \sim f_{2j}$. Let us proceed by contradiction. We begin with the case that $F$ is an upward fence. Thus $f_{2k} < f_{2k+1} > f_{2k+2}$ for any $k$.
    \begin{enumerate}[a)]
        \item  Suppose that $f_{2i} < f_{2j}$. Then, $F' = \{ x, ..., f_{2i}, f_{2j+1}, ... ,y\}$ is a upward fence from $x$ to $y$ with \mbox{$A(F') \leq A(F) - 2$}. Then $L(F') < L(F)$, contradicting the minimality of $F$.
        \item  Suppose that $f_{2i} > f_{2j}$ with $i > 0$. Then we can form a new fence $F' = \{ x, f_1, ..., f_{2i-1}, f_{2j}, ... ,y\}$ from $x$ to $y$. Then, $F'$ has at least two alternations less than $F$. Hence, $L(F') < L(F)$ contradicting the minimality of $F$.
        \item  Suppose that $f_{2i} > f_{2j}$ with $i =0$ and $j>1$ . Then $f_{2j} < f_{2i} = x$. Hence we can form the downward fence $F' = \{ x, f_{2j}, f_{2j+1}, ..., y \}$ with $A(F') \leq A(F) - 3$. Then
         \begin{align*}
            L(F') \leq \dfrac{A(F') + 3}{2} \leq \dfrac{A(F)- 3 + 3}{2} < \dfrac{A(F)+2}{2} = L(F) &  \text{ if } A(F) \text{ is even}. \\
            L(F') \leq \dfrac{A(F') + 2}{2} \leq \dfrac{A(F) -3 + 2}{2} < \dfrac{A(F) +1}{2} =  L(F) &  \text{ if } A(F) \text{ is odd}.
        \end{align*}
        That yields the contradiction $L(F') < L(F)$.
        
        \item Suppose that $f_{2i} > f_{2j}$, $i =0$ and $j=1$. Thus, $x > f_2$. In this case, we obtain the downward fence $F'= \{ x , f_2, ..., y\}$ with one less alternation. Then
        \begin{align*}
            L(F') = \dfrac{A(F') + 3}{2} =  \dfrac{A(F) + 2}{2} = L(F) &  \text{ if } A(F) \text{ is even}. \\
            L(F') = \dfrac{A(F') + 2}{2} =  \dfrac{A(F) + 1}{2} = L(F) &  \text{ if } A(F) \text{ is odd}.
        \end{align*}
     In this case, we switch $F$ by $F'$ as the minimal path we are choosing from the start, so the minimal fence is a downward fence.
    \end{enumerate}
    
   Consider the case that $F$, the minimal path, is a downward fence. Then, $f_{2k} > f_{2k+1} < f_{2k+2}$ for any $k$ in this case. We obtain that,
   \begin{enumerate}[a)]
        \item If $f_{2i} > f_{2j}$. Then, $F' = \{ x, ..., f_{2i}, f_{2j+1}, ... ,y\}$ is a downward fence from $x$ to $y$ with $A(F') \leq A(F) - 2$. Hence, we have the contradiction $L(F') < L(F)$.
        \item If $f_{2i} < f_{2j}$ and $i > 0$, we have the downward fence $F' = \{ x, f_1, ..., f_{2i-1}, f_{2j}, ... ,y\}$ from $x$ to $y$. Similar to before, we have removed two alternations, so $L(F') < L(F)$.
        \item If $f_{2i} < f_{2j}$ and $i = 0$ we have $x < f_{2j}$. Hence, we obtain the upward fence $F' = \{ x, f_{2j}, f_{2j+1}, ..., y \}$ with at least one less alternation. Then, 
        \begin{align*}
            L(F') \leq \dfrac{A(F') + 1}{2} \leq  \dfrac{A(F) -1 + 1}{2} < \dfrac{A(F)+ 2}{2} = L(F) &  \text{ if } A(F) \text{ is even}. \\
            L(F') = \dfrac{A(F') + 2}{2} \leq  \dfrac{A(F) -1  + 2}{2} < \dfrac{A(F)+ 3}{2} = L(F) &  \text{ if } A(F) \text{ is odd}. 
        \end{align*}
        
        Which contradicts the minimality of $F$.
        
    \end{enumerate}

    Since assuming that there are $i < j$ such that $f_{2i} \sim f_{2j}$, leads to a contradiction, we can conclude that $\{f_{2i}\}$ is an antichain of $P$. 
    Hence, if $A(F)$ is odd, $\{x , x_2, ..., x_{A(F)-1}, y\}$ is an antichain of $P$ of size $\frac{A(F) + 3 }{2}$. Therefore, \[L(F) \leq \frac{A(F) + 3 }{2} \leq w(P).\] If $F$ is even, $\{x , x_2, ..., x_{A(F)}\}$ is an antichain of size $\frac{A(F) + 2 }{2}$. Similarly 
    \[ L(F) = \frac{A(F) + 2 }{2} \leq w(P) .\] 

In this way we can conclude that for any $x$ and $y$ in $P$, $\E(x,y) = L(F) \leq w(P)$. So, $D_\E \leq w(P)$.
\end{proof}

By duality we obtain that  \hbox{$D_{\B} \leq w(P)$} where $D_{\B}$ is the diameter of $P$ with the fence-diver distance.


\section{Calculation of distances and comparisons}
\label{Sec:notable_posets}

In this section, we present explicit formulas for computing the climber and diver distances in some notable classes of posets, including linear orders and lattices. We also describe how these distances behave under standard constructions, such as direct product or lexicographic sums of posets.  We conclude the section comparing these metrics with the classic shortest-path and shortest-fence metrics, the Chebyshev distance, as well as establishing that our metrics are not induced by any valuation on the poset. For this section, let $\delta_{x,y}$ denote the Kronecker delta of $x$ and $y$.

\subsection{Distances on classes of posets}

\subsubsection*{Linear posets}

A poset $P$ is called \textit{linear} if $P$ is itself a chain i.e. all elements in $P$ are comparable. All linear posets are fence-connected posets. For any $x \leq y \in P$ we have that $\{x,y\}$ is a fence with zero alternations from $x$ to $y$. Then, we obtain that for the fence distances: 
\begin{equation*}
    \E(x,y) = \B(x,y) = 1 - \delta_{x,y}.
\end{equation*}

If $|[x,y]| < \infty$, this chain is the only path in $P$ from $x$ to $y$ with zero alternations. Hence for any $k$:

\begin{equation*}
    \E_k(x,y) = \B_k(x,y) = \left\{ 
    \begin{array}{cc}
        \myceil{\ell([x,y])} & \text{ if }|[x,y]| < \infty\\
         \infty & \text{else}.
    \end{array}
    \right. 
\end{equation*}

\subsubsection*{Zig-zag posets} 

A poset is called a \textit{zig-zag} poset if $P = \{x_i\}_{i \in A \subseteq \Z}$ and for all $i$ even, $x_{i-1} \prec x_i \succ x_{i+1}$ or $x_{i-1} \succ x_i \prec x_{i+1}$ and those are all covering relations of $P$. Notice that  for $i \leq j$, $j- i-1$ is the number of alternations of the path on $P$ from $x_i$ to $x_j$. Then we obtain for the climber distances:

\begin{equation*}
    \E_k(x_i,x_j) = \E(x_i,x_j) = \left\{ 
    \def\arraystretch{1.7}
    \begin{array}{cl}
        0 & i=j \\
        \frac{j-i +1 }{2} &  j-i \text{ odd} \\
        \frac{j-i}{2} & j-i \text{ even}, x_i \prec x_{i+1}\\
        \frac{j-i+2}{2} & j-i \text{ even}, x_i \succ x_{i+1}. \\
    \end{array}
    \right.
\end{equation*}

We have a similar formula for $ \B_k(x_i,x_j) = \B(x_i,x_j)$ interchanging $x_i \prec x_{i+1}$ and $x_i \succ x_{i+1}$ on the last two conditions.


\subsubsection*{Crown posets} 

Let $n \in \N$ even. A poset  is $P = \{ \tuple[0]{x}{2n-1}\}$ is called a \textit{$n$-crown poset} if \[x_0 \prec x_1 \succ x_2 \prec \cdots \succ x_{2n-2} \prec x_{2n-1} \succ x_0\]
such that these are all covering relations of $P$.  The diagram of a $n$-crown poset is the diagram of a zig-zag poset with the endpoints joined, in this way $P$ is a path-connected poset. Let $i \leq j$ and take $I,J \geq 0 $ the numbers between 0 and $2n-1$ such that $I \equiv i-j \mod 2n$ and $J \equiv j-i \mod 2n $. Then
    \begin{equation*}
        \E_k(x_i,x_j) = \E(x_i,x_j) = \left\{ 
        \def\arraystretch{1.7}
        \begin{array}{cl}
            \min \{ \frac{I +1 }{2}, \frac{J +1 }{2} \} &  j-i \text{ odd} \\
            \min \{ \frac{I }{2}, \frac{J}{2} \}  & j-i \text{ even and }  x_i \prec x_{i+1}\\
            \min \{ \frac{I +2 }{2}, \frac{J+2 }{2} \}  & j-i \text{ even and } x_i \succ x_{i+1}\\
        \end{array}
        \right.
    \end{equation*}

Similarly to zig-zag posets, the formula for $ \B_k(x_i,x_j) = \B(x_i,x_j)$ are the same interchanging $x_i \prec x_{i+1}$ and $x_i \succ x_{i+1}$ on the last two conditions.

\subsubsection*{Common bounds}

For $x,y$ in a poset $P$, $z$ is a \textit{common upper bound} of $x$ and $y$ if $x \leq z$ and $y \leq z$. Dually, $w$  is a \textit{common lower bound} of $x$ and $y$ if $w \leq x$ and $w \leq y$. A poset $P$ has the upper (or lower) filtering property if any two elements have a common upper (or lower, respectively) bound. 

Take $P$ to have the upper filtering property. For any $x$ and $y$ distinct elements in $P$ take $z$ a common upper bound. Then, there exists the upward fence with one alternation $x < z > y$. Then $\E(x,y) = 1$. Hence, $\E(x,y) = 1 - \delta_{x,y}$ for any $x,y \in P$. In a similar way, if $P$ has the lower filtering property  for any $x$ and $y$ distinct elements, there exists the downward fence with one alternation $x > w < y$ where $w$ is a common lower-bound. Hence, $\B(x,y) = 1 - \delta_{x,y}$. If $P$ has both properties this implies that for any $x,y \in P$ we have  \[\E(x,y) = \B(x,y) = 1 - \delta_{x,y}.\]


\subsubsection*{Modular posets}
 
A poset $P$ is called \textit{upper semi-modular} if and only if for all $x \neq y \in P$ with $w$ a common lower cover of $x$ and $y$ i.e $w \prec x \text{ and } w \prec y$ there exists $z \in P$ such that $x \prec z \text{ and } y \prec z$ . Dually, a poset is \textit{lower semi-modular} if the existence of a common upper cover implies the existence of a common lower cover. A poset is called \textit{modular} if is both upper and lower semi-modular \cite{monjardet}.

\begin{lemma}
    Let $P$ be a lower semi-modular poset. Let $x,y$ in $P$ and suppose there exists an upward path $\gamma = \{\tuple[0]{x}{n}\}$ from $x$ to $y$ with only one alternation at $x_i$ for some $0 <i < n$. Then, there exists a downward path $\gamma^*= \{\tuple[0]{x^*}{n}\}$ from  $x$ to $y$ with $\ell(\gamma^*) = \ell(\gamma)$ and only on alternation at $x^*_{n-i}$. 
    \label{lemma:upward_to_downward_path}
 \end{lemma}
\begin{proof}
    Let $P$ be a lower semi-modular poset and take $x,y$ in $P$. Let $\gamma = \{\tuple[0]{x}{n}\}$ be a path from $x$ to $y$ with only one alternation at $x_i$ for some $0 <i < n$. Let us prove by induction on the length $\ell(\gamma)$ that there exists a downward path $\gamma^*= \{\tuple[0]{x^*}{n}\}$ from  $x$ to $y$ with $\ell(\gamma^*) = \ell(\gamma)$ and only on alternation at $x^*_{n-i}$. 
    
    Suppose $\ell(\gamma) = 2$. Then $\gamma = \{ x \prec x_1 \succ y\}$. Lower semi-modularity ensures that  there exists an element $x_1^* \in P$ such that $x \succ x^*_1 \prec y$. Then, $\gamma^*= \{ x, x^*_1, y \}$ satisfies the claim. Assume that for a fixed $n \geq 2$, if $\gamma$ has elements  $\{\tuple[0]{x}{k}\}$ for $k \leq n$ and has one alternation at $x_i$ for some $0 < i < k$, this implies there exists a downward path $\gamma^*= \{\tuple[0]{x^*}{k}\}$ from $x$ to $y$ with only on alternation at $x^*_{k-i}$.

    Take an upward path $\gamma = \{\tuple[0]{x}{n+1}\} $ from $x$ to $y$ with one alternation at $x_i$ for some $0 < i < n+1$. We have two cases as illustrated in Figure \ref{fig:upward_to_downward_path}. In the first case, \textbf{A},  if $i = n$, then, as shown in the  diagram, the truncated path $\{\tuple[1]{x}{n+1}\}$ from $x_1$ to $y$ satisfies the induction hypothesis and has one alternation at position $n-1$. Hence, there exists a downward path \mbox{$\gamma' = \{\tuple[0]{y}{n}\}$} from $x_1$ to $y$ with with a single alternation at  $y_{n - (n-1) } = y_1$. Thus, 
    \[ \gamma': x_1 \succ y_1 \prec y_{2} \prec \cdots \prec y_{n-1} \prec y.\]

    Then, we have $x \prec x_1 \succ y_1$. Lower semi-modularity ensures that there exists an element $x^*_1 \in P$ such that $x \succ x^*_1 \prec y_1$. Then, the path $\gamma^*= \{\tuple[0]{x^*}{n+1}\}$ with  $x^*_0 = x$ and $x^*_{j} = y_{j-1}$ for $ 1 < j \leq n+1$ is a downward path from $x$ to $y$ with one alternation at $x^*_1 = x^{*}_{(n+1) - n}$ as desired.
    \begin{align*}
         x \succ x^*_1 \prec y_1 \prec y_{2} \prec \cdots \prec y_{n-1} \prec y \\
       \gamma^*:  x^*_0  \succ x^*_1 \prec x^*_2 \prec x^*_3 \prec \cdots \prec x^*_n \prec y \\
    \end{align*}

    \begin{figure}[h]
        \centering
        \includegraphics[width=0.75\textwidth]{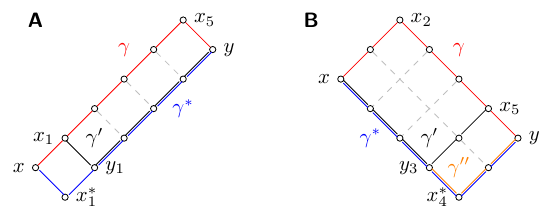}
        \caption{Example of upwards paths $\gamma =\{\tuple[0]{x}{n+1}\}$ from $x$ to $y$ with only one alternation at $x_i$ of length $n+1 =6$ and the resulting downward path $\gamma^*$ from $x$ to  $y$ with alternation at $x^*_{n-i}$. \textbf{A}. The path $\gamma$ has an alternation at $x_5$. The downward path $\gamma^*$ has an alternation at $x^*_1$. \textbf{B}. The path $\gamma$ has an alternation at $x_2$. The downward path $\gamma^*$ has an alternation at $x^*_4$.}
        
        \label{fig:upward_to_downward_path}
    \end{figure}
    
    In the second case, \textbf{B}, assume that $i < n$. As depicted in Figure \ref{fig:upward_to_downward_path}, the truncated path  $\{\tuple[0]{x}{n}\}$ from $x$ to $x_{n}$ satisfies the induction hypothesis: it is an upward path of length $n$ with a single alternation occurring at position $i$. Hence, there exists a downward path \mbox{$\gamma' = \{\tuple[0]{y}{n}\}$} from $x$ to $x_{n}$ with only one alternation at  $y_{n-i}$. Thus, 
    \[\gamma': x \succ y_1 \succ \cdots \succ y_{n-i} \prec y_{n-i + 1} \prec \cdots \prec y_{n-1} \prec x_n.\]
    
    Next, consider the path $\{ y_{n-i}, y_{n-i + 1}, ...,  y_n = x_n, y\}$. This is an upward  path of length $i + 1 \leq n$, with a single alternation at the penultimate element $x_n$, since \mbox{$ y_{n-i} \prec \cdots \prec y_{n-1} \prec y_n = x_n \succ y$}. Applying the inductive hypothesis again, there is a downward path, \mbox{$\gamma'' = \{\tuple[0]{z}{i + 1}\}$} from $y_{n-i}$ to $y$ with only one alternation at  $z_1$. Thus, 
    \[\gamma'': y_{n-i} \succ z_1 \prec z_2 \prec \cdots \prec z_i \prec y.\]

    Finally, we concatenate the downward chain of $\gamma'$ with $\gamma''$ to define $\gamma^*= \{\tuple[0]{x^*}{n+1}\}$. Specifically by 
    
    \[
    x^*_{j} = y_{j} \; \text{ for } \; 0 \leq j \leq n-i,  \qquad x^*_{j} = z_{j-(n-i)} \text{ for }\; n-i + 1 \leq j \leq n+1
    \]  

    Hence,
    \begin{equation*}
        \begin{array}{rcccl}
             x &\succ  y_1 \succ \cdots \succ y_{n-i} \succ &  z_1 &\prec z_2 &\prec \cdots \prec z_{i} \prec y \\
             \gamma^*:  x &\succ x^*_1 \succ \cdots \succ x^*_{n-i} \succ & x^*_{n-i+1} & \prec x^*_{n-i+2} &\prec \cdots \prec x^*_{n} \prec y
        \end{array}
    \end{equation*}
    
    Therefore, $\gamma^*$ is the required downward path from $x$ to $y$ of length $n+1$  with one alternation at \mbox{ $x^*_{(n+1)-i} = z_{n-i+1-(n-i)} = z_1$.}
\end{proof}

\begin{lemma}
   Let $P$ be an upper semi-modular poset. Let $x,y$ in $P$ and suppose there exists a downward path $\gamma = \{\tuple[0]{x}{n}\}$ from $x$ to $y$ with only one alternation at $x_i$ for some $0 <i < n$. Then, there exists an upward path $\gamma^*= \{\tuple[0]{x^*}{n}\}$ from  $x$ to $y$ with $\ell(\gamma^*) = \ell(\gamma)$ and only on alternation at $x^*_{n-i}$. 
    \label{lemma:downward_to_upward_path}
\end{lemma}
\begin{proof}
    This proof is analogous to the one before as this lemma is the dual of Lemma \ref{lemma:upward_to_downward_path} 
\end{proof}

\begin{theorem}
    If $P$ is a modular poset then for all $x,y \in P$ and for any $k$ we have $\E_k(x,y) = \B_k(x,y)$. If $P$ is also discrete then $\E(x,y) = \B(x,y)$.
    \label{th:E1_B1_equal_modular}
\end{theorem}
\begin{proof}
    Take $P$ a modular poset and $x,y$ in $P$. If $x$ and $y$ are in a different path-connected component then for any $k$, $\E_k(x,y) = \infty = \B_k(x,y)$ and we are done.  Assume that $x,y$ are in the same path-connected component and for a fixed $k$, let $\gamma$ be a path from $x$ to $y$  such that $L_k(\gamma) = \E_{k}(x,y)$ .  Let us prove that $\B_k(x,y) \leq \E_k(x,y)$. We have three cases 
    \begin{enumerate}[a)]
        \item Suppose that $A(\gamma)$ is even. Then, by definition of $L^*_k$, we have that  $L^*_k(\gamma) =  L_k(\gamma)$. Thus, as, desired, $\B_k(x,y) \leq L^*_k(\gamma) =  L_k(\gamma) =  \E_k(x,y)$.
         \item Suppose that $A(\gamma)$ is odd and $\gamma$ is an upward path with chain decomposition $\{ \tuple[0]{C}{A(\gamma)} \}$. Take the paths $\gamma_{i} = C_{2i} \cup C_{2i+1}$ for each integer $i$ between $0$ and $\frac{A(\gamma)-1}{2}$.
        
        \begin{minipage}{0.32\textwidth}
            \begin{center}
            \includegraphics[width=0.95\textwidth]{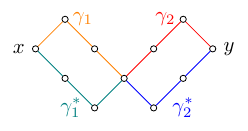}
        \end{center}
        \end{minipage}
        \begin{minipage}{0.6\textwidth}
             Each path $\gamma_i$ is an upward path with exactly one alternation. Since $P$ is lower semi-modular, Lemma \ref{lemma:upward_to_downward_path} guarantees the existence of a downward path $\gamma^*_i$ connecting the same endpoints as $\gamma_i$ with only one alternation. Take the downward path from $x$ to $y$
        \[\gamma^* = \gamma^*_1 \cup \gamma^*_2 \cup \cdots \cup \gamma^*_{\frac{A(\gamma)-1}{2}} \]
        \end{minipage}

       Let  $\{\tuple[0]{C^*}{A(\gamma*)}\}$ denote the chain decomposition of $\gamma^*$. By construction of each $\gamma^*_i$, $A(\gamma^*) = A(\gamma)$ and for all $i$, $\ell(C^*_{2i}) = \ell(C_{2i+1}) \text{ and } \ell(C^* _{2i+1}) = \ell(C_{2i})$. Then we obtain that

         \begin{align*}
            L_k(\gamma) &= \sum\limits_{i=0}^{A(\gamma)} \myceil{\ell(C_i)} - \frac{A(\gamma)+1}{2}             =\sum\limits_{i=0}^{\frac{A(\gamma)-1}{2}} (\myceil{\ell(C_{2i})} + \myceil{\ell(C_{2i+1})} -1) \\ &=\sum\limits_{i=0}^{\frac{A(\gamma)}{2} - 1} (\myceil{\ell(C^*_{2i+1})} + \myceil{\ell(C^*_{2i})} -1) = \sum\limits_{i=0}^{A(\gamma)} \myceil{\ell(C^* _i)} - \frac{A(\gamma)+1}{2} = L^*_k(\gamma^*).
        \end{align*}
        
        Hence, $\B_k(x,y) \leq L^*_k(\gamma) =  L_k(\gamma) =  \E_k(x,y)$. Thus, $\E_k(x,y)$ upper bounds $\B_k(x,y)$.

        \item Suppose that $A(\gamma)$ is odd and  $\gamma$ is a downward path. Then, by definition of $L^*_k$, we have that  $L^*_k(\gamma) =  L_k(\gamma) - 1$. Thus,  $\B_k(x,y) \leq L^*_k(\gamma) < L_k(\gamma) =  \E_k(x,y)$ as desired.
    \end{enumerate}

    In a dual way, we can prove that $\E_k(x,y) \leq \B_k(x,y)$. Take $\gamma^*$ to be a path from $x$ to $y$ such that $L^*_k(\gamma^*) = \B_{k}(x,y)$. We again have three cases
    \begin{enumerate}[a)]
        \item Suppose that $A(\gamma^*)$ is even. Then we have that  $L_k(\gamma^*) =  L^*_k(\gamma^*)$. Thus, we obtain that $\E_k(x,y) \leq L_k(\gamma^*) =  L^*_k(\gamma^*) = \B_k(x,y)$, so $\B_k(x,y)$ upper bounds $\E_k(x,y)$.
        
        \item Suppose that $A(\gamma^*)$ is odd and  $\gamma^*$ is a downward path. Analogous to the case b) above, by Lemma \ref{lemma:downward_to_upward_path} there exists an upward path $\gamma$ such that $L_k(\gamma) =  L^*_k(\gamma^*)$. Hence, $\E_k(x,y) \leq L_k(\gamma^*) =  L^*_k(\gamma^*) = \B_k(x,y)$ as desired.
        \item Suppose that $A(\gamma^*)$ is odd and $\gamma^*$ is an upward path. Hence, as $\gamma$ is an upward path,  $L^*_k$, we have that  $L_k(\gamma^*) =  L^*_k(\gamma^*) - 1$. Thus,  $\E_k(x,y) \leq L_k(\gamma^*) <  L^*_k(\gamma^*) = \B_k(x,y)$ as we wanted.        
    \end{enumerate}

    In this way, $\E_k(x,y) = \B_k(x,y)$ if $x,y$ are in the same path-connected component.

    For the second part, assume now that $P$ is also a discrete poset. If $x,y$ are in a different fence-connected components, then $\E(x,y) = \infty = \B(x,y)$. Assume instead that $x$ and $y$ are in the same fence-connected component, then, they also lie in the same path-connected component (Corollary \ref{coro: fence+discrete}). From the previous part, we therefore have $\E_k(x,y) = \B_k(x,y)$ for every $k$. By theorem \ref{th:E_k_converges_E} and its dual, the sequences $(\E_k(x,y))_{k=1}^\infty$ and $(\B_k(x,y))_{k=1}^\infty$ converges respectively to $\E(x,y) $ and $\B(x,y)$. Since the two sequences agree term-by-term, their limits must coincide. Hence, $\E(x,y) = \B(x,y)$.
\end{proof}

\subsection{Distances under constructions of posets}

In this subsection we describe how to create new posets from existing ones and characterize the climber and diver distances in these new posets. We will use the notation $\E^P$ and $\B^P$ and $\E_k^P$ and $\B_k^P$ to emphasize that each metric is computed with respect to a specific underlying poset $P$. 

\subsubsection*{Subposets}

We can create a new poset from an existing one by restricting to subsets. If $P$ is a poset and $Q \subset P$, then $Q$ becomes a poset by inheriting the order from $P$: whenever $x,y \in Q $ satisfy $x \leq_P y$, we set $x \leq_Q y$. Because the order is inherited, every path or fence that lies in $Q$ is contained in $P$ as well. 

Take $k$ fixed but arbitrary. Let $\gamma$ be a path in $Q$ from $x$ to $y$ such that $L_k(\gamma) = \E^Q_k(x,y)$. Since $\gamma \subset P$ we obtain that
\[ \E^P_k(x,y) \leq L_k(\gamma) = \E^Q_k(x,y). \]

Thus, the $k$-climber distance can only increase. The same monotonicity holds for the fence-climber distance. If $F$ is a fence from $x$ to $y$ in $Q$ that minimizes $\E^Q$ we have that 

\[ \E^P(x,y) \leq L(F) = \E^Q(x,y). \]

By the same argument we conclude as well that 

\[ \B^P_k(x,y) \leq \B^Q_k(x,y)  \quad \text{and} \quad \B^P(x,y) \leq \B^Q(x,y).\]

\subsubsection*{Opposite posets}

Let $P$ be a poset. The \textit{opposite poset} of $P$, $P^{\op}$, has the same underlying set with the order reversed: $x \leq_{P^{\op}} y$ if and only if $y \leq_P x$. Then, the set identity function $\text{Id}$ is a natural anti-isomorphism $\text{Id}: P \to P^{\op}$ from a poset to its opposite. 

Consider $\gamma$ a path in $P$ from $x$ to $y$. The path $\gamma^{\op}$ in $P^{\op}$ is obtained by taking the same elements of $\gamma$. Then, $\gamma^{\op}$ is a path from $x$ to $y$ that has the same length, the same number of alternations but with opposite orientation, i.e, if $\gamma$ is an upward path, then, $\gamma^{\op}$ is a downward path and viceversa. Thus, 
\[
L_k(\gamma) =  L^*_k(\gamma^{op}) \quad \text{and} \quad L^*_k(\gamma) =  L_k(\gamma^{op}).
\]

Moreover, $\gamma$ minimizes $L_k$ in $P$ if and only if $\gamma^{op}$ minimizes $L^*_k$ in $P^{\op}$. In a similar way, $\gamma$ minimizes $L^*_k$ in $P$ if and only if $\gamma^{op}$ minimizes $L_k$ in $P^{\op}$. Hence, for any $x,y$ in $P$ 
\[ \B^{P^{op}}_k(x,y) = \E^P_k(x,y)  \quad \text{and} \quad \E^{P^{op}}_k(x,y) = \B^P_k(x,y).   \]

The same correspondence holds for fences. If $F$ is a fence from $x$ to $y$ in $P$, then the fence $F^{\op}$ in $P^{\op}$ obtained by taking same elements of $F$, has the same number of alternations but with opposite orientation. Consequently,  

\[ \B^{P^{op}}(x,y) = \E^P(x,y)  \quad \text{and} \quad \E^{P^{op}}(x,y) = \B^P(x,y).   \]

In particular, we have the following result for modular posets.

\begin{corollary}
    Let $P$ be a modular poset. For any $k$ we have that for any $x,y \in P$,
    \[ \E^P_k(x,y) =  \E^{P^{\op}}_k(x,y) = \B^P_k(x,y) =\B^{P^{\op}}_k(x,y).\]
    If $P$ is also discrete we have 
    \[ \E^P(x,y) = \E^{P^{\op}}(x,y) =  \B^P(x,y)   =\B^{P^{\op}}(x,y).\]
    \label{cor:modular_posets_distance}  
\end{corollary}
\begin{proof}
    Let $P$ be a modular poset. It is known that if $P$ is a modular poset, then, $P^{\op}$ is also modular. Hence, by Theorem \ref{th:E1_B1_equal_modular} we obtain that 
    \[ \E^P_k(x,y) =  \B^P_k(x,y) \quad \text{ and } \quad \E^{P^{\op}}_k(x,y) =  \B^{P^{\op}}_k(x,y). \]
    
    By the duality of $P$ and $P^{\op}$ we conclude that 
    
    \[ \E^P_k(x,y) =  \B^P_k(x,y) = \E^{P^{\op}}_k(x,y) =\B^{P^{\op}}_k(x,y).\]

    Assume that $P$ is discrete. A similar argument  proves that 
     \[ \E^P(x,y) =  \B^P(x,y) = \E^{P^{\op}}(x,y) =\B^{P^{\op}}(x,y).\]
\end{proof}


\subsubsection*{Direct product of posets}

Take $\{P_i\}_{i=1}^n$ to be a collection of posets. The \textit{direct product} of this collection is the cartesian product  \mbox{$ \displaystyle \bigtimes_{i=1}^n P_i =  P_1 \times P_2 \cdots \times P_{n}$} with the order relation 
\begin{equation*}
    (\tuple{x}{n}) \leq_{\bigtimes_{i=1}^n P_i} (\tuple{y}{n}) \; \text{ if and only if } \; x_i \leq_{P_i} y_i\text{ for all } 1 \leq i \leq n.
\end{equation*}

\begin{lemma}
    Let $P$ and $Q$ be posets. Then, for any $(x,y), (x',y') \in P \times Q$ we obtain that 
    \begin{center}
        $\E^{P \times Q}((x,y), (x',y')) = \max\{ \E^P(x,x') , \E^Q(y,y')\}$.
    \end{center}
    \label{lemma:E_direct_product_of_2}
\end{lemma}
\begin{proof}
    Take $(x,y)$ and $(x',y')$ in $P \times Q$. Suppose there is a fence $F=\{\tuple[0]{f}{m+1}\}$ with $m$ alternations in $P$  from $x$ to $x'$ such that  $L(F) = \E^{P}(x, x')$. Similarly, assume there is a fence $G=\{\tuple[0]{g}{n+1}\}$ in $Q$  from $y$ to $y'$ with $n$ alternations such that  $L(G) = \E^{Q}(y, y')$. Let us prove that we can find a fence $H$ in $P \times Q$ from $(x,y)$ to $(x',y')$ with $L(H) \leq \max\{L(F),L(G)\}$. Then, we obtain the inequality
    \begin{equation*}
        \E^{P \times Q}((x,y), (x',y')) \leq L(H) \leq \max\{L(F),L(G)\} = \max\{ \E^P(x,x') , \E^Q(y,y')\}.
    \end{equation*}
    
    We have three cases for the fences $F$ and $G$: $A(F) < A(G)$, $A(F) = A(G)$ and $A(F) > A(G)$.
    \begin{enumerate}[a)]
        \item Suppose that $A(F) < A(G)$, so, $m < n$. If $F$ and $G$ have the same orientation, then, take the fence $H$ given by 
        \begin{equation*}
            H = \{ (x,y) , (f_1,g_1) ,  ... , (f_{m}, g_m) , (f_{m+1}, g_{m+1}), ...,  (f_{m+1}, g_{n}) = (x',y')\}.
        \end{equation*}
        If $F$ and $G$ have different orientations consider $H$ to be 
        \begin{equation*}
            H = \{ (x,y) , (x,g_1) , (f_1,g_2)  ... , (f_{m}, g_{m+1}) , (f_{m+1}, g_{m+2}), ...,  (f_{m+2}, g_{n}) = (x',y')\}.
        \end{equation*}

        In both cases, the fence $H$ has the same orientation and same number of alternations as $G$. Thus, \mbox{$L(H) = L(G) \leq \max\{L(F),L(G)\}$.}

        \item Suppose that  $A(F) = A(G)$, so, $m = n$.  In the case that $F$ and $G$ have the same orientation, then the fence 
        \[H = \{ (x,y) , (f_1,g_1) ,  ... , (f_{m}, g_m) , (f_{m+1}, g_{m+1}) = (x',y')\} \]
        
        has the same orientation and same number of alternations as both. Thus, \mbox{$L(H) = \max\{L(F),L(G)\}$.}
        
        In the counter case when $F$ and $G$ have different orientations, we use a different strategy to create the fence $H$. If $F$ is the upward fence, so $G$ is downward, and consider
        \begin{equation*}
            H = \{ (x,y) , (f_1,y) , (f_2,g_1)  ... , (f_{m+1}, g_{m}) , (f_{m+1}, g_{m+1}) = (x',y')\}.
        \end{equation*}
        If $G$ is the upward fence instead, take 
        \begin{equation*}
            H = \{ (x,y) , (x,g_1) , (f_1,g_2)  ... , (f_{m}, g_{m+1}) , (f_{m+1}, g_{m+1}) = (x',y')\}.
        \end{equation*}

        In this way, $H$ is an upward fence with $m+1$ alternations. When $m$ is even we obtain $L(H) = \frac{(m + 1)+ 1}{2} = L(F) = L(G)$. If instead $m$ is odd, then \mbox{$L(H) = \frac{(m + 1)+ 2}{2} = L(\cdot)$} where "$\cdot$" denotes whichever of $F$ and $G$ is the downward fence. Thus, \mbox{ $ L(H) \leq \max\{L(F),L(G)\}$}
        
        \item Assume that  $A(F) = A(G)$, then, $m > n$. With an analogous argument to a) we can find a fence $H$ that has the same orientation and same number of alternations of $F$. Thus, \mbox{$L(H) = L(F) \leq \max\{L(F),L(G)\}$.}
    \end{enumerate}

    Hence, we have proved that if $x$ and $x'$ and $y$ and $y$ are fence-connected respectively, then \begin{equation}
        \E^{P \times Q}((x,y), (x',y')) \leq  \max\{ \E^P(x,x') , \E^Q(y,y')\}.
        \label{eq:E_product_bounded}
    \end{equation}

    Nevertheless, \eqref{eq:E_product_bounded} still holds if either of the pairs are not fence-connected because in that case $\max\{ \E^P(x,x') , \E^Q(y,y')\} = \infty$.

    Now, let us prove that $\max\{ \E^P(x,x') , \E^Q(y,y')\} \leq \E^{P \times Q}((x,y), (x',y'))$. In this way, together with \eqref{eq:E_product_bounded} we can conclude that 
    \begin{equation*}
        \E^{P \times Q}((x,y), (x',y')) =  \max\{ \E^P(x,x') , \E^Q(y,y')\}.
    \end{equation*}

    If $(x,y)$ and $ (x',y')$ are not fence-connected we trivially obtain that 
    \[\max\{ \E^P(x,x') , \E^Q(y,y')\} \leq \infty = \E^{P \times Q}((x,y), (x',y')) .\] 
    
    Thus, assume there is a fence $H = \{ (x,y), (x_1,y_1), \ldots, (x_n,y_n), (x',y')\} $ from $(x,y)$ to $(x',y')$ such that \mbox{$L(H) = \E^{P \times Q}((x,y), (x',y'))$}. Consider the projection of $H$ onto $P$, namely $ \{x,\tuple{x}{n},x'\}$. If this set is itself a fence, then it has the same number of alternations and the same orientation as $H$. Thus $L(\{x,\tuple{x}{n},x'\}) = L(H)$. If instead, the projection fails to be a fence, then it necessarily contains a proper subset, $F$, that is a fence from $x$ to $x'$ with $A(F) < n$. Hence, we can check that $L(F) \leq L(H)$. In either case, 
    \[
        \E^P(x,x') \leq L(F) \leq \E^{P \times Q}((x,y), (x',y')).
    \]

    In a similar way, we can prove that $\E^Q(y,y') \leq \E^{P \times Q}((x,y), (x',y')$ considering the projection of $H$ onto Q, $\{y,\tuple{y}{n},y'\}$. Therefore, we have proved that
    \begin{equation*}
            \max\left\{ \E^P(x,x') , \E^Q(y,y') \right\} \leq \E^{P \times Q}((x,y), (x',y')).
    \end{equation*} 
\end{proof}

\begin{theorem}
    Let $(\tuple{x}{n}),(\tuple{y}{n}) \in \displaystyle \bigtimes_{i=1}^n P_i$. Then 
    \begin{equation*}
        \E^{\bigtimes_{i=1}^n P_i}((\tuple{x}{n}),(\tuple{y}{n})) = \max_{1 \leq i \leq n} \{ \E^{P_i}(x_i,y_i) \}. 
    \end{equation*}
\end{theorem}
\begin{proof}
    Assume all the hypothesis and let us prove the claim for the fence-climber distance by induction on $n$. If $n=1$, then $\bigtimes_{i=1}^n P_i = P_1$ and the claim is trivially true.
    
    Suppose that for a certain $n$, 
    \begin{equation*}
        \E^{\bigtimes_{i=1}^{n-1} P_i}((\tuple{x}{n-1}),(\tuple{y}{n-1})) = \max_{1 \leq i \leq n-1} \{ \E^{P_i}(x_i,y_i) \}
    \end{equation*}

    and consider the direct product $\bigtimes_{i=1}^{n} P_i$. Take $P = \bigtimes_{i=1}^{n-1} P_i$ and $Q = P_n$. Note that 
    \[ (\tuple{x}{n}) <_{\bigtimes P_i} (\tuple{y}{n}) \; \text { if and only if }\; ((\tuple{x}{n-1}) , x_n) <_{P \times Q} ((\tuple{y}{n-1}) , y_n)\]

    Hence, we can check that if $F = \{ (\tuple{x^i}{n}) \mid 0 \leq i \leq m\}$ is a fence in $\bigtimes_{i=1}^{n} P_i$ from $(\tuple{x}{n})$ to $(\tuple{y}{n})$, then the fence $F'  = \{ ((\tuple{x^i}{n-1}), x^i_n) \mid 0 \leq i \leq m\}$ is a fence in $P \times Q$ from $((\tuple{x}{n-1}),x_n)$ to $((\tuple{y}{n-1}),y_n)$ with the same length, same number of alternations and same orientation. Hence $L(F') = L(F)$. Moreover, $F$ minimizes $L$ in $\bigtimes_{i=1}^{n} P_i$ if and only if $F'$ minimizes $L$ in $P \times Q$. Thus, 
    \[
    \E^{\bigtimes_{i=1}^{n} P_i}((\tuple{x}{n}),(\tuple{y}{n})) = \E^{P \times Q} \big( ((\tuple{x}{n-1}),x_n), ((\tuple{y}{n-1}), y_n) \big).
    \]
    
    Then, by Lemma \ref{lemma:E_direct_product_of_2} and the induction hypothesis, 
    \begin{align*}
        \E^{\bigtimes_{i=1}^{n} P_i}((\tuple{x}{n}),(\tuple{y}{n})) &= \E^{P \times Q} \big( ((\tuple{x}{n-1}),x_n), ((\tuple{y}{n-1}), y_n) \big)\\
        &=\max \{ \E^P((\tuple{x}{n-1}),(\tuple{y}{n-1})) , \E^Q(x_n,y_n)\} \\
        &= \max \{ \max_{1 \leq i \leq n-1} \{ \E^{P_i}(x_i,y_i) \} , \E^{P_n}(x_n,y_n)\} \\
        &= \max_{1 \leq i \leq n} \{ \E^{P_i}(x_i,y_i) \}.
    \end{align*}
    Thus, concluding the proof. 
\end{proof}

\subsubsection*{Lexicographic sum of posets}

Let $T$ be a poset and take $\{P_t\}_{t \in T}$  a be a family of pairwise disjoint nonempty posets indexed by $T$ that are all disjoint from $T$.  The  \textit{lexicographic sum} $\displaystyle \lex[t \in T] P_t$ of this family is the set with elements $\displaystyle \bigcup_{t \in T} P_t$ and order relation $x \leq_{\lex} y$ if 1) $x \leq_{P_s} y$ for some $s \in T$ or 2) $x \in P_s$ and $y \in P_t$ with $s \leq_T t$.

The \textit{disjoint union} $\displaystyle \bigsqcup_{i=1}^n P_i$ is a special case of lexicographic sum with $T$ the set of $n$ elements not comparable pairwise. The \textit{ordinal sum} $\displaystyle \sum_{i=1}^n P_i$ is a lexicographic sum with $T = \{1,2, ... n\}$ with the usual order. 

\begin{lemma}
    Take $x,y$ in $\lex[t \in T] P_t$, the lexicographic sum of $\{P_t\}_{t \in T}$. Then,
    \begin{equation*}
        \E^{\lex} (x,y) = \left\{ 
        \def\arraystretch{1.5}
        \begin{array}{cc}
            1 - \delta_{x,y} & x,y \in P_s \text{ for some } s \text{ upper bounded in } T \\ 
            \min\{\E^{P_s}(x,y) , 2 \} & x,y \in P_s \text{ for some } s \text{ maximal and lower bounded in } T \\
            \E^{P_s}(x,y) & x,y \in P_s \text{ for some } s \text{ not comparable in } T \\
            \E^{T}(s,t) &  x\in P_s , y \in P_t.
        \end{array}
        \right.
    \end{equation*}
\end{lemma}
\begin{proof}
    Take $x,y$ in $\lex[t \in T] P_t$, the lexicographic sum of $\{P_t\}_{t \in T}$. Assume that $x,y$ are in the same poset $P_s$. First, suppose there is $t \in T$ such that $s < t$. Then there is an element $z \in P_t$ such that $x <_{\lex} z >_{\lex} y$. Hence $\{ x,z,y\}$ forms an upward fence from $x$ to $y$ with one alternation. This implies that $\E^{\lex}(x,y) = 1$. This establishes the first case.
    
    Second, suppose $s$ is maximal and there is $r \in T$ such that $r<s$. Hence, there is $w \in P_r$ such that $x >_{\lex} w <_{\lex} y$. Thus, $\{ x,w,y\}$ is a downward fence with one alternation. This implies that $\E^{\lex}(x,y) \leq 2$. This verifies the second case. Note that we require $s$ to be maximal to ensure that this case excludes the first one.

    Finally, in the case that $s$ is not comparable to any other element in $T$, the possible fences from $x$ to $y$ in $\lex[t \in T] P_t$ are exactly the fences in $P_s$. Thus $\E^{\lex}(x,y) = \E^{P_s}(x,y)$.
    
    Now assume that  $x$ in $P_s$ and $y$ in $P_t$ for some $s,t \in T$. If $s,t$ are in the same fence-connected component there is a fence $F = \{\tuple[0]{s}{n+1}\}$ from $s$ to $t$ such that $L(F) = \E^T(s,t)$. Then, if we choose $x_i \in P_{s_i} $ for each $i$ between 1 and $n$ we obtain that $F' = \{x, \tuple{x}{n},y\}$ is a fence from $x$ to $y$ with $L(F') = L(F)$. We can check that $L(F')$ is minimal in $\lex[t \in T] P_t$ if and only if $L(F)$ is minimal in $T$.  Hence, 
    \[  \E^{\lex} (x,y) = L(F') = L(F) =  \E^T(s,t). \]
    
    If $s$ and $t$ are not fence-connected, then $x$ and $y$ are not fence connected. Otherwise, it yields a contradiction. If there is a fence $F' = \{x, \tuple{x}{n},y\}$ from $x$ to $y$, the set $\{ s_i \mid x_i \in P_{s_i}\} $ would contain a fence from $s$ to $t$. Then, \[\E^{\lex} (x,y) = \infty = \E^T(s,t). \]
    This proves the last case.   
\end{proof}

\begin{corollary}
    Take $x,y$ in $\bigsqcup_{i=1}^n P_i$. Then, $\E^{\bigsqcup}(x,y) = \E^{P_i} $ if $x$ and $y$ are in $P_i$ and $\E^{\bigsqcup}(x,y) = \infty$ otherwise.
\end{corollary}

\begin{corollary}
    Take $x,y$ in $\sum_{i=1}^n P_i$. Then, $\E^{\sum}(x,y) = 1 - \delta_{x,y}$ if $x \in P_i$ and $y \in P_j$ for some $i< n$ and $j < n$ and $\E^{\sum}(x,y) = \min \{\E^{P_n}(x,y), 2\}$ if $x$ and $y$ are in $P_n$.
\end{corollary}

\subsection{Comparison with other metrics}

Let us compare the $k$-climber and $k$-diver distances to other known metrics defined on path-connected posets.

\subsubsection*{Shortest path and shortest fence metrics}

Let $P$ a path-connected poset and take $x,y \in P$. Recall that $\D$ denotes the shortest-path metric on $P$, defined by taking $\D(x,y)$ to be the minimal length of a path from $x$ to $y$ in $P$. That is,   
\begin{equation*}
    \D(x,y) = \min \{ \ell(\gamma) \mid \gamma \text { is a path that from } x \text{ to } y \}.
\end{equation*}

Given that for any path $\gamma$ and $k \geq 1$ we have $L_k(\gamma) \leq L_1(\gamma) \leq \ell(\gamma) $ we can conclude that 
\begin{equation*}
    \E_k(x,y) \leq \D(x,y)
\end{equation*}

for all $k \geq 1$. In a similar way we obtain that for any $k \geq 1$,
\begin{equation*}
    \B_k(x,y) \leq \D(x,y).
\end{equation*}

Therefore, the k-climber and k-diver metrics are bounded above by the shortest path metric.

In like manner,  for $P$ a fence-connected poset, $\F$ denotes the shortest-fence metric given by  
\begin{equation*}
    \F(x,y) = \min \{ n \mid \{ \tuple[0]{f}{n} \} \text { is a fence from } x \text{ to } y \}.
\end{equation*}
for $x,y \in P$.

We will prove that for any fence $F = \{ \tuple[0]{f}{n}\}$, the inequality $L(F) \leq n$ holds. In this way, by definition of both, the shortest fence metric and the fence climber metric we conclude that for any $x,y \in P$
\begin{equation*}
    \E(x,y) \leq \F(x,y).
\end{equation*}

By an analogous argument we can prove that $ \D(x,y) \leq \F(x,y)$. Hence, the fence-climber and fence-diver metrics are bounded above by the shortest fence metric.

Take $F = \{x\}$, then $n=0$ and $L(F) = 0 = n$. If instead $F = \{x,y\}$, then $n=1$ and $L(F) = 1 = n$.

Take $F = \{ \tuple[0]{f}{n} \}$ with $n \geq 2$. Recall that $A(F) = n-1 $. Then, for any $* \in \{1,2,3\}$ 

\begin{align*}
    * - 1 &\leq n \\
    n + * - 1 &\leq 2n \\
    \frac{n + * - 1}{2} &\leq n \\
    L(F) &\leq n
\end{align*}

\subsubsection*{Metrics induced by real valued functions}

Take a function $v: P \to \R$. If $P$ is a path-connected poset, then, the function $v$ induces a metric on $P$ in the following way \cite{remarques}. Take $x,y \in P$ and let $\gamma = \{ \tuple{x}{n}\}$  be a path from $x$ to $y$. We define the $v$-length of $\gamma$ by 
\begin{equation*}
    \tilde{v}(\gamma) =  \sum_{i = 1}^n | v(x_{i}) - v(x_{i-1})|  .
\end{equation*}

The $v$-distance of $x,y$ is defined by the minimal $v$-length over all paths from $x$ and $y$, i.e
\begin{equation*} 
    d_v(x,y) = \min \{ \tilde{v}(\gamma) \mid \gamma \text { is a  path from } x \text{ to } y \}.
\end{equation*}

\begin{lemma}
    Let $P$ a path-connected poset and $v:P \to \R$. Then, for all $x < y \in P$, $d_v(x,y) = |v(y) - v(x)|$.
\end{lemma}
\begin{proof}
    Let $\gamma = \{ \tuple{x}{n}\}$ a path from $x$ to $y$. Then 
    \begin{equation*}
        \tilde{v}(\gamma) = \sum_{i = 1}^n | v(x_{i}) - v(x_{i-1})|  \geq \left| \sum_{i = 1}^n  (v(x_{i}) - v(x_{i-1}) )\right| = |v(y) - v(x)| = \tilde{v}(\{x,y\}).
    \end{equation*}
    Hence the path $\{x,y\}$ has the minimal $v$-length which implies that $d_v(x,y) = |v(y) - v(x)|$.
\end{proof}

\begin{corollary}
    Let $\gamma = \{ \tuple{x}{n}\}$ a path from $x$ to $y$ such that $d_v(x,y) = \tilde{v}(\gamma)$. Then, 
    \begin{equation*}
        d_v(x,y) =  d_v(x,x_1) + d_v(x_1,x_2)+ \ldots + d_v(x_{n-1},y).
    \end{equation*}
    \label{coro:dv_splits}
\end{corollary}

Let $d: P \times P \to R$ be a distance defined on $P$. If $d$ is induced by a real valued function $v$, that is, $d = d_v$, then Corollary \ref{coro:dv_splits} provides a necessary condition that $d$ must satisfy. Explicitly, for every pair $x,y \in P$, there must exist a path $\gamma = \{ \tuple{x}{n}\} $ from $x$ to $y$ such that 

\begin{equation*}
    d(x,y) =  d(x,x_1) + d(x_1,x_2)+ \ldots + d(x_{n-1},y).
\end{equation*}

If no such path exists for some $x,y \in P$, then the distance $d$ cannot arise from any real valued function on $P$.

Using this result, in general, the $k$-climber and $k$-diver distances do not arise from real valued functions on $P$. Take $P$ a poset with elements $x,y,z$  such that $x \prec z \succ y$. We have that $\E_k(x,y) = \E_k(x,z) = \E_k(y,z) = 1$. The only path from $x$ to $y$ is $\{x,z,y\}$ but
\[ \E_k(x,y)  \neq  \E_k(x,z) + \E_k(y,z).\]

Similarly, the $k$-diver distance fails the condition when we consider  $w,x,y$ in $P$ such that $x \succ w \prec y$.

In particular, this implies that these metrics are not induced by valuations, a special type of real valued functions on $P$.

\subsubsection*{The Chebyshev distance}

Take $P$ a discrete poset. For $x \leq y$ in $P$ define the \textit{height} of $y$ above $x$, denoted $h(x,y)$ the smallest length of a finite maximal chain in $[x,y]$. We say that $P$ is a \textit{join semi-lattice} if for any $x,y \in P$ there exists a least common upper bound, denoted by $x \vee y$. In this way, for any $z \in P$ such that $x \leq z$ and $z \geq y$ we have $x \vee y \leq z$. 

The Chebyshev distance between $x$ and $y$ in $P$ is then defined by \[d(x,y) = \max\{ h(x, x \vee y), h(y, x \vee y) \}.\] This function defines a metric on $P$ when $P$ is a discrete upper semi-modular join semi-lattice poset \cite{discrete_poset_distances}. We will prove that for any $x,y \in P$,
    \[1 = \E(x,y) \leq d(x,y) \leq \E_1(x,y). \]

\begin{lemma}
     Let $P$ be an upper semi-modular poset and take $x,y \in P$. Let $\gamma$ be a path from $x$ to $y$ with chain decomposition $\{ \tuple[0]{C}{A(\gamma)}\}$. There exists an upward path $\hat{\gamma}$ from $x$ to $y$ in $P$ with chain decomposition $\{\hat{C}_0, \hat{C}_1\}$ such that \[\ell(\hat{C} _0) = \sum\limits_{i \text{ even}} \ell(C_i) \quad \text{and} \quad \ell(\hat{C}_1) = \sum\limits_{i \text{ odd}} \ell(C_i)\]
     if $\gamma $ is an upward path and 

    \[\ell(\hat{C} _0) = \sum\limits_{i \text{ odd}} \ell(C_i) \quad \text{and} \quad \ell(\hat{C}_1) = \sum\limits_{i \text{ even}} \ell(C_i)\]
    if $\gamma $ is a downward path.
    \label{lemma:path_to_one_upward_path} 
\end{lemma}
\begin{proof}
    Let $P$ be an upper semi-modular poset and take $x,y \in P$. Let $\gamma$ be a path from $x$ to $y$ with chain decomposition $\{ \tuple[0]{C}{A(\gamma)}\}$. We have two cases, if $\gamma$ is is an upward path or if $\gamma$ is a downward path. We will prove the statement for each case by induction.

    For the first case, assume that $\gamma$ is an upward path. For the base case, if $A(\gamma) = 1$, then $\hat{\gamma}= \gamma$. For hypothesis of induction, suppose that for some $k$, if $A(\gamma) = k$, then there exists  $\hat{\gamma}$ with chain decomposition $\{\hat{C}_0, \hat{C}_1 \}$ such that 
    \[\ell(\hat{C}_0) = \sum\limits_{i \text{ even}} \ell(C_i) \quad \text{and} \quad \ell(\hat{C}_1) = \sum\limits_{i \text{ odd}} \ell(C_i).\]

    Now, consider $\gamma$ to be a path with $k+1$ alternations with chain decomposition $\{ \tuple[0]{C}{k+1}\}$. The truncated path $\gamma' = C_0 \cup \cdots \cup C_k$ is a path with $k$ alternations, then, by hypothesis of induction, there is $\hat{\gamma'}$ with chain decomposition $\{\hat{C'}_0, \hat{C'}_1 \}$ such that 
    
    \[\ell(\hat{C'}_0) = \sum\limits_{\substack{0 \leq i \leq k \\ i \text{ even}}} \ell(C_i) \quad \text{and} \quad \ell(\hat{C'}_1) = \sum\limits_{\substack{0 \leq i \leq k \\ i \text{ odd}}} \ell(C_i).\]

    \begin{figure}[h]
        \centering
        \includegraphics[width=0.7\textwidth]{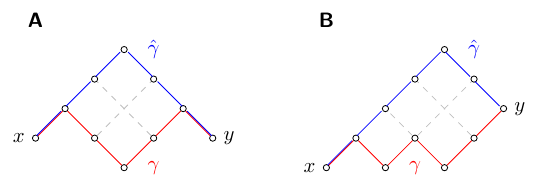}
         \caption{Example of upwards paths $\gamma =\{\tuple[0]{x}{n+1}\}$ from $x$ to $y$ with $k+1$ alternations and the resulting upward path $\hat{\gamma}$ from $x$ to $y$ with one alternation. \textbf{A}. The path $\gamma$ has an odd number of alternations, so $k$ is even. \textbf{B}. The path $\gamma$ has an even number of alternations, so $k$ is odd.}
        
        \label{fig:path_to_only_one_alternation}
    \end{figure}

    Take the path $\hat{\gamma'} \cup C_{k+1}$. If $k$ is even (Figure \ref{fig:path_to_only_one_alternation}, \textbf{A}), then $\hat{\gamma'} \cup C_{k+1}$ has only one alternation. Thus, take $\hat{\gamma} = \hat{\gamma'} \cup C_{k+1}$ and with chain decomposition $\{\hat{C}_0, \hat{C}_1 \}$ with $\hat{C}_0 = \hat{C'}_0$ and $ \hat{C}_1 = \hat{C'}_1 \cup C_{k+1}$. Hence,
    \[\ell(\hat{C}_0) = \sum\limits_{\substack{0 \leq i \leq k \\ i \text{ even}}} \ell(C_i) = \sum\limits_{\substack{0 \leq i \leq k+1 \\ i \text{ even}}} \ell(C_i) \]
    and 
    \[\ell(\hat{C}_1) = \sum\limits_{\substack{0 \leq i \leq k \\ i \text{ odd}}} \ell(C_i) + \ell(C_{k+1}) =  \sum\limits_{\substack{0 \leq i \leq k+1 \\ i \text{ odd}}} \ell(C_i)\]
    as desired.

    If $k$ is odd, then $\hat{\gamma'} \cup C_{k+1}$ has two alternations. In particular, $\hat{C'}_1 \cup C_{k+1}$ is a downward path with one alternation. By Lemma \ref{lemma:downward_to_upward_path}, since $P$ is an upper semi-modular poset, there is an upward path $\gamma^*$ with chain decomposition $\{C^*_0, C^*_1\}$ such that it has the same endpoints of $\hat{C'}_1 \cup C_{k+1}$ and
    \[ \ell(C^*_0) = \ell(C_{k+1}) \quad \text{and} \quad \ell(C^*_1) = \ell({\hat{C'}_1}).\]

    Hence, $\hat{\gamma} =  \hat{C'}_0 \cup C^*_0 \cup C^*_1$ is an upward path with chain decomposition $\{\hat{C}_0, \hat{C}_1 \}$ where $\hat{C}_0 = \hat{C'}_0 \cup C^*_0$ and $ \hat{C}_1 = C^*_1 $. In this way
    \[ 
    \ell(\hat{C}_0 ) = \ell(\hat{C'}_0) + \ell(C^*_0) = \sum\limits_{\substack{0 \leq i \leq k \\ i \text{ even}}} \ell(C_i) + \ell(C_{k+1}) = \sum\limits_{\substack{0 \leq i \leq k+1 \\ i \text{ even}}} \ell(C_i)\]
    and 
    \[\ell(\hat{C}_1 ) = \ell(C^*_1) = \ell(\hat{C'}_1) = \sum\limits_{\substack{0 \leq i \leq k \\ i \text{ odd}}} \ell(C_i)  = \sum\limits_{\substack{0 \leq i \leq k+1 \\ i \text{ odd}}} \ell(C_i)\]

    For the second case, assume that $\gamma$ is a downward path. If $A(\gamma) = 1$, take $\hat{\gamma}=\gamma^* $ the upward path from Lemma \ref{lemma:downward_to_upward_path} when considering $\gamma$. Hence,
    \[\ell(\hat{C}_0) = \ell(C^*_0) = \ell(C_1) \quad \text{and} \quad \ell(\hat{C}_1) = \ell(C^*_1) = \ell(C_0). \]
    
    Then, the induction process is analogous to the first case. 
\end{proof}

\begin{lemma}
    (Jordan-Dedekind condition) Let $P$ be a discrete upper semi-modular poset. Then in any given interval $[x, y]$ in $P$ all maximal chains have the same number of elements.
    \label{lemma:Jordan-Dedekind}
\end{lemma}
\begin{proof}
    Direct consequence of Theorem 3 by Haskins and Gudder \cite{haskins1971semimodular}, which proves that an upper semi-modular poset without infinite chains satisfies the Jordan-Dedekind chain condition.
\end{proof}

\begin{theorem}
    Let $P$ be a discrete upper semi-modular join semi-lattice poset. Then for any $x,y \in P$,
    \[1 - \delta_{x,y} = \E(x,y) \leq d(x,y) \leq \E_1(x,y). \]
\end{theorem}
\begin{proof}
    Let $P$ a discrete upper semi-modular join semi-lattice poset and take $x,y \in P$. Given that $x,y$ have $x \vee y$ as a common upper bound it is clear that 
    \[1 - \delta_{x,y} = \E(x,y) \leq d(x,y)\]
    
    Now, choose a path $\gamma$ from $x$ to $y$ with chain decomposition $\{ \tuple[0]{C}{A(\gamma)}\}$ such that $L_1(\gamma) = \E_1(x,y)$. By Lemma \ref{lemma:path_to_one_upward_path}, there is a path $\hat{\gamma}$ from $x$ to $y$ with chain decomposition $\{\hat{C}_0, \hat{C}_1 \}$ such that

    \[\ell(\hat{C}_0) = \sum\limits_{C_i \text{ is upward}} \ell(C_i) \quad \text{and} \quad \ell(\hat{C}_1) = \sum\limits_{C_i \text{ is downward}} \ell(C_i).\]

    Take $z \in P$ such that $\hat{\gamma}$ has an alternation at $z$. Then $x < z$ and $z > y$. On one hand,  $z > x \vee y$. Thus, given that $P$ satisfies the Jordan-Dedekind condition (Lemma \ref{lemma:Jordan-Dedekind}), $h(x,z) = h(x,x \vee y ) + h(x \vee y,z)$. In particular, this implies that 
    \begin{equation}
        d(x,y) = \max\{h(x,x \vee y), h(y,x \vee y)\} \leq \max\{h(x,z) , h(y,z)\}.
        \label{eq:h_h}
    \end{equation}

     On the other hand, since $\hat{C_0}$ is a path from $x$ to $z$ and $\hat{C_1}$ is a path from $z$ to $y$, then 
    \begin{equation}
        \max\{h(x,z), h(y,z)\} \leq \max\{\ell(\hat{C}_0) , \ell(\hat{C}_1)\}.
        \label{eq:h_l}
    \end{equation}
 
    If we prove that 
    \begin{equation}
        \max\{\ell(\hat{C}_0) , \ell(\hat{C}_1)\} \leq \E_1(x,y)
        \label{eq:l_l}
    \end{equation}

Then, inequalities \eqref{eq:h_h}, \eqref{eq:h_l} and \eqref{eq:l_l} together imply the desired inequality
\[ d(x,y) \leq \E_1(x,y).\]
as desired.

Hence, let us prove inequality \eqref{eq:l_l}. Let $E = \sum\limits_{i \text{ even}} \ell(C_i)$ and $E' = \sum\limits_{i \text{ odd}} \ell(C_i)$. Thus $E + E' = \ell(\gamma)$.

Suppose that $A(\gamma)$ is even. Then $\frac{A(\gamma)}{2} < E$ and $\frac{A(\gamma)}{2} \leq E'$. 

Hence,
\begin{align}
    0 & \leq  E' - \frac{A(\gamma)}{2} \nonumber\\
    E & \leq E + E'- \frac{A(\gamma)}{2} = L_1(\gamma).
    \label{eq:even_chains}
\end{align}

In a similar way,
\begin{equation}
    E' \leq E + E'- \frac{A(\gamma)}{2} = L_1(\gamma).
    \label{eq:odd_chains}
\end{equation}
    
Then, inequalities \eqref{eq:even_chains}  and \eqref{eq:odd_chains} together imply \eqref{eq:l_l} given that
\[ \max\{\ell(\hat{C}_0) , \ell(\hat{C}_1)\} = \max\{ E, E'  \} \leq L_1(\gamma) = \E_1(x,y). \]

Analogously, if $A(\gamma)$ is odd then it can be proved that $\frac{A(\gamma)+1}{2}$ is bounded by $E$ and $E$'. Therefore $E$ and $E'$ are both bounded by $L_1(\gamma)$.

\end{proof}


\section{Characterization of finite posets}
\label{Sec:characterization}

Let $P$ and $Q$ be posets equipped with (extended) metrics, and let $f: P \to Q$ be a bijective isometry. It is natural to ask whether $f$ must then be an order isomorphism. This is already false in very simple cases: let $P$ be a three element chain and let $Q$ and $R$  be the three-element zig-zag posets (Figure \ref{fig:3posets}). Equip each poset with the shortest-path metric $\D$. The map $f:P \to Q$ defined by $f(a) = a'$, $f(b) = b' $ and $f(c) = c' $ is a bijective isometry, but $f$ is not an order isomorphism because $b \prec c$ while $f(b) \succ f(c)$. 

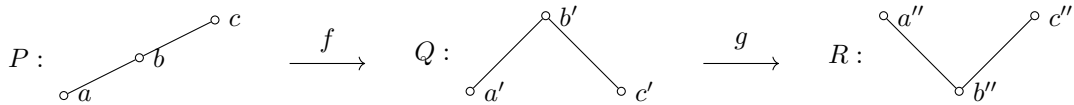
\begin{figure}[h]
\begin{center}
    \begin{tikzpicture}[
        node/.style={
        circle, draw=black,
        inner sep=0pt, minimum size=3pt,
        fill=white},
        ]

        \node at (-0.5,0.5){$P:$};
        
        \node (a) [node, label = right: $a$ ]  at (0,0){};
        \node (b) [node, label = right: $b$ ]  at (1,0.5){};
        \node (c) [node, label = right: $c$]  at (2,1){};
        
        \draw (a) -- (b) (b) -- (c);
\end{tikzpicture}
\hspace{1em}
\begin{tikzpicture}
    \node at (0,.3) {$f$};
    \draw[->] (-0.5,0) -- (0.5,0);
    \node at (0,-.5){};
\end{tikzpicture}
\hspace{1em}
\begin{tikzpicture}[
        node/.style={
        circle, draw=black,
        inner sep=0pt, minimum size=3pt,
        fill=white},
        ]

         \node at (-0.5,0.5){$Q:$};
         
        \node (a) [node, label = right: $a'$ ]  at (0,0){};
        \node (b) [node, label = right: $b'$ ]  at (1,1){};
        \node (c) [node, label = right: $c'$]  at (2,0){};
        
        \draw (a) -- (b) (b) -- (c);
\end{tikzpicture}
\hspace{1em}
\begin{tikzpicture}
    \node at (0,.3) {$g$};
    \draw[->] (-0.5,0) -- (0.5,0);
    \node at (0,-.5){};
\end{tikzpicture}
\hspace{1em}
\begin{tikzpicture}[
        node/.style={
        circle, draw=black,
        inner sep=0pt, minimum size=3pt,
        fill=white},
        ]

         \node at (-0.5,0.5){$R:$};
         
        \node (a) [node, label = right: $a''$ ]  at (0,1){};
        \node (b) [node, label = right: $b''$ ]  at (1,0){};
        \node (c) [node, label = right: $c''$]  at (2,1){};
        
        \draw (a) -- (b) (b) -- (c);
\end{tikzpicture}
\end{center}

\caption{Example of posets with cardinality 3.  We can set the bijective functions $f: P \to Q$ and $g: Q \to R$ given by $f(a) = a'$, $f(b) = b' $ and $f(c) = c' $ and $g(a') = a''$, $g(b') = b'' $ and $g(c') = c'' $.}
\label{fig:3posets}
\end{figure}

Nevertheless, Belding \cite{belding} proved that ``path-connected posets are characterized to within isomorphism and duality by their distance and fence values''. Concretely, for path-connected  posets $(P, \leq_P)$ and $(Q, \leq_Q)$,  a bijection $f: P \to Q$ preserves both the shortest-path metric $\D$ and the shortest-fence metric $\F$ if and only if $f$ is either an isomorphism or an anti-isomorphism. 

Belding’s theorem therefore classifies path-connected posets up to duality. The appearance of anti-isomorphisms in the conclusion arises from a fundamental limitation: the metrics $\D$ and $\F$ are insensitive to orientation. They cannot distinguish a poset from its opposite. For instance, the zig-zag posets $Q$ and $R$ in Figure \ref{fig:3posets} are not isomorphic, yet they have identical distance and fence metrics. Indeed, the natural bijection $g: Q \to R$ defined by $g(a') = a''$, $g(b') = b'' $ and $g(c') = c'' $ preserves $\D$ and $\F$, while reversing the order; thus $g$ is an anti-isomorphism.

In contrast, the metrics $\E_1$ and $\B_1$ are orientation sensitive. There are no bijections between $Q$ and $R$ that preserve $\E_1$ and $\B_1$ because $\E_1(a',c') = 1 \neq 2 = \E_1(a'',c'') $ and $\B_1(a',c') = 2 \neq 1 = \B_1(a'',c'')$. This is because $\E_1$ and $\B_1$ are orientation dependent.

One might therefore expect that preserving $\E_1$ and $\B_1$ forces a bijection to be an isomorphism, However, this fails for modular posets: if $P$ modular, then, the bijection $f: P \to P^{\op}$ defined by $f(x) = x$ preserves $\E_1$ and $\B_1$. 

In fact, by Corollary \ref{cor:modular_posets_distance}, for any $x,y \in P$ we obtain that 

\[\E_1^P(x,y) = \E_1^{P^{\op}}(x,y) \quad \text{ and } \quad \B_1^P(x,y) = \B_1^{P^{\op}}(x,y), \] 
yet, $f$ is an anti-isomorphism. Figure \ref{fig:modular_non_connected} illustrates this phenomenon for two opposite modular posets  $\{a,b,c,d,e\}$ and $\{a',b',c',d',e'\}$.

\begin{figure}[h]
    \centering
    \includegraphics[width=0.6\textwidth]{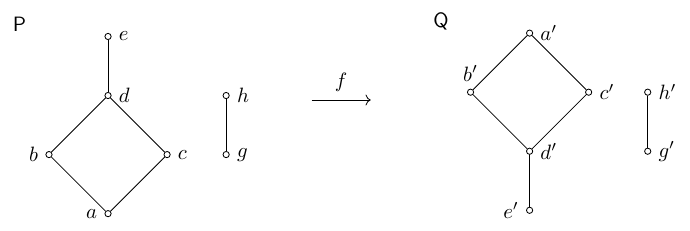}
    \caption{Example of two non path-connected modular posets. The function $f: P \to Q$ defined by $f(x) = x'$ is a bijection that preserves  $\E_1$ and $\B_1$ but is not an isomorphism neither an anti-isomorphism.}
    \label{fig:modular_non_connected}
\end{figure}

Moreover, as in Belding’s result, the hypothesis of path-connectedness is essential. Take $P = \{a,b,c,d,e\} \sqcup \{g,h\}$ with $g \prec h$ and $Q = \{a',b',c',d',e'\} \sqcup \{g',h'\}$ with $g' \prec h'$. The bijection $f: P \to Q$ defined by $f(x) = x'$ is a bijection that preserves  $\E_1$ and $\B_1$ but is not an isomorphism (since $a \prec b$ but $f(a) \succ f(b)$) nor an anti-isomorphism ($g \prec h$ but $f(g) \succ f(h)$).

\subsection{Main Result}

Motivated by Belding’s characterization and by the fact that the metrics $\E_1$ and $\B_1$ do distinguish orientation, our main result (Theorems \ref{th:f_iso_preserves_EB} and \ref{th:f_anti_preserves_EB}) show that the 1-climber and 1-diver metrics characterize most discrete path-connected posets up to isomorphism, with a single exception: modular posets, for which anti-isomorphisms continue to preserve $\E_1$ and $\B_1$. 

We emphasize “most” path-connected posets because of an obstruction arising from the class of infinite diamond–pentagon ladders, a family of posets that will be described later.

\begin{theorem}
    Let $P$ and $Q$ be discrete path-connected posets, neither of which contains a subposet isomorphic to an infinite diamond–pentagon ladder, and let $f: P \to Q$ be a bijection. The following are equivalent:
    \begin{enumerate}[a)]
        \item $f$ is an isomorphism.
        \item The metrics $\E_1$ and $\B_1$ are invariant under $f$ and $f$ preserves at least one covering relation. 
        \item The maps $f$ and $f^{-1}$ preserve covering relations.
    \end{enumerate}

    Here, $f$ is said to preserve covering relations if $x \prec y$ implies that $f(x) \prec f(y)$.
    \label{th:f_iso_preserves_EB}
\end{theorem}

\begin{theorem}
  Let $P$ and $Q$ be discrete path-connected posets, neither of which contains a subposet isomorphic to an infinite diamond–pentagon ladder. The following are equivalent:
    \begin{enumerate}[a)]
         \item There exists a bijection $f: P \to Q$ that preserves $\E_1$ and $\B_1$ and reverse at least one covering relation. 
        \item There exists an anti-isomorphism $f: P \to Q$ that preserves $\E_1$ and $\B_1$.
        \item $P$ is a modular poset and $ P^{op} \cong Q$. 
    \end{enumerate}

    Here, $f$ is said to reverse covering relations if $x \prec y$ implies that $f(x) \succ f(y)$.
    \label{th:f_anti_preserves_EB}
\end{theorem}

In this way, our results strengthen Belding’s classification: any bijection preserving $\E_1$ and $\B_1$ must be an isomorphism unless the poset is modular. However, the exclusion of non-discrete posets and infinite diamond–pentagon ladders is essential, as the conclusion fails in both settings, as illustrated in Examples~\ref{eg:non_discrete_fails} and \ref{eg:finite_fails}, respectively.

\begin{example}
     \label{eg:non_discrete_fails}
     Consider the non-discrete poset $P = \N \cup \{\omega\} \cup \{a\}$ where the order is defined as follows: for $m,n \in N$ we set $m < n$ if $m < n$ in $\N$; moreover, $\omega > m$ for all $m \in \N$ and $0 \prec a \succ \omega$. Similarly, consider the poset $Q = \N \cup \{a\} \cup \{b\}$ where the order extends the usual order on $\N$ and $0 \prec a \succ b$.
     \begin{center}
         \includegraphics[width=0.7\linewidth]{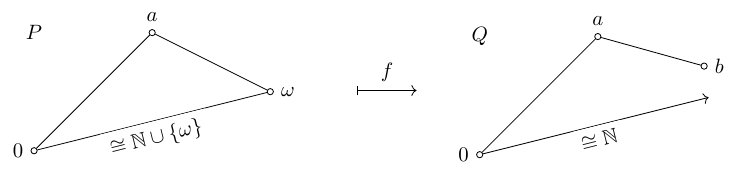}
     \end{center}

     Define a map $f: P \to Q$ given by $f(n) = n $ if $n \in \N$, $f(a) =a$ and $f(\omega) = b$. Then, $f$ is a bijection. Moreover, it preserves $\E_1$ and $\B_1$. Indeed, for any $x,y \in P$, let $\gamma = \{\tuple[0]{x}{n}\}$ be the unique path from $x$ to $y$. Then, $f(\gamma) = \{f(x_0), ... f(x_n)\}$ is the unique path from $f(x)$ to $f(y)$ in $Q$. However, $f$ is not an isomorphism as $0 < \omega$ in $P$, but, $f(0)$ and $f(\omega)$ are not comparable in $Q$.
\end{example}

\begin{example}
    \label{eg:finite_fails}
    Consider the posets $P = \{ x_i \mid i \in \Z \} \cup \{y_i \mid i \in \Z\}$ defined by the covering relations  $x_i \prec x_{i+1}$,  $y_i \prec y_{i+1}$ and  $x_i \prec y_{i}$. Similarly define $Q = \{ x'_i \mid i \in \Z \} \cup \{y'_i \mid i \in \Z\}$ with the covering relations $x'_i \prec x'_{i+1}$,  $y'_i \prec y'_{i+1}$ and  $y'_i \prec x'_{i+1}$.
    \begin{center}
        \includegraphics[width=0.5\textwidth]{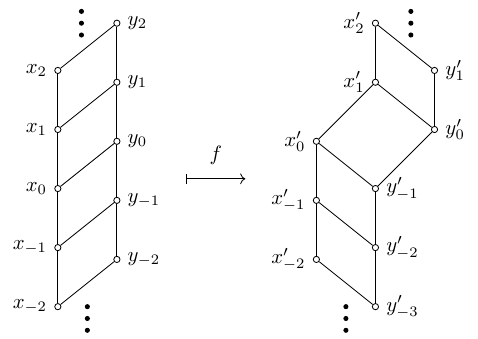}
    \end{center}

    Take $f: P \to Q$ given by $f(x_i) = x'_i$ and $f(y_i) = y'_i$ for all $i \in \Z$. Hence, the function $f$ is a bijection. Let us prove that $f$ preserves $\E_1$ and $\B_1$ but is not an isomorphism nor an anti-isomorphism since $x_0 \prec y_0$ but $f(x_0)$ and $f(y_0)$ are not comparable.

    Note that $P$ and $Q$ are modular posets, thus, let us calculate the 1-climber distance for each pairs of elements in each poset and the 1-diver distance will be the same. On one hand, for the pairs that have the same letter we can check that for any $i,j \in \Z$
    \begin{equation*}
        \E_1(x_i,x_j) =  \E_1(x'_i,x'_j)  = \E_1(y_i,y_j) =  \E_1(y'_i,y'_j) = |i-j|.
    \end{equation*}

    because if $i \leq j$, $x_i$ and $x_j$ are connected by the chain $x_i \prec x_{i+1} \prec \cdots \prec x_j$. The same reasoning applies to the elements with the letters  $x',y$ and $y'$.
    
    On the other hand, for any $i \in \Z$, $x_i \prec y_i$. Also the elements $x'_i$ and $y'_i$ are connected by the upward fence $x'_i \prec x'_{i+1} \succ y'_i$. Thus, 
    \[\E_1(x_i,y_i) = 1  \quad \text{and} \quad \E_1(x'_i, y'_i)  = 2 - \frac{1+1}{2} = 1.\]

    In the case that $j < i $, the elements $x_{i}$ and $y_j$ are connected by the upward path 
    \[x_{i} \prec y_{i} \succ y_{i-1} \succ  \cdots \succ y_j \]  
    of length $i-j+1$ with one alternation. 
    We can check that this path is minimal for $L_1$. Hence,
    \[
    \E_1(x_i, y_j) =  i-j+1 - \frac{1+1}{2} = i-j. 
    \]

    Besides, the elements $x'_{i}$ and $y'_j$ are connected by the chain $x'_i \succ y'_{i-1} \succ \cdots \succ y'_j$  of length $i - j $. Thus, we obtain that 
     \[
        \E_1(x'_i, y'_j) = i-j - \frac{0}{2} = i-j.
    \]
    
    In the case that $j > i$, the elements $x_{i}$ and $y_j$ are connected by the chain $x_i \succ y_i \succ \cdots \succ y_j$ of length $j - i + 1$ Hence
     \[
        \E_1(x_i, y_j) = j-i+1 - \frac{0}{2} = j-i + 1
    \]

    Meanwhile, $x'_{i}$ and $y'_j$ are connected  by the upward path 
    \[x'_{i} \prec x'_{i+1} \succ y'_{i} \prec  \cdots \prec y'_j.\]  
     of length $i-j+2$ with two alternations
    
    We can check that this path is minimal for $L_1$. Hence,
    \[
    \E_1(x'_i, y'_j) =  j-i+2 - \frac{2}{2} = j-i + 1 
    \]
 
    Thus, in summary 

    \begin{equation*}
        \E_1(x_i,y_j)   = \B_1(x_i,y_j) = \E_1(x'_i,y'_j)  = \B_1(x'_i,y'_j)  = 
        \begin{cases}
            1 & \text{ if } j = i  \\
           i-j  & \text{ if } j < i \\
           j-i + 1 & \text{ if } j > i 
        \end{cases}
    \end{equation*}
\end{example}

The main reason the conclusion fails in this second example is the existence of adjacent elements in $P$ for which their images are not adjacent in $Q$. In the next section we will prove that if $P$ does not contain a subposet isomorphic to an infinite diamond–pentagon ladder, then, for all adjacent $x,y \in P$, $f(x)$ and $f(y)$ must be adjacent in $Q$.

\subsection{Supporting theorem}

Let us start by defining a diamond poset and a concave pentagon poset. A \textit{diamond poset} is a 4 element poset $\{a,b,c,d\}$ with cover relations $a \prec b \prec d$ and $a \prec c \prec d$.
A \textit{up concave pentagon poset} is 5 element poset $\{a,b,c,d,e\}$ with cover relations $a \prec b \prec e $, $a \prec d$, $c \prec d$ and $c \prec e$. The dual of the up concave pentagon is the \textit{down concave pentagon poset}. 

\begin{center}
    \begin{tikzpicture}[
        node/.style={
        circle, draw=black,
        inner sep=0pt, minimum size=3pt,
        fill=white},
        ]
        
        \node (y) [node, label = right: $b$ ]  at (3,0){};
        \node (x) [node, label = left: $a$ ]  at (2,-0.5){};
        \node (z) [node, label = left: $c$ ]  at (2,0.5){};
        \node (w) [node, label = right: $d$]  at (3,1){};
        
        \draw (x) -- (y) -- (w) -- (z) -- (x);

        \node at (2.5,-1){Diamond};
    \end{tikzpicture}
    \hspace{5em}
    \begin{tikzpicture}[
        node/.style={
        circle, draw=black,
        inner sep=0pt, minimum size=3pt,
        fill=white},
        ]
        
        \node (y) [node, label = right: $b$ ]  at (2.5,0){};
        \node (x) [node, label = left: $a$ ]  at (1,-0.5){};
        \node (z) [node, label = below: $c$ ]  at (1.75,0.4){};
        \node (v) [node, label = left: $d$ ]  at (1,0.5){};
        \node (w) [node, label = right: $e$]  at (2.5,1){};
        
        \draw (x) -- (y) -- (w) -- (z) -- (v) -- (x);

        \node at (1.75,-1){Up concave pentagon};
    \end{tikzpicture}
    \hspace{5em}
    \begin{tikzpicture}[
        node/.style={
        circle, draw=black,
        inner sep=0pt, minimum size=3pt,
        fill=white},
        ]

        \node (a) [node, label = left: $a$ ]  at (1,1){};
        \node (b) [node, label = right: $b$ ]  at (2.5,0.5){};
        \node (c) [node, label = below: $c$ ]  at (1.75,0.1){};
        \node (d) [node, label = left: $d$ ]  at (1,0){};
        \node (e) [node, label = right: $e$]  at (2.5,-0.5){};
                
        \draw (a) -- (b) -- (e) -- (c) -- (d) -- (a);

        \node at (1.75,-1){Down concave pentagon};
    \end{tikzpicture}
\end{center}

Let $P_1$ and $P_2$ posets, each of which is either a diamond poset or an up concave pentagon poset. An \textit{up diamond–pentagon ladder with two steps} is a poset obtained by “gluing” $P_1$ and $P_2$ along a common covering relation, in a compatible orientation, yielding one of the following posets:

\begin{center}
    \begin{tikzpicture}[
        node/.style={
        circle, draw=black,
        inner sep=0pt, minimum size=3pt,
        fill=white},
        ]
        \node (b) [node]  at (3,0){};
        \node (a) [node]  at (2,-0.5){};
        \node (c) [node]  at (2,0.5){};
        \node (d) [node]  at (3,1){};
        \node (e) [node]  at (2,1.5){};
        \node (f) [node]  at (3,2){};

        \node at (2.5,0.25) {$P_1$}; 
         \node at (2.5,1.25) {$P_2$};
        
        \draw (a) -- (b) -- (d) -- (c) -- (a);
        \draw (c) -- (e) -- (f) -- (d);
        
    \end{tikzpicture}
    \hspace{3em}
     \begin{tikzpicture}[
        node/.style={
        circle, draw=black,
        inner sep=0pt, minimum size=3pt,
        fill=white},
        ]

        \node (b) [node]  at (3,0){};
        \node (a) [node]  at (2,-0.5){};
        \node (c) [node]  at (2,0.5){};
        \node (d) [node]  at (3,1){};
        \node (e) [node]  at (2,1.5){};
        \node (f) [node]  at (2.5,1.4){};
        \node (g) [node]  at (3,2){};

        \node at (2.5,0.25) {$P_1$}; 
        \node at (2.5,1.1) {$P_2$};
        
        \draw (a) -- (b) -- (d) -- (c) -- (a);
        \draw (c) -- (e) -- (f) -- (g) -- (d);
    \end{tikzpicture}
    \hspace{3em}
     \begin{tikzpicture}[
        node/.style={
        circle, draw=black,
        inner sep=0pt, minimum size=3pt,
        fill=white},
        ]

        \node (b) [node]  at (2.5,0){};
        \node (a) [node ]  at (1,-0.5){};
        \node (c) [node]  at (1.75,0.4){};
        \node (d) [node ]  at (1,0.5){};
        \node (e) [node]  at (2.5,1){};
        \node (g) [node]  at (1.75,1.4){};
        \node (f) [node ]  at (1,1.5){};

        \node at (1.75,0.05) {$P_1$}; 
        \node at (1.38,1) {$P_2$};
        
        \draw (a) -- (b) -- (e) -- (c) -- (d) -- (a);
        \draw (d) -- (f) -- (g) -- (c);
    \end{tikzpicture}
    \hspace{2em}
     \begin{tikzpicture}[
        node/.style={
        circle, draw=black,
        inner sep=0pt, minimum size=3pt,
        fill=white},
        ]

        \node (b) [node]  at (2.5,0){};
        \node (a) [node ]  at (1,-0.5){};
        \node (c) [node]  at (1.75,0.4){};
        \node (d) [node ]  at (1,0.5){};
        \node (e) [node]  at (2.5,1){};
        \node (g) [node]  at (1.75,1.4){};
        \node (f) [node ]  at (2.5,2){};

        \node at (1.75,0.05) {$P_1$}; 
        \node at (2.13,1.2) {$P_2$};
        
        \draw (a) -- (b) -- (e) -- (c) -- (d) -- (a);
        \draw (e) -- (f) -- (g) -- (c);
    \end{tikzpicture}
    \hspace{2em}
     \begin{tikzpicture}[
        node/.style={
        circle, draw=black,
        inner sep=0pt, minimum size=3pt,
        fill=white},
        ]

        \node (b) [node]  at (2.5,0){};
        \node (a) [node ]  at (1,-0.5){};
        \node (c) [node]  at (1.75,0.4){};
        \node (d) [node ]  at (1,0.5){};
        \node (e) [node]  at (2.5,1){};
        \node (g) [node]  at (1.75,1.4){};
        \node (f) [node]  at (1,1.5){};
        \node (h) [node] at (1.38, 1.25){};

        \node at (1.75,0.05) {$P_1$}; 
        \node at (1.38,0.9) {$P_2$};
        
        \draw (a) -- (b) -- (e) -- (c) -- (d) -- (a);
        \draw (d) -- (f) -- (h) -- (g) -- (c);
    \end{tikzpicture}
    \hspace{2em}
     \begin{tikzpicture}[
        node/.style={
        circle, draw=black,
        inner sep=0pt, minimum size=3pt,
        fill=white},
        ]

        \node (b) [node]  at (2.5,0){};
        \node (a) [node ]  at (1,-0.5){};
        \node (c) [node]  at (1.75,0.4){};
        \node (d) [node ]  at (1,0.5){};
        \node (e) [node]  at (2.5,1){};
        \node (g) [node]  at (1.75,1.5){};
        \node (f) [node ]  at (2.5,2){};
        \node (h) [node] at (2.13, 1.4){};

        \node at (1.75,0.05) {$P_1$}; 
        \node at (2.13,1.1) {$P_2$};
        
        \draw (a) -- (b) -- (e) -- (c) -- (d) -- (a);
        \draw (e) -- (f) -- (h) -- (g) -- (c);
    \end{tikzpicture}    
\end{center}

Recursively, an \textit{up diamond–pentagon ladder with $n$ steps} is a poset obtained by “gluing” a sequence of poset $\{P_i\}_{i=1}^n$, where each $P_i$  is either a diamond poset or an up concave pentagon poset, and where each consecutive pair $P_i$ and $P_{i+1}$ is an up diamond–pentagon ladder with two steps. In the limit, as $n$ tends to infinity, the resulting poset is an \textit{infinite up diamond–pentagon ladder}.

Dually, a \textit{down diamond–pentagon ladder with $n$ steps} is a sequence of poset $\{P_i\}_{i=1}^n$, where each $P_i$ is either a diamond poset or a down concave pentagon poset an each consecutive pair $P_i$ and $P_{i+1}$ is the dual of an up diamond–pentagon ladder with two steps. When $n$ tends to infinity, the resulting poset is an \textit{infinite down diamond–pentagon ladder}.

\begin{lemma}
    Let $f: P \to Q$ a bijection that preserves $\E_1$. If there are $x,y,z \in P$ such that $x \prec z$ and $y \prec z$ and $f(x) \prec f(y)$. Then, the set $\{f(x), f(y), f(z)\}$ together with one or two elements, is a subposet of $Q$ isomorphic to a diamond poset or an up concave pentagon poset respectively.
    \label{lemma:diamond_pentagon_shape}
\end{lemma}
\begin{proof}
    Take $f: P \to Q$ and $x,y,z \in P$ that satisfy the statement. Given that  $x \prec z$ and $y \prec z$ we obtain that $\E_1(x,z) = \E_1(y,z) = 1$. Since $f$ preserves $\E_1$  we have that $\E_1(f(x),f(z)) = \E_1(f(y),f(z)) = 1$. Thus, we have four possible cases:

    \begin{enumerate}[i) ]
        \item $f(y) \prec f(z)$ and $f(x) \prec f(z)$.
        \item $f(y) \prec f(z)$ and there is $w' \in Q$ such that $f(x) \prec w'$ and $f(z) \prec w'$.
        \item There is $w \in Q$ such that $f(y) \prec w$ and $f(z) \prec w$ and $f(x) \prec f(z)$.
        \item  There is $w \in Q$ such that $f(y) \prec w$ and $f(z) \prec w$ and there is $w' \in Q$ such that $f(x) \prec w'$ and $f(z) \prec w'$.
    \end{enumerate}

We see that case i) leads to a contradiction: $f(z)$ is an upper cover of $f(x)$ as $f(x) \prec f(z)$. However, $f(z)$ is not an upper cover of $f(x)$  because $f(x) \prec f(y) \prec f(z)$. Similarly, case ii) leads to a similar contradiction for $w'$ and $f(x)$ since $f(x) \prec w'$ and $f(x) \prec f(y) \prec f(z) \prec w'$ cannot hold simultaneously. Finally, cases iii) and iv) give us the desired subposets. Case iii) correspond to diamond poset and case iv) to an up concave pentagon poset.
\begin{center}
    \begin{tikzpicture}[
        node/.style={
        circle, draw=black,
        inner sep=0pt, minimum size=3pt,
        fill=white},
        ]

        \node at (1,1){iii)};
        
        \node (y) [node, label = right: $f(y)$ ]  at (3,0){};
        \node (x) [node, label = left: $f(x)$ ]  at (2,-0.5){};
        \node (z) [node, label = left: $f(z)$ ]  at (2,0.5){};
        \node (w) [node, label = right: $w$]  at (3,1){};
        
        \draw (x) -- (y) -- (w) -- (z) -- (x);
    \end{tikzpicture}
    \hspace{5em}
    \begin{tikzpicture}[
        node/.style={
        circle, draw=black,
        inner sep=0pt, minimum size=3pt,
        fill=white},
        ]

        \node at (0.5,1){iv)};
        
        \node (y) [node, label = right: $f(y)$ ]  at (2.5,0){};
        \node (x) [node, label = left: $f(x)$ ]  at (1,-0.5){};
        \node (z) [node, label = below: $f(z)$ ]  at (1.75,0.4){};
        \node (v) [node, label = left: $w'$ ]  at (1,0.5){};
        \node (w) [node, label = right: $w$]  at (2.5,1){};

        \draw (x) -- (y) -- (w) -- (z) -- (v) -- (x);
    \end{tikzpicture}
\end{center}
\end{proof}

Dually, let $f: P \to Q$ a bijection that preserves $\B_1$. If there are $x,y,z \in P$ such that $x \succ z$ and $y \succ z$ and $f(x) \succ f(y)$. Then, the set $\{f(x), f(y), f(z)\}$ together with one or two elements, is a subposet of $Q$ isomorphic to a diamond poset or a down concave pentagon poset respectively.

\begin{lemma}
   Let $P$ and $Q$ be posets and let $f: P \to Q$ be a bijection that preserves $\E_1$. If $P$ has elements $x$ and $y$ that are adjacent in $P$ but $f(x)$ and $f(y)$ are not adjacent in $Q$, then $P$ contains a subposet isomorphic to an infinite up diamond–pentagon ladder.
    \label{lemma:covers_preserved_E}
\end{lemma}
\begin{proof}
    Assume that there are $x \prec y$ such that $f(x)$ and $f(y)$ are not adjacent. Given that $\E_1(x,y) = 1$, $f(x)$ and $f(y)$ are not adjacent and $f$ preserves $\E_1$, then there is $z_1 \in Q$ such that  $f(x) \prec z_1$ and $f(y) \prec z_1$. Recall that $f^{-1}$ also preserves $\E_1$ (Lemma \ref{lemma_f-1_preserves}). Hence, by Lemma \ref{lemma:diamond_pentagon_shape}, we obtain that the set 
    $\{x, y, f^{-1}(z_1)\}$ together with one or two elements, is a subposet of $P$ isomorphic to a diamond poset or a concave pentagon poset respectively.  In particular call $w_1 \in P$ the element such that $y = f^{-1}(f(y)) \prec w_1$ and $f^{-1}(z_1) \prec w_1$. Define the  subset $A_1 = \{ x,y,f^{-1}(z_1), w_1 \} \subseteq P$. We use the figure of a trapeze with a left dashed vertical line to represent either a diamond or a up concave pentagon omitting one of the elements

    \begin{center}
    \begin{tikzpicture}[
        node/.style={
        circle, draw=black,
        inner sep=0pt, minimum size=3pt,
        fill=white},
        ]
        \node at (0,0){$A_1:$};
        
        \node (y) [node, label = right: $y$ ]  at (3,0){};
        \node (x) [node, label = left: $x$ ]  at (1.5,-0.5){};
        \node (z1) [node, label = left: $f^{-1}(z_1)$ ]  at (1.6,0.5){};
        \node (w1) [node, label = right: $w_1$]  at (3,1){};
        
        \draw (x) -- (y) (y) -- (w1) (z1) -- (w1);
        \draw [dashed] (z1) -- (x);
    \end{tikzpicture}
    \end{center}

    Consider $\hat{x} = y$, $\hat{y} = f^{-1}(z_1)$ and $\hat{z} = w_1$ in $P$. Given that 
    
    \[\hat{x} = y \prec w_1 = \hat{z}, \quad  \hat{y} = f^{-1}(z_1) \prec w_1 = \hat{z} \quad f(\hat{x}) = f(y) \prec z_1 = f(\hat{y}), \] 
    
    by Lemma \ref{lemma:diamond_pentagon_shape} we obtain there is $z_2 \in Q - f(A_1)$ such that $z_1 = f(\hat{y}) \prec z_2$ and $f(w_1) = f(\hat{z}) \prec z_2$. In an analogous way, we find $w_2 \in P - A_1$ with the covering relations $w_1 = f^{-1}(\hat{y}) \prec w_2$ and $f^{-1}(z_2) = f^{-1}(\hat{z}) \prec w_2$ considering in $Q$ the elements 
    \[\hat{x} = z_1, \quad \hat{y} = f(w_1), \quad \hat{z} = z_2, \] 
    again, by Lemma \ref{lemma:diamond_pentagon_shape},  Then, take $A_2 = A_1 \cup \{ f^{-1}(z_2) , w_2) \}$.

    \begin{center}
    \begin{tikzpicture}[
        node/.style={
        circle, draw=black,
        inner sep=0pt, minimum size=3pt,
        fill=white},
        ]

        \node at (-0.5,0){$A_2:$};
        
        \node (y) [node, label = right: $y$ ]  at (3,0){};
        \node (x) [node, label = left: $x$ ]  at (1.5,-0.5){};
        \node (z1) [node, label = left: $f^{-1}(z_1)$ ]  at (1.6,0.5){};
        \node (w1) [node, label = right: $w_1$]  at (3,1){};
        \node (z2) [node, label = left: $f^{-1}(z_2)$]  at (1.7,1.5){};
        \node (w2) [node, label = right: $w_2$]  at (3,2){};
        
        \draw (x) -- (y) (y) -- (w1) (z1) -- (w1) -- (w2) (z2) -- (w2);
        \draw [dashed] (z2) -- (z1) -- (x);
    \end{tikzpicture}
\hspace{5em}
    \begin{tikzpicture}[
        node/.style={
        circle, draw=black,
        inner sep=0pt, minimum size=3pt,
        fill=white},
        ]

        \node at (-0.5,0){$f(A_2 - \{w_2\}):$};
        
        \node (y) [node, label = right: $f(y)$ ]  at (4.5,0){};
        \node (x) [node, label = left: $f(x)$ ]  at (2,0){};
        \node (z1) [node, label = left: $z_1$]  at (3,0.5){};
        \node (z2) [node, label = left: $z_2$]  at (3,1.5){};
        \node (w1) [node, label = right: $f(w_1)$]  at (4.4,1){};
        \node at (1,-1){};
        
        \draw (x) -- (z1) (y) -- (z1) -- (z2) (w1) -- (z2);
        \draw [dashed] (w1) -- (y);
    \end{tikzpicture}
    \end{center}

    Recursively, for $n \geq 3$, define $A_n = A_{n-1} \cup \{ f^{-1}(z_n),w_n\}$ where $z_n \in Q - f(A_{n-1})$ is the result of Lemma \ref{lemma:diamond_pentagon_shape} for the elements in $P$
    \begin{equation*}
        \hat{x} = w_{n-2} \quad \hat{y} = f^{-1}(z_{n-1}) \quad \hat{z} = w_{n-1}, 
    \end{equation*}
    
    and $w_n \in P - A_{n-1}$ the result of Lemma \ref{lemma:diamond_pentagon_shape} considering in $Q$,
     \begin{equation*}
        \hat{x} = z_{n-1} \quad \hat{y} = f(w_{n-1}) \quad \hat{z} = z_{n}, 
    \end{equation*}

    \begin{center}
    \begin{tikzpicture}[
        node/.style={
        circle, draw=black,
        inner sep=0pt, minimum size=3pt,
        fill=white},
        point/.style={
        circle, draw=black,
        inner sep=0pt, minimum size=2pt,
        fill=black},
        ]

        \node at (-0.5,1){$A_n:$};
        
        \node (y) [node, label = right: $y$ ]  at (3,0){};
        \node (x) [node, label = right: $x$ ]  at (1.5,-0.5){};
        \node (z1) [node, label = left: $f^{-1}(z_1)$ ]  at (1.6,0.5){};
        \node (w1) [node, label = right: $w_1$]  at (3,1){};

        \node [point] at (2.4,1.2){};
        \node [point] at (2.4,1.4){};
        \node [point] at (2.4,1.6){};

        \node (w) [node, label = right: $w_{n-2}$]  at (3,2){};
        \node (z2) [node, label = left: $f^{-1}(z_{n-1})$]  at (1.6,2.5){};
        \node (w2) [node, label = right: $w_{n-1}$]  at (3,3){};
        \node (z3) [node, label = left: $f^{-1}(z_n)$]  at (1.7,3.5){};
        \node (w3) [node, label = right: $w_n$]  at (3,4){};
        
        \draw (x) -- (y) (y) -- (w1) (z1) -- (w1)  (z2) -- (w2) (z3) -- (w3) (w) -- (w2) -- (w3);
        \draw [dashed] (z3) -- (z2)  (z1) -- (x);
    \end{tikzpicture}
\hspace{5em}
    \begin{tikzpicture}[
        node/.style={
        circle, draw=black,
        inner sep=0pt, minimum size=3pt,
        fill=white},
        point/.style={
        circle, draw=black,
        inner sep=0pt, minimum size=2pt,
        fill=black},
        ]

        \node at (-0.5,2){$f(A_n - \{w_n\}):$};

        \node (x) [node, label = left: $f(x)$ ]  at (2,0){};
        \node (y) [node, label = right: $f(y)$ ]  at (4.5,0){};
        \node (z1) [node, label = left: $z_1$]  at (3,0.5){};
        \node (z2) [node, label =  left: $z_2$]  at (3,1.5){};
        \node (w1) [node, label = right: $f(w_1)$]  at (4.4,1){};
        
        \node [point] at (3.6,1.8){};
        \node [point] at (3.6,2){};
        \node [point] at (3.6,2.2){};

        \node (z) [node, label = left: $z_{n-2}$]  at (3,2.5){};
        \node (w2) [node, label = right: $f(w_{n-2})$ ]  at (4.4,3){};
        \node (w3) [node, label = right: $f(w_{n-1})$]  at (4.3,4){};
        \node (z3) [node, label = left: $z_{n-1}$]  at (3,3.5){};
        \node (z4) [node, label =  left: $z_n$]  at (3,4.5){};
        
        \draw (x) -- (z1) (y) -- (z1) -- (z2) (w1) -- (z2) (z) -- (z3) -- (z4) (z3) -- (w2) (z4) -- (w3);
        \draw [dashed] (w3) -- (w2) (w1) -- (y);
    \end{tikzpicture}
    \end{center}

    Let $A = \lim_{n \to \infty} A_n$.  Then $A$ is a subposet of $P$ isomorphic to an infinite up diamond–pentagon ladder.
\end{proof}

\begin{theorem}
    Let $P$ and $Q$ be posets such that $P$ does not contain a subposet isomorphic to an infinite diamond–pentagon ladder, and let $f: P \to Q$ be a bijection that preserves $\E_1$ and $\B_1$. Then for all adjacent elements $x$ and $y$ in $P$, the elements $f(x)$ and $f(y)$ are also adjacent in $Q$.
    \label{th:covers_preserved}
\end{theorem}
\begin{proof}
    Suppose, for contradiction, that there exist adjacent elements $x$ and $y$ in $P$ such that  $f(x)$ and $f(y)$ are not adjacent in $Q$. Since $f$ preserves $\E_1$, Lemma \ref{lemma:covers_preserved_E} implies that $P$ contains a subposet isomorphic to an infinite up diamond–pentagon ladder. Applying the dual statement of the same lemma and using that $f$ also preserves $\B_1$, we likewise obtain that $P$ contains a subposet isomorphic to an infinite down diamond–pentagon ladder.
    
    In either case, $P$ contains a subposet isomorphic to an infinite diamond–pentagon ladder, contradicting the hypothesis. Therefore, $f(x)$ and $f(y)$ must be adjacent in $Q$ for all adjacent $x,y \in P$.
\end{proof}

\subsection{Proof of the main result}

We are ready to prove Theorems \ref{th:f_iso_preserves_EB} and \ref{th:f_anti_preserves_EB}.

\begin{proof}[Proof of Theorem \ref{th:f_iso_preserves_EB}]

    Let $P$ and $Q$ be discrete path-connected posets, neither of which contains a subposet isomorphic to an infinite diamond–pentagon ladder, and let $f: P \to Q$ be a bijection.
    
    a) $\implies$ b). Assume that $f$ is an isomorphism. Let us prove the metrics $\E_1$ and $\B_1$ are invariant under $f$ and $f$ preserves at least one covering relation.
    Given that $f$ is an isomorphism, then it is clear that $f$ preserves covering relations. Take $x,y \in P$ and $\gamma$ a path connecting $x$ and $y$ in $P$. Since $f$ is a bijective order-preserving map, the image $f(\gamma)$ is a path in $Q$ connecting $f(x)$ and $f(y)$ with the same length and same number of alternations. Then, take $\gamma$ a path in $P$ with $L_1(\gamma) = \E_1(x,y)$.  Hence,
    \[\E_1(f(x),f(y)) \leq L_1(f(\gamma)) = L_1(\gamma) = \E_1(x,y).\]
    
    Similarly, a path connecting $f(x)$ and $f(y)$ in $Q$ yields a path connecting $x$ and $y$  in $P$ because $f^{-1}$ is also a bijective order-preserving map. Applying the same argument to $f^{-1}$ and a minimal path connecting $f(x)$ and $f(y)$ yields the reverse inequality \mbox{$\E_1(x,y) \leq \E_1(f(x),f(y))$}.  Therefore, $\E_1(x,y) = \E_1(f(x),f(y))$.  An analogous argument shows that $\B_1(x,y) = \B_1(f(x),f(y))$.

    b) $\implies$ c). Assume that $\E_1$ and $\B_1$ are invariant under $f$ and $f$ preserves at least one covering relation. We will prove that $f$ is an isomorphism. By hypothesis, take $x \prec y$ in $P$ such that $f(x) \prec f(y)$. Let us prove $f$ preserves all covering relations involving $x$ and $y$. Then we have four cases.
    
    \begin{minipage}{0.81\textwidth}
        i) $a \prec y$ and $a \neq x$. Then $a \nsim x$ otherwise the relations $x \prec y$ and $y \succ a$ cannot hold simultaneously. Since $a$ is a  lower cover of $y$, Lemma \ref{th:covers_preserved} implies that either $f(a) \prec f(y)$ or \mbox{$f(a) \succ f(y)$}. Suppose, for contradiction, that $f(a) \succ f(y)$. Then, $f(x) \prec f(y) \prec f(a)$. Given that \[\E_1(f(x),f(a)) = \E_1(x,a) =  1\] but $f(a)$ and $f(x)$ are not adjacent, there exist $z \in Q$ such that $f(x) \prec z$ and $f(a) \prec z$. However $z$ cannot be an upper cover of $f(x)$ because $f(x) < f(a) < z$. Hence, $f(a) \prec f(y)$.
    \end{minipage}
    \hfill
    \begin{minipage}{0.15\textwidth}
        \begin{tikzpicture}[
            node/.style={
            circle, draw=black,
            inner sep=0pt, minimum size=3pt,
            fill=white},
            ]
            \node (x) [node, label = right: $x$ ]  at (0,0){};
            \node (y) [node, label = right: $y$ ]  at (1,1){};
            \node (a) [node, label = right: $a$]  at (2,0){};
            \draw (x) -- (y) -- (a);
        \end{tikzpicture}
    \end{minipage}
    
     \begin{minipage}{0.81\textwidth}
        ii) $a \succ y$. Since $a$ is an upper cover of $y$, by Lemma \ref{th:covers_preserved}  we obtain that $f(a) \succ f(y)$ or $f(a) \prec f(y)$. Suppose, for contradiction, that $f(a) \prec f(y)$. Hence, $\E_1(f(x),f(a)) = 1$ because $f(x) \prec f(y) \succ f(a)$ is a minimal path for this metric. Given that $f^{-1}$ also preserves $\E_1$, it implies that $\E_1(x,a) = 1$. By the same argument from before, since  $x$ and  $a$ are not adjacent, such distance between them yields a contradictory element $z \in P$ that at the same time, is and is not an upper cover of $x$. Hence, $f(a) \succ f(y)$.
     \end{minipage}
      \hfill
     \begin{minipage}{0.15\textwidth}
          \begin{tikzpicture}[
            node/.style={
            circle, draw=black,
            inner sep=0pt, minimum size=3pt,
            fill=white},
            ]
            \node (x) [node, label = right: $x$ ]  at (0,0){};
            \node (y) [node, label = right: $y$ ]  at (0.8,0.8){};
            \node (a) [node, label = right: $a$]  at (1.6,1.6){};
            \draw (x) -- (y) -- (a);
        \end{tikzpicture}
     \end{minipage}
     
     \begin{minipage}{0.81\textwidth}
        iii) $b \succ x$ and $b \neq y$. Then $y \nsim b$, otherwise the relations $x \prec y$ and  $b \succ x$ can not hold  simultaneously. This case is analogous to i). Here, $\B_1(b,y) = 1$, then, it follows that  $\B_1(f(b),f(y)) = 1$. The assumption $f(b) \prec f(x)$ yields the contradiction that there is $z \in Q$ such that $ z \prec f(b) $ and $z \prec f(y)$ but $z < f(b) < f(y) $. Therefore, $f(b) \succ f(x)$.
     \end{minipage}
      \hfill
     \begin{minipage}{0.15\textwidth}
         \begin{tikzpicture}[
            node/.style={
            circle, draw=black,
            inner sep=0pt, minimum size=3pt,
            fill=white},
            ]
            \node (x) [node, label = right: $x$ ]  at (1,0){};
            \node (y) [node, label = right: $y$ ]  at (2,1){};
            \node (a) [node, label = right: $b$]  at (0,1){};
            \draw (y) -- (x) -- (a);
        \end{tikzpicture}
     \end{minipage}
     
     \begin{minipage}{0.81\textwidth}
        iv) $b \prec x$. This case is dual to ii). If we assume $f(b) \succ f(x)$, then $\B_1(f(b),f(y)) = 1$ and it implies $\B_1(b,y) = 1$. As in case ii), this distance produces a contradiction, now using lower rather than upper covers. Therefore, $f(b) \prec f(x)$.
     \end{minipage}
      \hfill
     \begin{minipage}{0.15\textwidth}
         \begin{tikzpicture}[
            node/.style={
            circle, draw=black,
            inner sep=0pt, minimum size=3pt,
            fill=white},
            ]
            \node (x) [node, label = right: $x$ ]  at (0,0){};
            \node (y) [node, label = right: $y$ ]  at (0.8,0.8){};
            \node (a) [node, label = right: $b$]  at (-0.8,-0.8){};
            \draw (y) -- (x) -- (a);
        \end{tikzpicture}
     \end{minipage}
     
    We now prove that for any $a \in P$, $f$ preserves every cover relation involving $a$. It suffices to show that for every $a \in P$ there exists an adjacent $b \in P$ such that the covering relation of $a$ and $b$ is preserved. Once such $b$ is found, the previous argument shows that indeed $f$ preserves all covering relations involving $a$, and the desired result follows.
    
    Let us proceed by induction. First, consider $A_1 = \{ a \in P \mid \E_1(a,y) \leq 1 \text{ or } \B_1(a,y) \leq 1\}$. We show that for each $a \in A_1$, there exists an adjacent element $b \in P$ such that the covering relation of $a$ and $b$ is preserved. Hence, for every $a \in A_1$, $f$ preserves the covering relation involving $a$. If $ a = y$, the claim is immediate. If $a \prec y$, then case i) implies $f(a) \prec f(y)$. Similarly, if $a \succ y$, case ii) gives $f(a) \succ f(y)$. Suppose now that $a$ and $y$ are not adjacent. Given that $\E_1(a,y) = 1$, then there is $z \in P$ with  $a \prec z$ and $y \prec z$. By case ii) we obtain $f(y) \prec f(z)$. Since $y \prec z$ and $f(y) \prec f(z)$, the previous argument ensures that all covering relations involving $z$ are preserved, in particular $a \prec z$. Hence, we have $a \prec z$ and $f(a) \prec f(z)$ as desired. Likewise, if $\B_1(a,y) = 1$, we can find $z \in P$ such that $a \succ z$ and $f(a) \succ f(z)$. 
    
    Let $A_n = \{ a \in P \mid \E_1(y,a) \leq n \text{ or } \B_1(y,a) \leq n\} $. Suppose that for every $a \in A_n$, all covering relations involving $a$ are preserved. Let $a \in A_{n+1}$ with $\E_1(y,a) = n+1$ and consider a minimal path $\gamma = \{ \tuple[0]{x}{m}\}$ from $x_0 = y$ to $x_m = a$ in $P$  such that $L_1(\gamma) = n + 1$. Take $\gamma' = \{ \tuple[0]{x}{m-2}\}$. Hence, $\gamma'$ decreases the length of $\gamma$ by 2 and has at most 2 less alternations than $\gamma$. If $A(\gamma')$ and $ A(\gamma)$  have the same parity,
    
    \[L_1(\gamma') = \ell(\gamma') - \frac{A(\gamma') + \cdot}{2} \leq   \ell(\gamma) - 2 + \frac{A(\gamma) - 2 + \cdot}{2} = \ell(\gamma) + \frac{A(\gamma) + \cdot}{2} - 1 = L_1(\gamma) - 1 = n.\]

    If $A(\gamma')$ and $ A(\gamma)$ have different parity, then $A(\gamma') = A(\gamma) - 1$. In the case that $A(\gamma)$ is even we obtain that 

    \[L_1(\gamma') \leq \ell(\gamma') - \frac{A(\gamma') -1}{2} =   \ell(\gamma) - 2 + \frac{A(\gamma) - 2}{2} = \ell(\gamma) + \frac{A(\gamma) }{2} - 1 = L_1(\gamma) - 1 = n.\]

    On the contrary, if  $A(\gamma)$ is odd,

     \[L_1(\gamma') = \ell(\gamma') - \frac{A(\gamma')}{2} =   \ell(\gamma) - 2 - \frac{A(\gamma) - 1}{2} = \ell(\gamma) + \frac{A(\gamma) + 1}{2} - 1 \leq L_1(\gamma) - 1 = n.\]

    Hence, in all cases $x_{m-2} \in A_n$ as $\E_1(y,x_{m-2}) \leq L_1(\gamma') \leq n$. This implies that $f$ preserves all covering relations involving $x_{m-2}$. In particular, $f$ preserves the covering relation with $x_{m-1}$.  Since the covering relation of $x_{m-2}$ and $x_{m-1}$ is preserved, the previous argument by cases ensures that all covering relations involving $x_{m-1}$ are preserved, in particular $x_{m-1}$ and $a$. Thus, $x_{m-1}$ is the desired element. By duality, if $\B_1(y,a) = n+1$ an analogous argument shows that $\B_1(y,x_{m-2}) \leq n$ ensuring that $x_{m-2} \in A_n$ and leading to the same conclusion.

    By induction, it follows that for all $n \in \N$ and every $ a \in A_n$, all covering relations involving $a$ are preserved. Since $P$ is path-connected, there exists $n$ such that $P = A_n$.  Therefore, $f$ preserves all covering relations in $P$ completing the argument.

    Symmetrically, the same argument proves that $f^{-1}$ preserves all covering relations on $Q$.
    
    c) $\implies$ a). Assume that $f$ and $f^{-1}$ preserve all covering relations. Let us show that $f$ is an isomorphism. We start by proving that $f$ is an order-preserving map.  Take $x < y $. If there is no $z \in P$ such that $x < z < y$, then $x \prec y$ and since $f$ preserves all covering relations, we obtain that $f(x) \prec f(y)$. Otherwise,  given that $P$ is discrete, we can find a chain $C = \{\tuple[0]{c}{n}\}$ such that $c_0 = x$, $c_n = y$ and $c_{i-1} \prec c_i$ for $i \in \{ 1, ..., n\}$. Because $f$ preserves covering relations we obtain that $f(c_{i-1}) \prec f(c_i)$. Hence, by transitivity of the order $f(x) < f(y)$. Symmetrically, if we assume that $f^{-1}$ preserves covering relations, the same argument proves that $f^{-1}$ is an order preserving map. 
\end{proof}

\begin{proof}[Proof of Theorem \ref{th:f_anti_preserves_EB}]
    Let $P$ and $Q$ be discrete path-connected posets, neither of which contains a subposet isomorphic to an infinite diamond–pentagon ladder.
    
    a) $\implies$ b). Assume there exists a bijection $f: P \to Q$ that preserves $\E_1$ and $\B_1$ and reverse at least one covering relation. We will prove that $f$ is an anti-isomorphism. By hypothesis, take $x \prec y$ in $P$ such that $f(y) \prec f(x)$. An analogous argument of b) $\implies$ c) from Theorem \ref{th:f_iso_preserves_EB} proves that $f$ reverses all covering relation involving $x$ and $y$. Let us prove it for the first case we considered  before.
    
   Take $a \in P$  such that $a \prec y$ and $a \neq x$. Since $a$ is a lower cover of $y$,  by Lemma \ref{th:covers_preserved} we know that $f(a)$ and $f(y)$ are adjacent. Suppose for contradiction that $f(a) \prec f(y)$. Then $f(a) \prec f(y) \prec f(x)$. Since 
    \[\B_1(f(x),f(a)) = \B_1(x,a) =1 \]
    but $f(a)$ and $f(x)$ are not adjacent, there exist $w \in Q$ such that $w \prec f(a)$ and $w \prec f(x)$. However, $w$ cannot be a lower cover of $f(x)$ as $w \prec f(a) < f(x)$. Thus, as we found this contradiction we conclude that $f(a) \succ f(y)$, i.e., $f$ reverses the covering relation $a  \prec y$.

    In a similar way, we can prove by induction on $A_n = \{ a \in P \mid \E_1(y,a) \leq n \text{ or } \B_1(y,a) \leq n\}$, that $f$ reverses all covering relations for this set. Since there exists $n$ such that $P = A_n$,  therefore, $f$ reverses all covering relations in $P$ completing the argument.

    Now take $x < y $. If there is no $z \in P$ such that $x < z < y$, then $x \prec y$ and since $f$ reverses all covering relations, we obtain that $f(x) \succ f(y)$. Otherwise,  given that $P$ is discrete, we can find a chain $C = \{\tuple[0]{c}{n}\}$ such that $c_0 = x$, $c_n = y$ and $c_{i-1} \prec c_i$ for $i \in \{ 1, ..., n\}$. Because  $f$ reverses covering relations we obtain that $f(c_{i-1}) \succ f(c_i)$. Hence, by transitivity of the order $f(x) > f(y)$. In this way, $f$ reverses all order relations of $P$.

    Symmetrically, the same argument proves that $f^{-1}$ reverses all covering relations on $Q$, and then, reverses all order relations. In this way $f$ is an anti-isomorphism that preserves $\E_1$ and $\B_1$.

    b) $\implies$ c). Assume $f: P \to Q $ is an anti-isomorphism that preserves $\E_1$ and $\B_1$.  We will prove that $P$ is a modular poset and $ P^{op} \cong Q$. Define $f^*:  P^{\op} \to Q$ by $f^*(x) = f(x)$. Is easy to check that $f^*$ is an isomorphism. By definition, this implies that $ P^{\op} \cong Q$.
    
    Let us prove that $P$ is lower semi-modular. Take $x$ and $y$ in $P$ such that $z$ is a common upper cover of $x$ and $y$. In particular, this implies that $x$ and $y$ are not comparable. Given that $f$ is an anti-isomorphism we obtain that $\{f(x),f(z),f(y)\}$ is a downward path with one alternation from $x$ to $y$. Then,  $\B_1(f(x),f(y)) = 1$. Since $f^{-1}$ preserves $\B_1$ (Lemma \ref{lemma_f-1_preserves}) this implies that $\B_1(x,y) = 1$. So, as $x$ and $y$ are not comparable, there is $w \in P$ that is a common lower cover of $x$ and $y$.

    Similarly, we can prove that $P$ is upper semi-modular by considering a common lower cover of $x$ and $y$ and showing that a common upper cover exists to preserve $\E_1$. Therefore, as $P$ is both upper an lower semi-modular, then, $P$ is modular.

    c) $\implies$ a). Assume that $P$ is a modular poset and $ P^{op} \cong Q$. We will prove that there exists a bijection $f: P \to Q$ that preserves $\E_1$ and $\B_1$ and reverse at least one covering relation. Indeed, given that $P^{op} \cong Q$, there is $f^*: P^{op} \to Q$ isomorphism, thus, $f: P \to Q$ defined by $f(x) = f^*(x)$ is an anti-isomorphism. In particular, $f$ is a bijection that reverses at least one covering relation. 
    
    Let us show that $f$ preserves $\E_1$, this is enough to show that $f$ preserves $\B_1$ as well since $P$ is a modular poset (Theorem \ref{th:E1_B1_equal_modular}). On one hand, by Theorem \ref{cor:modular_posets_distance} we know that for any $x,y \in P$, $\E_1^P(x,y) = \E_1^{P^{op}}(x,y)$. On the other hand, since $f^*$  is an isomorphism, then it preserves $\E_1$ as shown in a) $\implies$ b) from Theorem \ref{th:f_iso_preserves_EB}. Hence $\E_1^{P^{op}}(x,y) = \E_1^{Q}(f^*(x),f^*(y))$. Both equations imply that
    \begin{equation*}
        \E_1^P(x,y) = \E_1^{P^{op}}(x,y) = \E_1^{Q}(f^*(x),f^*(y)) = \E_1^{Q}(f(x),f(y)).
    \end{equation*}
    Proving that $f$ preserves $\E_1$.

\end{proof}

\section{Concluding Remarks and Future directions}
\label{Sec:concluding}

\subsection{Summary}

In this paper we have produced two new families of metrics for path connected posets called the $k$-climber distance $\E_k$ and the $k$-diver distance $\B_k$, one dual to the other.  In this context $k$ is a natural number. Given a path-connected poset $P$, such distances are defined on the power set of the poset and then it induces a metric on the poset itself. We show that the k-climber distance of $x,y \in P$ is equivalent to the minimal $k$-climber length over all paths from $x$ to $y$ and the $k$-climber length is a value that depend solely on the number of alternations and the length of the path (see Section \ref{Sec:k-distances}). 

In the case that $k=1$, the metrics $\E_1$ and $\B_1$ determine, up to isomorphism, discrete path-connected posets that do not contain a subposet isomorphic to an infinite diamond–pentagon ladder, with a single exception: modular posets. More precisely, if $P$ and $Q$ are such posets, $f:P \to Q$ is a bijection that preserves $\E_1$ and $\B_1$ if and only if $P \cong Q$. In the case that $P$ is a modular poset, the posets $P$ and $Q$  may instead be duals of one another (see Section \ref{Sec:characterization}).

In Section \ref{Sec:k-distances} we prove that when $k$ goes to infinity $\E_k$ converges to what we call the $\infty$-climber distance  $\E_\infty$. Similarly, we have a $\infty$-diver distance, $\B_\infty$, defined as the limit of $\B_k$. The $\infty$-distances between $x$ and $y$ are realized by a minimal length over all paths from $x$ to $y$ that depends only on the number of alternations of the path. This observation lead us to define the fence-climber and the fence-diver distances for fence-connected posets. The distances coincide with $\E_\infty$ and $\B_\infty$ respectively when calculated over a discrete poset. 

In Section \ref{Sec:notable_posets} we present explicit formulas for computing the climber and diver distances in some notable classes of posets, including linear orders and lattices. Likewise, we describe how these distances behave under standard constructions, such as direct product or lexicographic sums of posets. Also, we compare the shortest path, the shortest fence metrics and the Chebyshev distance with the climber and diver metrics. Additionally, we prove that our metrics are not induced by real valued functions defined on the poset. We prove instead, that the climber and diver metrics are instances of interleaving distance, a metric defined over categories equipped with a flow (Section \ref{Sec:concluding}).

Finally, in SageMath \cite{sagemath}, we implemented a collection of functions that, for any pair of elements in a poset, compute their $k$-climber and $k$-diver distance for arbitrary $k$ as well as determine the corresponding shortest paths in this setting (see Appendix \ref{Appendix}).


\subsection{A note in interleavings on categories}

Every metric presented here is an instance of an interleaving distance, one of the most used metrics to quantify the dissimilarity between a broad class of geometric–topological constructions commonly used in topological data analysis \cite{Flow}.

\subsubsection*{Definitions and notations}

We give basic definitions. For more categorical terminology; see \cite{riehl2017category}.  A \emph{category} $C$ consists of a collection of objects, e.g.~$X,Y,Z \in C$, and for any pair of objects $X,Y$ a collection (possibly empty) of morphisms denoted $C(X,Y)$. We can realize a poset $P$ as a category, where the elements $a,b$ of the poset are the objects and $a \leq b$ gives raise to a unique morphism in $P(a, b)$. 

An order-preserving self map of $P$, $f: P \to P$, is called an endomorphism of $P$. For example, the identity map $id_P: P \to P$ is an endomorphism. We denote the set of all endomorphisms of $P$ by $\End(P)$. The set $\End(P)$ is a poset as well. Take $f,g$ endomorphisms of $P$. We define that $f \leq g$ if $f(x) \leq g(x)$ for all $x \in P$. We can check that such order is indeed reflexive, antisymmetric, and transitive.
 
Given a poset $P$, we can equip it with a flow $\T$. A \textit{flow} is an order preserving map $\T : ([0,\infty), \leq ) \to \End(P)$ given by $\epsilon \mapsto \T_{\epsilon}$, such that $id_P \leq \T_0$ and  $\T_{\delta} \circ \T_{\epsilon} \leq \T_{\delta + \epsilon}$ for all $\delta, \epsilon \geq 0$. 

\begin{definition}
 Let $x,y$ be in $P$. The interleaving distance with respect to $\T$ for $x,y$ is defined as 
 \begin{equation*}
     d_{P,\T}(x,y) = \inf \{ \epsilon  \mid x \leq \T_\epsilon (y) \text{ and } y \leq \T_\epsilon(x) \}
 \end{equation*}
 And $ d_{P,\T}(x,y) = \infty$ if such infimum does not exist.

This interleaving distance, $d_{P,\T}$, is an extended pseudometric \footnote{The distance is pseudometric because there could be $x \neq y$ in $P$ with $d_{P,\T}(x,y) = 0$.} on $P$ \cite{Flow}.
\end{definition}

\subsubsection*{Realizing our metrics as interleaving distances}

Let $P$ be a poset. Remember that in Section \ref{Sec:k-distances} we defined the 1-climber distance on $\P(P)$, the power set of $P$. We will show there is a flow $\T: ([0,\infty), \leq ) \to \End(\P(P))$ such that for all $R,S \in \P(P)$, 
\begin{equation*}
    d_{\P(P),\T}(R,S) = \E_1(R,S)
\end{equation*}

Recall that for $R,S \in \P(P)$, the 1-climber distance between $R$ and $S$ is given by

 \begin{equation*}
    \E_1(R,S) = \min \{ n \mid R \subseteq (\updown[1])^{n} (S) \text{ and } S \subseteq (\updown[1])^{n} (R) \}
\end{equation*}
and $\E_1(R,S)= \infty $ if such $n$ does not exist.

Consider $\T: ([0,\infty), \leq ) \to \End(\P(P))$ defined by $\T(\epsilon) = (\updown[1])^{\lfloor \epsilon \rfloor}$. The map is well defined as $(\updown[1])^{\lfloor \epsilon \rfloor}$ is an order preserving self map of $\P(P)$. Also $\T$ is a flow given that it satisfies the next conditions: (1) $\T$ is an order preserving map as for $\delta \leq \epsilon$ we have that $(\updown[1])^{\lfloor \delta \rfloor} \leq (\updown[1])^{\lfloor \epsilon \rfloor} $, (2) $id_P  =   (\updown[1])^{\lfloor 0 \rfloor} = \T_0$ and (3) for all $ \delta, \epsilon \geq 0$, $\T_{\epsilon} \circ \T_{\delta} \leq \T_{\delta + \epsilon}$  given that 
\begin{equation*}
    \T_{\epsilon} \circ \T_{\delta} = (\updown[1])^{\lfloor \delta \rfloor} \circ (\updown[1])^{\lfloor \epsilon \rfloor}  = (\updown[1])^{\lfloor \delta \rfloor + \lfloor \epsilon \rfloor} \leq (\updown[1])^{\lfloor \delta + \epsilon \rfloor} 
\end{equation*}

In this way,  considering that the order on $\P(P)$ is given by inclusion we can rewrite the definition of $\E_1(R,S)$ by
 \begin{align*}
    \E_1(R,S) &= \min \{ n \mid R\leq (\updown[1])^{n} (S) \text{ and } S \leq (\updown[1])^{n} (R) \} \\
    &= \min \{ \lfloor \epsilon \rfloor \mid R\leq (\updown[1])^{\lfloor \epsilon \rfloor} (S) \text{ and } S\leq (\updown[1])^{\lfloor \epsilon \rfloor} (R) \} \\ &= \min \{ \epsilon \mid R \leq \T_\epsilon (S) \text{ and } S \leq \T_\epsilon(R) \} \\
    &= d_{\P(P),\T}(R,S)
\end{align*}

Equivalently, we have the following metric equalities for the $k$-climber, $k$-diver, fence-climber and fence-diver distances with respect to $\T: ([0,\infty), \leq ) \to \End(\P(P))$.
\begin{equation*}
    \begin{array}{lll}
          \E_k(R,S) =  d_{\P(P),\T}(R,S) & \text{ if } & \T(\epsilon) = (\updown[k])^{\lfloor \epsilon \rfloor} \\
          \B_k(R,S) =  d_{\P(P),\T}(R,S) & \text{ if } & \T(\epsilon) = (\downup[k])^{\lfloor \epsilon \rfloor} \\
          \E(R,S) =  d_{\P(P),\T}(R,S) & \text{ if } & \T(\epsilon) = (\updown)^{\lfloor \epsilon \rfloor} \\
          \B(R,S) =  d_{\P(P),\T}(R,S) & \text{ if } & \T(\epsilon) = (\downup)^{\lfloor \epsilon \rfloor} 
    \end{array}
\end{equation*}

\subsection{Future directions}

The climber and diver metrics address the theoretical and practical need for metrics that can be applied to general posets without structural constraints. Consequently, studying the computation and interpretation of these metrics on posets arising from real datasets constitutes a natural continuation of this project. An important question is what kinds of structural or hierarchical information these metrics are capable of detecting and highlighting within a dataset. 

From a computational perspective, an important open problem is the analysis of the algorithmic complexity of computing $k$-climber and $k$-diver distances. The SageMath implementation (Appendix \ref{Appendix}) is intended primarily as a proof of concept and may require optimization for large-scale computations. Future research  could focus on complexity estimates, optimized algorithms, and the identification of families of posets where these computations can be performed efficiently. The development of better implementations and computational experiments would also strengthen the practical applicability of these metrics.

In a similar vein, it is known that if a finite set with a partial ordering has known shortest-path and shortest-fence values, then both the partial order and its dual order can be recovered through a computable algorithm \cite{belding}. Since we have shown that finite posets are characterized by their 1-climber and 1-diver metrics, it is natural to conjecture that a similar reconstruction algorithm may be developed using these distance values. Such procedure could provide a new method for recovering  a poset directly from its metric information.

Another natural direction is the study of the metrics in broader classes of posets. While this work establishes explicit formulas and properties for some families of posets like linear and modular posets, many important classes remain unexplored, including distributive lattices, interval posets and graded posets. Such studies may reveal additional simplifications, structural properties, and new relationships satisfied by these metrics.

Finally, it would be interesting to determine whether richer classes of posets admit more diverse notions of categorical flow. Such flows could give rise to new families of interleaving-type distances and, consequently, to new metrics associated with ordered sets.

\paragraph{Acknowledgements}
AO is grateful to Elizabeth Munch and Avery St. Dizier for the helpful ideas, discussions and reviews for this project.
This work was supported in part by the National Science Foundation through grants CCF-2142713.

\newpage
\bibliographystyle{mystyle2}
\bibliography{references}

\appendix
\section{Computation}
\label{Appendix}
The code of the following algorithms and support functions are accessible online at \url{https://github.com/aaolaveh/climber_diver_metrics}.

\subsection{Calculating climber and diver metrics}

We propose Algorithm \ref{alg:dis_computation} to compute the distance matrix of all of our metrics over a finite connected poset $P$. 

Take $x,y \in P$. Recall that the $k$-climber distance between $x$ and $y$ is defined as the minimum $n$ such that \mbox{$x \in (\updown[k])^n(\{y\})$} and $y \in (\updown[k])^n(\{x\})$. By symmetry, is enough to find the minimum $n$ satisfying the second condition. 

Hence, our approach to computing $\E_k(x,y)$ consists of iteratively calculating $(\updown[k])^n(\{x\})$, starting with $n=0$ and increasing $n$ by $1$ at each step. At each step, we let $\E_k(x,z) = n$ for all $z \in (\updown[k])^n(\{x\})$. Given that $P$ is connected, $y$ would appear eventually for some $(\updown[k])^n(\{x\})$.

If we replace $\updown[k]$ by $\T$ it the process above, then, we can generalize this algorithm to any $\T: \P(P) \to \P(P) $ and the distance defined to be the minimum $n$ such that $x \in (\T)^n(\{y\})$ and $y \in (\T)^n(\{x\})$. So, the metric is fully characterized by $\T$. In particular, this is true for the $k$-diver distance with $\T =  \downup[k]$, the fence-climber distance with $\T =  \updown$ and the fence-diver distance with $\T =  \downup$. Moreover, as $P$ is finite, then the $\infty$-climber and $\infty$-diver distances coincide with the fence-climber and the fence-diver distances respectively.

\begin{algorithm}[h] 
\caption[Algorithm to compute distance]{Algorithm to compute distance matrix of finite connected poset $P$}
\label{alg:dis_computation}

\begin{algorithmic}[1]
\State \textbf{Given:} $P = \{\tuple{x}{m}\}$ finite connected poset and $\T: \P(P) \to \P(P) $
\State $M \gets |P| \times |P|$ matrix with all entrees 0
\For{$i \in \{1, ..., m\} $}
\State $n \gets 0$, $A \gets \{x_i\}$, $Q \gets \{x_{i+1}, ..., x_m\}$
\While{$Q \neq \emptyset$}
\State $n \gets n+1$
\State $B = \T(A) - A$
\State $A = \T(A)$
\For{$j$ with $x_j \in B \cap Q$}
\State $M[i,j] \gets n$
\State $M[j,i] \gets n$
\EndFor
\State $Q \gets Q - B$
\EndWhile
\EndFor
\State \Return $M$ 
\end{algorithmic}
\end{algorithm}

\subsection{Finding the climber and diver shortest path}

Given $P$ a connected poset and $x,y \in P$ we propose a modification of Dijkstra's algorithm \cite{Dijkstra1959} (Algorithm \ref{alg:path_computation}) to compute a shortest path between $x$ and $y$ in the context of our metric induced by a fixed $\T$ considered before. In this algorithm, $\texttt{LEN}_{\T}(\gamma)$ is a function that returns the length of the path. Also, we maintain a \texttt{parent} array during the algorithm to recover the shortest path from any $z \in P$ to $x$.

\begin{algorithm}[ht] 
\caption[Algorithm to compute shortest path]{Algorithm to compute a shortest path between $x$ and $y$ in a finite connected poset $P$}
\label{alg:path_computation}

\begin{algorithmic}[1]
\State \textbf{Given:} $P$ finite connected poset, $x,y \in P$ and $\T: \P(P) \to \P(P) $
\For{$z \in P$}
\State \texttt{dist}[$z$] $\gets$ Infinity
\State \texttt{parent}[$z$] $\gets$ None
\EndFor
\State \texttt{Visited} $\gets \emptyset$
\State \texttt{dist}[$x$] $\gets 0$
\While{$y \notin$ Visited}
\State $w \gets$ element in $P$ with min \texttt{dist}[$w$]
\State add $w$ to \texttt{Visited}
\State $N_w \gets$ the set of adjacent elements to $w$ that are not in \texttt{Visited}
\State $\gamma \gets$ the shortest path from $x$ to $w$ (using \texttt{parent}) 
\For{$z \in N_w$}
\State $\gamma_z \gets$ append $z$ to $\gamma$
\State \texttt{alt }$\gets \texttt{LEN}_\T(\gamma_z)$ 
\If{\texttt{alt }$ < \texttt{dist}[z]$}
\State $\texttt{dist}[z] \gets $ alt
\State $\texttt{parent}[z] \gets w$
\EndIf
\EndFor
\EndWhile
\State \Return(shortest path from $x$ to $y$ (using \texttt{parent}) )
\end{algorithmic}
\end{algorithm}

Note that, in the original Dijkstra's algorithm, for an adjacent element $z$ of $w$, the variable \texttt{alt} in Step 15 is assigned the value of \texttt{dist}[$w$] plus the edge weight between $w$ and $z$. In our case \texttt{dist}[$w$] corresponds to $ \texttt{LEN}_\T(\gamma)$ while \texttt{alt} corresponds to $ \texttt{LEN}_\T(\gamma_z)$ where $\gamma_z = \gamma \cup \{z\}$.

Since $\gamma_z$ extends $\gamma$ by one element, it increases the length of the path by 1 and it may introduce at most one additional alternation with respect to $\gamma$.
Consequently, by checking the possible cases, we obtain that either $\texttt{LEN}_\T(\gamma_z) =  \texttt{LEN}_\T(\gamma)$ or $\texttt{LEN}_\T(\gamma_z) =  \texttt{LEN}_\T(\gamma) + 1$. Therefore, our algorithm coincides with the original Dijkstra's algorithm after virtually assigning the edge between $w$ and $z$ a weight of either $0$ or $1$, respectively. Since Dijkstra's algorithm is guaranteed to recover shortest paths for non-negative edge weights, our algorithm is likewise guaranteed to recover the shortest path.

\end{document}